\documentclass[11pt]{amsart}
\usepackage{amssymb,amsxtra,amsmath,amsfonts}
\usepackage{eucal,mathrsfs}
\usepackage{verbatim,enumerate}
\usepackage{graphicx,color,graphics} 
\usepackage{xypic}
\usepackage{color}
\usepackage[linktocpage=true,colorlinks=true, linkcolor=blue, citecolor=red, urlcolor=green]{hyperref}

\newtheorem{theorem}{Theorem}
\newtheorem{lemma}[theorem]{Lemma}
\newtheorem{proposition}[theorem]{Proposition}
\newtheorem{corollary}[theorem]{Corollary}

\theoremstyle{definition}
\newtheorem{definition}[theorem]{Definition}
\newtheorem{example}[theorem]{Example}

\theoremstyle{remark}
\newtheorem{remark}[theorem]{Remark}

\numberwithin{equation}{section}
\numberwithin{theorem}{section}

\newcommand\thref{Theorem \ref}
\newcommand\leref{Lemma \ref}
\newcommand\prref{Proposition \ref}
\newcommand\coref{Corollary \ref}
\newcommand\deref{Definition \ref}

\newcommand\reref{Remark \ref}
\newcommand\seref{Section \ref}
\renewcommand{\comment}[1]{}
\def\CC{\mathbb{C}}

\def\N{\mathscr{S}}
\def\NN{\hat{\mathscr{S}}}
\def\Y{\mathcal{Y}}

\def\V{\mathcal{V}}

\def\A{\mathcal{A}}

\def\T{\mathcal{T}}
\def\U{\mathcal{U}}
\def\ZZ{\mathbb{Z}}
\def\h{\mathfrak{h}}

\DeclareMathOperator\End{End}
\DeclareMathOperator\Der{Der}
\DeclareMathOperator\Hom{Hom}

\DeclareMathOperator\ad{ad}

\DeclareMathOperator\rank{rank}
\DeclareMathOperator\Ker{Ker}
\DeclareMathOperator\Res{Res}

\DeclareMathOperator\Span{span}

\DeclareMathOperator\LF{LFie}

\def\LFT{\LF_T}

\def\vac{{\boldsymbol{1}}}  %{|0\rangle} % vacuum vector

\def\ii{\mathrm{i}} %{\sqrt{-1}}
\def\al{\alpha}
\def\be{\beta}

\def\de{\delta}

\def\ep{\varepsilon}

\def\om{\omega}
\def\ze{\zeta}
\def\z{z}
\def\d{\partial}

\def\lieh{{\mathfrak{h}}}

\begin{document}

\title[Logarithmic Vertex Algebras]{Logarithmic Vertex Algebras}
\author{Bojko Bakalov}
\author{Juan J. Villarreal}
\address{Department of Mathematics,
North Carolina State University,
Raleigh, NC 27695, USA}
\email{bojko\_bakalov@ncsu.edu}
\email{juanjos3villarreal@gmail.com}

%\thanks{The first author is supported in part by NSF and NSA grants} %DMS-0701011 

\date{July 21, 2021; Revised August 8, 2021}



\begin{abstract}
We introduce and study the notion of a logarithmic vertex algebra, which is a vertex algebra with logarithmic singularities in the operator product expansion of quantum fields;
thus providing a rigorous formulation of the algebraic properties of quantum fields in logarithmic conformal field theory.
We develop a framework that allows many results about vertex algebras to be extended to logarithmic vertex algebras, including in particular the Borcherds identity and Kac Existence Theorem.
Several examples are investigated in detail, and they exhibit some unexpected new features that are peculiar to the logarithmic case.
\end{abstract}

\maketitle
\tableofcontents

\section{Introduction}\label{s1}
Vertex algebras essentially appeared in physics as chiral algebras in conformal field theory \cite{BPZ,Go,DMS}, and in mathematics in the construction of the moonshine module of the Monster group \cite{Bo,FLM}. Since then, the theory of vertex algebras has been developed in different directions (see, e.g., \cite{FB, LL, K, KRR}).
In particular, a generalization called \emph{generalized vertex algebras} was later given in \cite{FFR, DL, M}  (called vertex operator para-algebras in \cite{FFR}). Important examples of generalized vertex algebras have been constructed, in particular, in \cite{FFR, DL, M, GL, BK2}.

In this work, we define and study a generalization of vertex algebras, which we call \emph{logarithmic vertex algebras}. We develop a framework in terms of a locally-nilpotent endomorphism $\N\in \End(V)\otimes \End(V)$ for a vector space $V$, which 
we call an (infinitesimal) \emph{braiding map} on $V$. This allows many results about vertex algebras to be extended to logarithmic vertex algebras. A logarithmic vertex algebra with a trivial braiding map $\N=0$ is an ordinary vertex algebra, and in this case our results reduce to known statements about vertex algebras. 

Now we describe two motivations for studying logarithmic vertex algebras. 
Our first motivation comes from the so-called \emph{twisted modules} of vertex algebras, which were introduced and developed in \cite{Le, FLM, FFR, D}. Let $\sigma$ be a finite-order automorphism of a vertex algebra $U$. 
Then for any $\sigma$-twisted $U$-module $W$, the direct sum $U\oplus W$ naturally has a structure of a generalized vertex algebra (see, e.g., \cite{BK2}).
Note that every finite-order automorphism is semisimple. The notion of a twisted module has been generalized for non-semisimple automorphisms \cite{H,B};  in this case, the quantum fields involve the logarithm of the formal variable, and such modules are called \emph{twisted logarithmic modules}. Let $\mathcal{N} \in\Der(U)$ be a locally-nilpotent derivation of a vertex algebra $U$; then $\varphi=e^{-2\pi \ii \mathcal{N}}$ is an automorphism of $U$. 
In \seref{ex4.3} below, we show that $V=U\oplus W$ has a natural structure of a logarithmic vertex algebra for any $\varphi$-twisted logarithmic $U$-module $W$. In this case, the braiding map $\N$ on $V$ is given by:
\begin{equation*}
\N\big|_{U \otimes U} = 0 \,, \quad \N\big|_{U \otimes W} = \mathcal{N} \otimes I \,, \quad  \N\big|_{W \otimes U} = I \otimes \mathcal{N} \,, \quad  \N\big|_{W \otimes W} = 0 \,,
\end{equation*}
where $I$ denotes the identity operator.
The notion of a logarithmic vertex algebra extends that of a twisted logarithmic module by allowing logarithmic singularities in the algebra itself and not just in the action of the algebra on its module.
Many of the ideas and results of the present paper generalize and extend those developed in \cite{B} for twisted logarithmic modules.

Our second and main motivation comes from \emph{logarithmic conformal field theory} (LCFT); see, e.g., \cite{G1, G2, Ga, CR} and the references therein.
We are motivated by LCFT to define \emph{conformal} logarithmic vertex algebras (Definition \ref{d3.10}), in which we have a Virasoro Lie algebra action on the vector space $V$.
In LCFT, the Virasoro operator $L_{0}$ is not semisimple as in usual conformal vertex algebras, but acts as a Jordan matrix on $V$. 
We denote by $L_{0}^{(s)}$ the semisimple part of $L_{0}$ and by $L^{(n)}_{0}=L_{0}-L^{(s)}_{0}$ its locally-nilpotent part, according to its Jordan--Chevalley decomposition.
Then we have the following relation between $L_{0}^{(n)}$ and the braiding map $\N$ in a conformal logarithmic vertex algebra $V$. Suppose that
\begin{equation*}
L_{0}^{(n)}=d+\sum_{i} f_i g_i \qquad\text{where}\quad d,f_i,g_i\in \Der(V) \,, 
\end{equation*}
an assumption that holds in many examples. Then 
\begin{equation*}
\N=-\sum_{i} (f_i\otimes g_i+g_i\otimes f_i) \,;
\end{equation*}
see Proposition \ref{pro3.11} below for a more precise statement (in the fermionic case, one has to include appropriate signs).
Notice that, as already observed in \cite{G1}, if $L_0$ is not semisimple so $L^{(n)}_{0} \ne0$, then we are forced to have $\N\ne0$ and hence logarithms.

In this paper, we present a rigorous formulation of the algebraic properties of quantum fields in LCFT, which
can be used to do explicit calculations with the operator product expansion (OPE). 
We derive a version of the Kac Existence Theorem, 
thus providing a way to construct a logarithmic vertex algebra from a collection of generating fields.
In particular, we interpret an example of a LCFT at central charge $c=0$, due to Gurarie--Ludwig \cite{GL1,GL2,G2}, as an explicit example of a conformal logarithmic vertex algebra (see \seref{e4.7}).
In a future work, we will study the representations of this logarithmic vertex algebra, which we hope will shed light on LCFT at $c=0$.

In more details, to every vector (state) $a$ in a logarithmic vertex algebra $V$, there corresponds a field $Y(a,z)$, which is a linear map from $V$ to the space of Laurent series in a formal variable $z$
that are also polynomials in another formal variable $\ze=\log z$,
\begin{equation*}
Y(a,z) \colon V \to V(\!(z)\!)[\ze] \,, \qquad a\in V \,.
\end{equation*}
The map $a\mapsto Y(a,z)$ is called the \emph{state-field correspondence} and has the property that
\begin{equation*}
Y(\vac,z)=I \,, \qquad Y(a,z)\vac \in V[\![z]\!] \,, \qquad Y(a,z)\vac \big|_{z=0} = a \,,
\end{equation*}
where $\vac$ denotes the vacuum vector.
The fields $Y(a,z)$ have a mode expansion that involves the braiding map $\N$:
\begin{equation*}
Y(a,z)b = \sum_{n\in \ZZ} z^{-n-1-\N} a_{(n+\N)}b \,, \qquad z^{-\N} := e^{-\ze\N} \,,
\end{equation*}
which generalizes the corresponding expression about twisted logarithmic modules from \cite{B}. The precise meaning of this formula is that we have bilinear products
\begin{equation*}
\mu_{(n)} \colon V\otimes V\rightarrow V\,, \qquad {\mu}_{(n)}(a\otimes b)= a_{(n+\N)}b \,, \qquad n\in\ZZ \,,
\end{equation*}
and then
\begin{align*}
Y(a,z)b &= \sum_{n\in \ZZ} \mu_{(n)} \bigl( z^{-n-1-\N} (a \otimes b) \bigr) \\
&=\sum_{ \substack{i,n\in\ZZ \\ i\ge0} } \frac{(-1)^{i}}{i!} \ze^i z^{-n-1} \mu_{(n)} \bigl( \N^i (a \otimes b) \bigr) \,.
\end{align*}
Here the sum over $i$ is finite, because $\N$ is locally nilpotent, so $\N^i (a \otimes b)=0$ for $i\gg0$.
For every fixed $a$ and $b$, we also have $a_{(n+\N)}b = 0$ for $n\gg0$.
All these products, together with the action of $\N$, encode the OPE of logarithmic fields as follows:
\begin{equation*}
Y(a,z_1)Y(b,z_2)\sim \sum_{ \substack{i,n\in\ZZ \\ i\ge0} } \frac{(-1)^{i}\log^{i}(z_1-z_2)}{i!(z_{1}-z_{2})^{n+1}} \, 
Y\bigl(\mu_{(n)} (\N^i (a \otimes b)),z_2\bigr)\,.
\end{equation*}
%In particular, this implies that all OPEs have a nested structure given by the action of $\N$.
In conformal field theory, it is customary to include only the \emph{singular part} of the OPE, which corresponds to the sum over $n\ge0$
(and $i=0$ as there are no logarithms). In LCFT, the singular part of the OPE involves the sets of indices
$\{i=0,\,n\ge 0\}$ and $\{i\ge1,\,n\in\ZZ\}$, because all logarithmic terms are singular.
Fortunately, the terms with $\log(z_1-z_2)$ are controlled by the action of $\N$.
Once we know the terms with $i=0$ that are determined by the products $a_{(n+\N)}b$, then the terms with $\log^i(z_1-z_2)$
are obtained from taking $(n+\N)$-th products of the vectors appearing in the expression $\N^i (a \otimes b)$.
The fact that such a nested structure exists was one of our main initial observations.

%We derive the explicit algebraic relations satisfied by the $(n+\N)$-th products and the braiding map $\N$, which are useful for computations with logarithmic OPEs.

Our next goal was to understand the algebraic relations satisfied by all $(n+\N)$-th products and the braiding map $\N$.
The answer is given by the following logarithmic analogue of the \emph{Borcherds identity} on $V^{\otimes 3}$
for all $k,m,n\in\ZZ$:
\begin{align*}
\sum_{j=0}^\infty &(-1)^{j}\mu_{(m+n-j)}(I\otimes \mu_{(k+j)})\binom{n+\N_{12}}{j}\\
& -\sum_{j=0}^\infty (-1)^{n+j}\mu_{(n+k-j)}(I\otimes \mu_{(m+j)})\binom{n+\N_{12}}{j} (P \otimes I) \\
& =\sum_{j=0}^\infty \mu_{(m+k-j)}( \mu_{(n+j)}\otimes I) \binom{m+\N_{13}}{j}\,,
\end{align*}
where $P\colon V\otimes V\to V\otimes V$ is the transposition of factors (with the appropriate sign in the fermionic case).
Above, we use the standard notation
\begin{equation*}
\N_{12}=\N\otimes I \, , \qquad \N_{23}= I \otimes \N \, , \qquad \N_{13}=(I\otimes P) \N_{12} (I\otimes P) 
\end{equation*}
for the action of $\N$ on the corresponding factors in the tensor power $V^{\otimes 3}$.

Notice that, unlike the case of ordinary vertex algebras, the Borcherds identity for $n=0$ does not give a formula for the commutator of modes.
Instead, in general, one gets a sequence of quadratic relations (see \seref{e4.7}).
Other unexpected new features arise that are peculiar to the logarithmic case. For example, it is possible to have $a,b\in V$ such that the corresponding fields have
a trivial singular OPE $Y(a,z_1)Y(b,z_2) \sim 0$, but nevertheless do not commute, $[Y(a,z_1),Y(b,z_2)] \ne 0$. In other words, even if $\N(a\otimes b)=0$ and
$a_{(j+\N)}b=0$ for all $j\ge0$, one may have some $[a_{(m+\N)},b_{(k+\N)}]\ne0$, due to the presence of $\N_{13}$ in the right-hand side
of the Borcherds identity (see \reref{remnew}). More generally, it is possible to have a subspace of $V$ that is closed under the OPEs, 
which however does not close a subalgebra of $V$. This suggests that one has to work with subspaces of $V$ that are invariant
under all tensor components of the braiding map $\N$.

Let us list the defining properties of the infinitesimal \emph{braiding map} $\N$ in a logarithmic vertex algebra $V$. As already said, $\N$ is locally nilpotent. 
In addition, $\N$ is symmetric (i.e., $\N P=P\N$) and satisfies:
\begin{equation*}
[\N_{12},\N_{23}]=[\N_{13},\N_{23}]=[\N_{12},\N_{13}]=0 \,,
\end{equation*}
\begin{equation*}
\N ({\mu}_{(n)}\otimes I) = ({\mu}_{(n)}\otimes I) (\N_{13}+\N_{23}) \,, \qquad n\in\ZZ\,.
\end{equation*}
The first identity above trivially implies that $\N$ is a solution of the classical Yang--Baxter equation. 
The second identity, which we call the hexagon axiom, together with the symmetry of $\N$ is equivalent to the condition that
\begin{equation*}
\N \in \Der(V) \otimes \Der(V) \,.
\end{equation*}
Our hexagon axiom is similar to the hexagon identity for the braiding map in \emph{quantum vertex algebras} (see \cite{FR, EK, DGK}). 
In fact, much of the theory developed here has parallels for quantum vertex algebras.
%and in a future publication we plan to apply our methods to that case.

In the theory of vertex algebras, the main property of the fields is their \emph{locality} (see \cite{DL, Li, K}):
\begin{equation*}
z_{12}^{N}\,Y(a,z_1)Y(b,z_2) = z_{12}^{N}\, Y(b,z_2)Y(a,z_1)\, , \qquad \z_{12}:=\z_1-\z_2 \,,
\end{equation*}
for some integer $N\ge0$ depending on $a,b\in V$ (and again there is an extra sign factor in the case when $a$ and $b$ are fermions).
For generalized vertex algebras, one allows $N$ to be a complex number, provided that $z_{12}^{N}$ is expanded in the domain $|z_1|>|z_2|$ on the left-hand side
and for $|z_2|>|z_1|$ on the right-hand side (see \cite{DL, M,BK2}). Our definition of locality for logarithmic vertex algebras generalizes this further by replacing $z_{12}^{N}$ with $z_{12}^{N+\N}$, which is interpreted as
\begin{equation*}
z_{12}^{N+\N} = z_{12}^{N} \, e^{\N\log |z_{12}|} 
\end{equation*}
and again expanded in the appropriate domain. More precisely, introduce the formal power series
\begin{align*}
\vartheta_{12}=\zeta_{1}-\sum_{n=1}^\infty \frac{1}{n}z_{1}^{-n} z_{2}^{n} \,, \qquad
\vartheta_{21}=\zeta_{2}-\sum_{n=1}^\infty \frac{1}{n}z_{1}^{n} z_{2}^{-n} \,,
\end{align*}
which can be thought of as expansions of $\log z_{12} = \log z_1+\log(1-\frac{z_2}{z_1})$ for $|z_1|>|z_2|$, and
of $\log z_{21}$ for $|z_2|>|z_1|$, respectively.
Then the \emph{locality axiom} in a logarithmic vertex algebra $V$ states that
for every $a,b\in V$, there exists an integer $N\ge0$ such that for all $c\in V$, 
\begin{align*}
z_{12}^{N} \, Y(z_1)& (I \otimes Y(z_2)) e^{\vartheta_{12}  \N_{12}} (a\otimes b\otimes c )\\
&=z_{12}^{N} \, Y(z_2) (I \otimes Y(z_1)) e^{\vartheta_{21}  \N_{12}} (b\otimes a\otimes c ),
\end{align*}
where $Y(z)(a\otimes b)=Y(a,z)b$. This identity is again reminiscent of what happens for quantum vertex algebras (see \cite{FR, EK, DGK}).
It is possible to generalize the notion of a logarithmic vertex algebra further by allowing non-integral powers of $z$ in the fields and the number $N$
in the locality condition to be non-integer, as for generalized vertex algebras. For simplicity of the exposition, we have restricted to the integer case in the present paper.
Another possible generalization could be to include polylogarithms in the OPEs, as in \cite{AH,V}.

The paper is organized as follows. In Section \ref{s2}, we define the notions of a logarithmic quantum field, braiding map, locality of fields and of spaces of fields, $(n+\NN)$-th products, OPEs, and translation covariance of logarithmic fields. We obtain a generalization of Dong's Lemma (Lemma \ref{l2.12b}), and we show how, under certain assumptions, one can construct a state-field correspondence for logarithmic fields (Theorem \ref{l2.19}). Here $\NN$ is a braiding map on the space of logarithmic fields (see Definition \ref{lm}).

In Section \ref{s3}, we provide two definitions of a logarithmic vertex algebra $V$ (Definitions \ref{d3.1} and \ref{d3.11}).
One is given in terms of a braiding map $\N$ on $V$, and the other in terms of a braiding map $\NN$ on the space of logarithmic fields on $V$
obtained as the image of the state-field correspondence $Y$. 
We prove that the two definitions are equivalent and that the two braiding maps are related by:
\begin{equation*}
\NN = (\ad\otimes\ad)(\N) \,,
\end{equation*}
\begin{equation*}
\N(a\otimes b) = \NN\bigl(Y(a,z_1) \otimes Y(b,z_2) \bigr) (\vac\otimes\vac) \big|_{z_1=z_2=0} \,.
\end{equation*}
This equivalence allows us to apply the results of Section \ref{s2} and obtain analogues of Kac's Existence Theorem (Theorem \ref{t3.14}) and Goddard's Uniqueness Theorem (Theorem \ref{t3.7}).
%which provide a way to construct a logarithmic vertex algebra from a collection of generating fields.
Then we derive logarithmic versions of the $n$-th product identity (Proposition \ref{pro3.12}), skew-symmetry and associativity (Propositions \ref{pro3.16}, \ref{pro3.17}), Borcherds identity (Theorem \ref{t3.20} and \eqref{borcherds2}) and Jacobi identity (Proposition \ref{pro3.21}). We also define conformal logarithmic vertex algebras and describe the relation between $L_0^{(n)}$ and $\N$ mentioned above (Proposition \ref{pro3.11}).
We finish the section with a discussion of the relation to twisted logarithmic modules.

Finally, in Section \ref{s4}, we present in detail several examples of conformal logarithmic vertex algebras. These include the well-known examples of symplectic fermions \cite{G1,Kau1,Kau2} and free bosons, 
Gurarie--Ludwig's LCFT at $c=0$ from \cite{GL1,GL2,G2}, as well as a new example of a logarithmic vertex algebra associated to an integral lattice. As an application of our Borcherds identity, we derive a complete list of explicit algebraic relations satisfied by the modes of generating fields in Gurarie--Ludwig's example (\leref{lem4.5}), which was previously unknown.

Throughout the paper, we denote by $\ZZ_+$ the set of non-negative integers, and we use the divided-powers notation $x^{(k)}=x^k/k!$ for $k\ge0$; $x^{(k)}=0$ for $k<0$. We will work with vector superspaces over $\CC$. 
For a vector superspace $V=V_{\bar0}\oplus V_{\bar1}$ and a vector $v\in V_{\alpha}$, where $\alpha\in \ZZ/2\ZZ=\{\bar0,\bar1\}$, we denote the parity of $v$ by $p(v)=\alpha$. Unless otherwise specified, all tensor products and $\Hom$'s are over $\CC$.
We denote the identity operator by $I$.

\section{Logarithmic quantum fields}\label{s2}

We start this section with a review of some basic facts and notations on superspaces. Then we define the notions of a logarithmic quantum field, locality, $(n+\NN)$-th products, OPEs, and translation covariance of logarithmic fields.
We prove a generalization of Dong's Lemma, and we show how, under certain assumptions, one can obtain a state-field correspondence for logarithmic fields.

\subsection{Vector superspaces}\label{s2.0}
%Recall that a vector superspace is a vector space $V$ with a decomposition $V=V_{\bar0}\oplus V_{\bar1}$.
%For $v\in V_{\alpha}$, where $\alpha\in \ZZ/2\ZZ=\{\bar0,\bar1\}$, we will write $p(v)=\alpha$ and call it the parity of $v$. 
%The elements of $V_{\bar0}$ are called even, while those of $V_{\bar1}$ are odd.
%
Recall that the tensor product of two superspaces $V=V_{\bar0}\oplus V_{\bar1}$ and $W=W_{\bar0}\oplus W_{\bar1}$ is the superspace $V\otimes W$ with
\begin{equation*}
(V\otimes W)_{\alpha}=\bigoplus_{\beta\in \ZZ/2\ZZ}V_{\beta}\otimes W_{\alpha-\beta}\,,
\qquad \alpha\in \ZZ/2\ZZ = \{\bar0,\bar1\} \,.
\end{equation*}
We denote the transposition operator on the superspace $V\otimes V$ as
\begin{equation}\label{logf1}
P(a\otimes b)=(-1)^{p(a)p(b)}b\otimes a\, , \qquad a, b\in V\, .
\end{equation}
 As customary when writing such identities, by linearity, we assume that the elements $a$ and $b$ are homogeneous with respect to parity.

The vector space $\End(V)$, consisting of all linear operators on $V$, is naturally a superspace with 
\begin{equation*}
\End(V)_{\alpha}=\{\varphi\in \End(V)  \,|\, \varphi V_{\beta}\subset V_{\alpha+\beta} \;\; \forall\beta\in \ZZ/2\ZZ\}\,, \qquad \alpha\in \ZZ/2\ZZ\,.
\end{equation*}
Moreover, on $\End(V)$ we have a product defined by the composition, which will be denoted as
\begin{equation}\label{logf2}
\mu\colon\End(V)\otimes \End(V)\rightarrow \End(V)\,,\qquad \mu(\varphi\otimes\psi)= \varphi\psi \,.
\end{equation}
This makes $\End(V)$ an associative superalgebra. For any associative superalgebra $A$, the commutator bracket is defined by 
\begin{equation*}
[a,b]=ab-(-1)^{p(a)p(b)}ba\,, \qquad a,b\in A\,,
\end{equation*}
and it defines on $A$ the structure of a Lie superalgebra. We will use the standard notation $\ad(a)b=[a,b]$.
Then for $a,b\in A$ we have the endomorphism  
\[(\ad\otimes\ad)(a\otimes b) = \ad (a)\otimes \ad (b)\in \End(A)\otimes \End(A)\, , \] 
which is defined by 
\begin{equation}\label{logf3}
\left(\ad (a)\otimes \ad (b)\right)(c\otimes d) =(-1)^{p(b)p(c)}[a,c]\otimes [b,d]\,, \qquad c,d\in A\, . 
\end{equation}

In this paper, we will work with even endomorphisms 
\[\N\in \End(V)\otimes \End(V)\, .\] 
We can write $\N$ as a finite sum
\begin{equation}\label{logf-n}
\N=\sum_{i=1}^{L}\phi_{i}\otimes  \psi_{i}\, , \qquad \phi_{i}, \psi_{i}\in \End(V) \,;
\end{equation}
then the condition $p(\N)=\bar0$ means that $p(\phi_i)=p(\psi_i)$ for all $i$.
The action of $\N$ on $V\otimes V$ is given by:
\begin{equation}\label{logf-nact}
\N(a\otimes b)=\sum_{i=1}^{L}(-1)^{p(\psi_{i})p(a)}\phi_{i}(a)\otimes  \psi_{i}(b)\, , \qquad a,b\in V\, .
\end{equation}
We say that $\N$ is \emph{locally nilpotent} on $V \otimes V$ if for all $a,b\in V$ there exists a positive integer $r$ such that  
$\N^r (a\otimes b) = 0$.
%\begin{equation}\label{logf-ln}
%\N^r (a\otimes b) = 0 \, .
%\end{equation}

Given $\N$ as above, we define the endomorphisms  $\N_{12}, \N_{23},\N_{13}\in \End(V)\otimes \End(V)\otimes \End(V)$ as follows:
\begin{equation}\label{logf-1}
\begin{split}
\N_{12}&=\N\otimes I = \sum_{i=1}^{L} \phi_{i} \otimes \psi_{i} \otimes I \, , \\
\N_{23}&=I\otimes\N = \sum_{i=1}^{L} I \otimes \phi_{i} \otimes \psi_{i} \, , \\
\N_{13}&=(I\otimes P) \N_{12} (I\otimes P) = \sum_{i=1}^{L} \phi_{i} \otimes I \otimes \psi_{i} \,,
\end{split}
\end{equation}
where $I$ denotes the identity operator and $P$ is the transposition \eqref{logf1} on $V\otimes V$.

\subsection{Logarithmic quantum fields}\label{s2.1}

Throughout the paper, we will denote by $\z,\z_1,\z_2,\dots$ and $\ze,\ze_1,\ze_2,\dots$ commuting even formal variables,
and we will use the notation $z_{ij}=z_i-z_j$.
The variables $\ze,\ze_1,\ze_2,\dots$ will be thought of as $\ze=\log\z$ and $\ze_i=\log\z_i$. As in \cite{B}, we introduce the operators
\begin{equation*}
D_\z = \d_\z+\z^{-1} \d_\ze\,, \qquad D_{\z_i} = \d_{\z_i}+\z_i^{-1} \d_{\ze_i}\,.
\end{equation*}
Note that $D_\z$ is the total derivative with respect to $\z$ if we set $\ze=\log\z$.
We will also need the formal power series
\begin{equation}\label{logf7}
\vartheta_{12}:=\zeta_{1}-\sum_{n=1}^\infty \frac{1}{n}z_{1}^{-n} z_{2}^{n} \,,
\end{equation}
which can be thought of as an expansion of $\log z_{12} = \log z_1+\log(1-\frac{z_2}{z_1})$.
Similarly, we define
\begin{equation}\label{logf8}
\vartheta_{21}:=\zeta_{2}-\sum_{n=1}^\infty \frac{1}{n}z_{1}^{n} z_{2}^{-n} \,,
\end{equation}
which is an expansion of $\log z_{21}$.

\begin{remark}\label{r-ze}
Note that $D_{z_1} \vartheta_{12}$ is the expansion of $z_{12}^{-1}$ in the domain $|z_1|>|z_2|$, while 
$D_{z_1} \vartheta_{21}$ is the expansion of $z_{12}^{-1}$ in the domain $|z_2|>|z_1|$. This implies that
\begin{equation}\label{logf9}
z_{12} D_{z_1} \vartheta_{12} = z_{12} D_{z_1} \vartheta_{21} =1 \,,
\end{equation}
and
\begin{equation}\label{logf10}
\de(z_1,z_2) = D_{z_1} (\vartheta_{12}-\vartheta_{21}) = \sum_{n\in\ZZ} z_{1}^{-n-1} z_{2}^{n}
\end{equation}
is the formal \emph{delta function}. 
\end{remark}

Now we introduce the notion of a logarithmic field, which is a central object of our study.

\begin{definition}\label{d2.1}
Let $V$ be a vector superspace.
The superspace 
\begin{equation*}
\LF(V) = \LF(V)_{\bar0} \oplus \LF(V)_{\bar1}
\end{equation*}
of \emph{logarithmic (quantum) fields} is defined by
\begin{equation*}
\LF(V)_\alpha = \Hom(V_{\bar0},V_{\alpha}(\!(\z)\!)[\ze]) \oplus \Hom(V_{\bar1},V_{\alpha+\bar1}(\!(\z)\!)[\ze]) \,,
\qquad \al\in\ZZ/2\ZZ \,.
\end{equation*}
In particular, when $V$ is a (purely even) vector space, we have
\begin{equation*}
\LF(V) = \Hom(V,V(\!(\z)\!)[\ze]) \,.
\end{equation*}
We denote the elements of $\LF(V)$ as $a(z,\ze)$ or $a(\z)$ for short. %, and the parity $p(a)\in\ZZ/2\ZZ$ of the field is defined by the parity of the coefficients.
\end{definition}

In the above definition, $V(\!(\z)\!)[\ze] = V[\![z]\!][z^{-1},\ze]$ stands for the space of polynomials in $\ze$ whose coefficients are
formal Laurent series in $z$ (i.e., formal power series in $z$ and $z^{-1}$ with finitely many negative powers of $z$).
When expanded, a logarithmic field $a(z,\ze)$ can be viewed as a formal power series in $z^{\pm1}$ and $\zeta$ (with uniformly bounded powers of $\ze$) with coefficients in $\End(V)$.
If $a(z,\ze)$ has parity $p(a)$, then all coefficients in its expansion have the same parity $p(a)$.

\begin{remark}
It is possible to consider more general logarithmic fields, which are formal power series in $\ze$ (cf.\ \cite{B,BS2}), and many of our results generalize to this case.
As in \cite{B}, the condition that the logarithmic fields are polynomials in $\ze$ is a consequence of the local nilpotency of the operator $\N$ defined below
(see Definition \ref{lm}).
\end{remark}

We can compose two logarithmic fields $a,b\in\LF(V)$ only if they depend on different formal variables.
This defines a linear map 
 \begin{equation*}
 a(z_{1})b(z_{2}) \colon V\rightarrow V(\!(z_{1})\!)[\zeta_{1}](\!(z_{2})\!)[\zeta_{2}]\, .
 \end{equation*}
 Note that $V(\!(z_{1})\!)[\zeta_{1}](\!(z_{2})\!)[\zeta_{2}]$ and $V(\!(z_{2})\!)[\zeta_{2}](\!(z_{1})\!)[\zeta_{1}]$ are two different subspaces of $V[\![\z_{1}^{\pm1},\z_{2}^{\pm1},\ze_1,\ze_2]\!]$, 
 whose intersection is $V[\![\z_1,\z_2]\!] [\z_1^{-1}, \z_2^{-1},\ze_1,\ze_2]$.
 
As for fields in vertex algebras, the product $a(z)b(z)$ is not well defined in general, and one has to do a normal ordering of the product.
Every logarithmic field $a\in\LF(V)$ can be expanded as
\begin{equation*}
a(z,\ze) = \sum_{n\in\ZZ} a_n(\ze) \z^{-n-1} \,, \qquad a_n(\ze) \in\Hom(V,V[\ze]) \,,
\end{equation*}
where for each $v\in V$ we have $a_n(\ze)v=0$ for sufficiently large $n$.
The \emph{annihilation} and \emph{creation parts} of $a(\z)$ are defined respectively as
\begin{align*}
a(\z)_- &= a(z,\ze)_- = \sum_{n\ge0} a_n(\ze) \z^{-n-1} \,, \\
a(\z)_+ &= a(z,\ze)_+ = \sum_{n<0} a_n(\ze) \z^{-n-1} \,.
\end{align*}
 The \emph{normally ordered product} of two logarithmic fields $a,b\in \LF(V)$ is defined as
 \begin{equation*}
 {:}a(z)b(z){:}=a(z)_{+} b(z)+(-1)^{p(a)p(b)}b(z)a(z)_{-} \, .
 \end{equation*}
Note that the normally ordered product ${:}a(z)b(z){:}$ is again a logarithmic field, because for $v\in V$, we have:
\begin{equation*}
a(z)_{+}v \in V[\![z]\!][\ze] \,, \qquad a(z)_{-}v \in V[z^{-1},\ze] \,,
\end{equation*}
which imply that
\begin{equation*}
a(z)_{+}b(z)v\in V(\!(z)\!)[\ze] \,, \qquad b(z)a(z)_{-}v\in V (\!(z)\!)[\ze] \,.
\end{equation*}
 
%The following definition of locality for logarithmic fields was considered in \cite[Definition 3.2]{B}.
 %\begin{definition}\label{d2.2}
%A pair of logarithmic fields $a,b$ is called local if it exists an integer $N\in \ZZ_{+}$ such that
%\begin{equation*}
%z_{12}^{N}\,a(z_1)b(z_2)=(-1)^{p(a)p(b)}z_{12}^{N}\,b(z_2)a(z_1)\, ,
%\end{equation*}
%where $z_{12}=z_1-z_2$. 
%\end{definition}

\subsection{Locality of logarithmic fields}\label{s2.1l}

Recall that, for ordinary vertex algebras, the locality of a 
pair of fields $a,b$ (assumed homogeneous with respect to parity)
is the condition
\begin{equation}\label{logf4}
z_{12}^{N}\,a(z_1)b(z_2)=(-1)^{p(a)p(b)}z_{12}^{N}\,b(z_2)a(z_1)\, , \qquad \z_{12}:=\z_1-\z_2 \,,
\end{equation}
for some integer $N\ge0$ (see  \cite{DL, Li, K}).
For generalized vertex algebras, one allows the power $N$ to be non-integral, where $z_{12}^{N}$ is expanded in the domain $|z_1|>|z_2|$
in the left-hand side of \eqref{logf4} and for $|z_1|<|z_2|$ in the right-hand side (see \cite{DL,M,BK2}).

For logarithmic vertex algebras, we will generalize \eqref{logf4} further, by replacing $N$ with a suitable linear operator acting on tensor products
of logarithmic fields. To this end, we introduce the following notion, which will be important throughout the paper.

 %the endomorphisms $\NN_{12}, \NN_{23},\NN_{13}$ in \eqref{logf-1} are given by
%\begin{equation*}
%\NN_{12} = \sum_{i=1}^{L} \Psi_{i} \otimes \Phi_{i} \otimes I \,, \quad
%\NN_{23} = \sum_{i=1}^{L} I \otimes \Psi_{i} \otimes \Phi_{i} \,, \quad
%\NN_{13} = \sum_{i=1}^{L} \Psi_{i} \otimes I \otimes \Phi_{i} \,.
%\end{equation*}
%Then, the condition $iii)$ in Definition \ref{lm} is equivalent to 

\begin{definition}\label{lm}
Let $U$ be a vector superspace. An (infinitesimal) \emph{braiding map} $\N$ on $U$ is an even linear operator
\begin{equation*}%\label{logf5}
\N\in \End(U)\otimes \End(U)\, ,
\end{equation*}
satisfying the following two conditions:
\begin{enumerate}
%\item\label{2.4-1}
%$\N$ is \emph{locally nilpotent}, i.e., for any $a,b\in U$ there exists a positive integer $r$ (depending on $a,b$) such that  $\N^r (a\otimes b) = 0$;
%
\item\label{2.4-2}
$\N$ is \emph{symmetric}, i.e., $\N P=P \N\,$ where $P$ is the transposition \eqref{logf1};

\smallskip
\item\label{2.4-3}
%$\N$ satisfies the identity 
$[\N_{12},\N_{23}]=[\N_{13},\N_{23}]=[\N_{12},\N_{13}]=0$, where $\N_{ij}$ are as in \eqref{logf-1}.
%{iii)}\; $\N ({\mu}\otimes I) = (\mu\otimes I) (\N_{23}+\N_{13})\,$ where $\mu$ is the composition in \eqref{logf2}.
\end{enumerate}
\end{definition}

%In this section, we will work with a braiding map $\NN$ on some subspace $\V$ of the space of logarithmic fields $\LF(V)$. 
Now we state the definition of locality for logarithmic fields and we provide one example; more examples are given in Section \ref{s4}. 

\begin{definition}\label{d2.3}
Let $\V$ be a subspace of $\LF(V)$, and $\NN$ be a braiding map on $\V$, which is \emph{locally nilpotent}, 
i.e., for any $a,b\in\V$ there exists a positive integer $r$ (depending on $a,b$) such that  $\NN^r (a\otimes b) = 0$.
Two logarithmic fields $a,b\in \V$ are called $\NN$-\emph{local} (or just \emph{local} for short) if 
\begin{equation}\label{logf11}
\mu\bigl(\z_{12}^N\,e^{\vartheta_{12}  \NN} a(\z_1) \otimes b(\z_2)\bigr) =(-1)^{p(a)p(b)}\mu\bigl( \z_{12}^N\,e^{\vartheta_{21} \NN}  b(\z_2)\otimes a(\z_1)\bigr) \,,
\end{equation}
for some integer $N\ge0$ (depending on $a$ and $b$). 
\end{definition}

%\begin{remark}Definition \ref{d2.2} is the particular case of Definition \ref{d2.3} when $\N=0$.\end{remark}

By abuse of notation, we will sometimes write for short 
$e^{\vartheta_{12}  \NN} a(\z_1) b(\z_2)$ instead of $\mu\bigl(e^{\vartheta_{12}  \NN} a(\z_1) \otimes b(\z_2)\bigr)$.
Furthermore, we will denote
\begin{equation}\label{z12N}
\z_{12}^N\,e^{\vartheta_{12} \NN}=\z_{12}^{N+\NN} \,, \qquad
\z_{12}^N\,e^{\vartheta_{21} \NN}=\z_{21}^{N+\NN} \,,
\end{equation}
provided that $N$ is even so that $\z_{12}^N = \z_{21}^N$.
Then the locality condition \eqref{logf11} can be reformulated so that it resembles \eqref{logf4}:
\begin{equation}\label{logf12}
\z_{12}^{N+\NN}a(\z_{1})b(\z_2)=(-1)^{p(a)p(b)}\z_{21}^{N+\NN}b(\z_{2})a(\z_1) 
%\,, \quad N\in 2\ZZ_+ \,.
\end{equation}
for some $N\in 2\ZZ_+$.
When necessary for clarity, we use the notation \eqref{logf11}.

Another way to rewrite \eqref{logf11} is
\begin{equation*}%\label{logf1p}
\mu\bigl(\z_{12}^N\,e^{\vartheta_{12}  \NN} (a \otimes b)(z_1,z_2)\bigr) =(-1)^{p(a)p(b)}\mu\bigl( \z_{12}^N\,e^{\vartheta_{21} \NN}  (b \otimes a)(z_2,z_1)\bigr) \,,
\end{equation*}
where we use the notation
\begin{equation}\label{logf2p}
(a \otimes b)(z_1,z_2) = a(z_1) \otimes b(z_2) \,, \qquad a,b\in\LF(V) \,.
\end{equation}
Then \eqref{logf11} is equivalent to the equation
\begin{equation}\label{logf3p}
\mu\bigl(\z_{12}^N\,e^{\vartheta_{12}  \NN} \al(z_1,z_2)\bigr) = \mu\bigl( \z_{12}^N\,e^{\vartheta_{21} \NN}  (P\al)(z_2,z_1)\bigr) \,,
\end{equation}
for $\al=a \otimes b$, where we used the transposition $P$ from \eqref{logf1} acting on $\V \otimes \V$.
The advantage of \eqref{logf3p} is that it can be extended by linearity to other elements $\al\in\V \otimes \V$,
and the set of such $\al$ is a subspace of $\V\otimes\V$.

\begin{remark}\label{rsymloc}
The symmetry of $\NN$ implies that if $\al\in\V \otimes \V$ satisfies \eqref{logf3p}, then \eqref{logf3p} holds with
$P\al$ in place of $\al$.
In particular, if a pair of logarithmic fields $a,b\in\V$ is $\NN$-local, then the pair $b,a$ is $\NN$-local.
\end{remark}

\begin{example}[Free boson]\label{ex2.5}  
Let $V=\mathbb{C}[x_{0},x_{1}, x_2,\dots]$, and $x(z)$ be the logarithmic field  
\begin{equation*}
x(z) = \sum_{n=0}^\infty x_n z^n + \partial_{x_0} \ze + \sum_{n=1}^\infty \partial_{x_n} \frac{z^{-n}}{-n} \,.
\end{equation*}
Then it is easy to see that 
\begin{equation*}
x(z_1)x(z_2)=\vartheta_{12}+{:}x(z_1)x(z_2){:} \,.
\end{equation*}
Letting $\NN=-\ad(\partial_{x_0})\otimes \ad(\partial_{x_0})$ and $\V=\Span\{x(z), I\}$ where $I$ denotes the identity operator, we find that
\begin{equation*}
\NN \bigl(x(z_1) \otimes x(z_2)\bigr) = -I \otimes I \,.
\end{equation*}
 Hence
\begin{align*}
e^{\vartheta_{12}\NN}x(z_1)x(z_2)
&=x(z_1)x(z_2)-\vartheta_{12} \\
&={:}x(z_1)x(z_2){:}
={:}x(z_2)x(z_1){:} =e^{\vartheta_{21}\NN}x(z_2)x(z_1)\,,
\end{align*}
which means that the field $x(z)$ is $\NN$-local with respect to itself. 
Additionally, we have 
\begin{equation*}
\NN \bigl(x(z_1) \otimes I\bigr) = \NN \bigl(I \otimes I\bigr)=0  \,,
\end{equation*}
which trivially implies that $x,I$ and $I,I$ are $\NN$-local. 
%Therefore, the space $\V$ is local.
\end{example}

\subsection{Local spaces of logarithmic fields}\label{s2.1ls}
From now on, for simplicity, we will assume that $\NN$ is a braiding map on the whole space $\LF(V)$, i.e.,
\begin{equation}\label{logf14-1}
\NN\in \End(\LF(V))\otimes \End(\LF(V))\,.
\end{equation} 
Thus, we can write $\NN$ as a finite sum
\begin{equation}\label{logf14}
\NN=\sum_{i=1}^{L}\Phi_{i}\otimes  \Psi_{i}\,, \qquad \Phi_{i}, \Psi_{i}\in \End(\LF(V)) \,, \quad p(\Phi_i)=p(\Psi_i) \,.
\end{equation}
Without loss of generality, we will assume that $\Phi_1,\dots,\Phi_L$ are linearly independent
and $\Psi_1,\dots,\Psi_L$ are linearly independent.
Then conditions \eqref{2.4-2} and \eqref{2.4-3} from Definition \ref{lm} are equivalent to the identities:
\begin{equation}\label{logf18s}
\NN=\sum_{i=1}^{L}\Phi_{i}\otimes  \Psi_{i}=\sum_{i=1}^{L} (-1)^{p(\Phi_i)} \Psi_{i}\otimes  \Phi_{i}
\end{equation}
and
\begin{equation}\label{logf18e}
[\Phi_i, \Psi_j] = [\Phi_i, \Phi_j] = [\Psi_i, \Psi_j] = 0
\,,\qquad 1\le i,j\le L\,,
\end{equation}
respectively. %These are easy to verify in specific examples and hold on the whole $\LF(V)$.
Note, however, that the operator $\NN$ will be locally nilpotent not on the whole $\LF(V) \otimes \LF(V)$ but only on $\V\otimes\V$
for some subspace $\V\subset\LF(V)$. 

\begin{remark}
The assumptions \eqref{logf14-1}--\eqref{logf18e} are satisfied in all examples that we know (cf.\ Section \ref{s4} below), and in all of them $\NN$ has the form (see \eqref{logf3}):
\begin{equation}\label{eq2.34}
\NN=\sum_{i=1}^{L} \ad (\phi_{i})\otimes \ad(\psi_{i})\,, \qquad \phi_{i},\psi_{i}\in \End(V)\,.
\end{equation}
\end{remark}

Let $\V$ be a subspace of $\LF(V)$. We say that $\V$ is \emph{component-wise $\NN$-invariant} if
%\begin{equation}\label{logf-ls1}
%\NN \bigl(\V\otimes\LF(V) \bigr) \subset \V\otimes\LF(V) \,.
%\end{equation} 
%By the symmetry of $\NN$, we also have
%\begin{equation}\label{logf-ls2}
%\NN \bigl(\LF(V)\otimes\V \bigr) \subset \LF(V)\otimes\V \,.
%\end{equation} 
%In terms of components, \eqref{logf-ls1} and \eqref{logf-ls2} are equivalent to the conditions
\begin{equation}\label{logf-ls3}
\Phi_i \V \subset\V \,, \quad \Psi_i \V \subset\V \,, \qquad  1\le i\le L \,.
\end{equation} 
In particular, this implies $\NN(\V\otimes\V) \subset \V\otimes\V$, and it makes sense to talk about the restriction of $\NN$ to $\V\otimes\V$.

\begin{definition}\label{d2.5}
Let $\NN$ be a braiding map on $\LF(V)$.
We will say that a subspace $\V\subset\LF(V)$ is $\NN$-\emph{local} (or just \emph{local} when $\NN$ is fixed) if the following three conditions hold:

\begin{enumerate}
\item\label{2.5-1}
$\V$ is component-wise $\NN$-invariant; %(see \eqref{logf-ls3});

\item\label{2.5-2}
the restriction of $\NN$ to $\V\otimes\V$ is locally nilpotent; %(see \deref{d2.3});

\item\label{2.5-3}
every two fields $a,b\in \V$ are $\NN$-local.
\end{enumerate}
\end{definition}

Note that, by linearity, if $\V$ is a local space as in Definition \ref{d2.5},  then all elements $\al\in \V \otimes \V$ are $\NN$-local in the sense of
\eqref{logf3p}.
In particular, $\NN^i (a\otimes b)$ is local
for all $a, b\in \V$ and $i\in\ZZ_{+}$; hence, there exists an integer $N\in\ZZ_{+}$ such that
\begin{equation}\label{logf15}
\begin{split}
&\mu\bigl(z_{12}^{N}e^{\vartheta_{12}\NN}\NN^{i} a(z_{1})\otimes b(z_{2})\bigr)\\
&\qquad =(-1)^{p(a)p(b)}\mu\bigl(z_{12}^{N}e^{\vartheta_{21}\NN}\NN^{i} b(z_{2})\otimes a(z_{1})\bigr)\, .
\end{split}
\end{equation}
By abuse of notation, we will sometimes write for short 
$e^{\vartheta_{12}  \NN}\NN^{i}  a(\z_1) b(\z_2)$ instead of $\mu\bigl(e^{\vartheta_{12}  \NN}\NN^{i}  a(\z_1) \otimes b(\z_2)\bigr)$.
Then the locality condition \eqref{logf15} can be reformulated as
\begin{equation*}%\label{logf12i}
\z_{12}^{N+\NN}\NN^{i} a(\z_{1})b(\z_2)=(-1)^{p(a)p(b)}\z_{21}^{N+\NN}\NN^{i} b(\z_{2})a(\z_1) 
%\,, \quad N\in 2\ZZ_+ \,.
\end{equation*}
for some $N\in 2\ZZ_+$. Since $\NN$ is nilpotent on $a\otimes b$, we can pick one $N$ that works for all $i\in\ZZ_{+}$.

In the following lemmas, we will extend a local space $\V$ to a larger space of logarithmic fields that includes the identity field and all the derivatives of the fields in $\V$. 
Next, in Section \ref{sdong} below, we will obtain a larger extension containing all $(n+\NN)$-th products of the quantum fields in $\V$ (cf.\ Definition \ref{d2.10} below). 

Let $\V_{I}=\V+\CC I$, where $I$ is the identity operator considered as a field.

\begin{lemma}\label{lem2.6} 
Let\/ $\NN$ be a braiding map on\/ $\LF(V)$, and\/ $\V\subset \LF(V)$ be an\/ $\NN$-local space.
Suppose that
\begin{equation}\label{1.2}
\NN(a\otimes I)=\NN(I\otimes a)=0 \quad\text{for all}\;\; a\in\LF(V)\, .
\end{equation}
Then\/ $\V_{I}$ is an\/ $\NN$-local space.
\end{lemma}
\begin{proof}  
This is obvious from the definitions.
%The first claim is a trivial consequence of \eqref{1.2} and Definition \ref{lm}. The second claim also follows from \eqref{1.2} by observing that \eqref{logf11} holds for $a\in\LF(V)$, $b=I$ and $N=0$.
%By assumption, the restriction of $\NN$ to $\V\otimes\V$ satisfies conditions (i)--(iii) from Definition \ref{lm}. Then it is obvious from \eqref{1.2} that the restriction of $\NN$ to $\V_I\otimes\V_I$ satisfies the same conditions;
%hence, $\NN$ is a braiding map on $\V_I$. If $\V$ is a local space then all $a,b\in \V$ are $\NN$-local, while \eqref{1.2} implies that \eqref{logf11} holds for $a\in\V_I$, $b=I$ and $N=0$.
%Thus $\V_{I}$ is a local space.
\end{proof}

%Note that if $I\in \V$ and $\V$ is a local space, then from the trivial identity $a(z)I=Ia(z)$ and the locality between $I,a\in \V$ we have that \eqref{eq2.12} is satisfied. Hence, \eqref{1.2} is true for any local space such that $I\in\V$.

Because of the above lemma, from now on we will always assume that $I\in\V$.
Now let $\V_{D}=\CC[D_z]\V$ be the minimal $D_{z}$-invariant subspace of $\LF(V)$ containing the space $\V$. 

\begin{lemma}\label{l2.9} 
Let\/ $\NN$ be a braiding map on\/ $\LF(V)$ such that
\begin{equation}\label{logfdn} 
\bigl[D_{z}\otimes I,\NN\bigr] = \bigl[I\otimes D_{z},\NN\bigr] = 0\ \quad \text{on} \quad \LF(V) \otimes \LF(V) \,.
\end{equation}
Then for every\/ $\NN$-local space\/ $\V\subset \LF(V)$, the space\/ $\V_D$ is\/ $\NN$-local.
\end{lemma}
\begin{proof} 
Assumption \eqref{logfdn} is equivalent to:
\[
[D_z,\Phi_i] = [D_z,\Psi_i] = 0 \,, \qquad 1\le i\le L\, .
\]
It is obvious that if $\V$ satisfies conditions \eqref{2.5-1}, \eqref{2.5-2} from Definition \ref{d2.5}, then $\V_D$ satisfies the same conditions.
It remains to prove that any two fields from $\V_D$ are local. 

Suppose that $\al(z_1,z_2) \in\V_D\otimes\V_D$ is such that all $\NN^i\al$ ($i\in\ZZ_+$) are local in the sense of \eqref{logf3p}.
Applying the derivative $D_{z_1}$ to both sides of \eqref{logf3p} with $N+1$ in place of $N$, we obtain using \eqref{logf9} that 
\begin{align*}
&z_{12}^{N+1+\NN}D_{z_1} \al(z_1,z_2) + (N+1+\NN)z_{12}^{N+\NN} \al(z_1,z_2) \\
& \;\;= z_{12}^{N+1+\NN} D_{z_1} (P\al)(z_2,z_1) + (N+1+\NN) z_{12}^{N+\NN} (P\al)(z_2,z_1) \,.
\end{align*}
Taking $N$ sufficiently large so that \eqref{logf3p} is satisfied for both $\al$ and $\NN\al$, we obtain that the second terms on the left and right side above cancel each other. 
Hence, $D_{z_1}\al$ is $\NN$-local. By symmetry and repeating the same argument, we get that all $D_{z_1}^m D_{z_2}^n \al$ are local for $m,n\in\ZZ_+$.

Starting from $a,b\in\V$, the above claim for $\al(z_1,z_2)=a(z_1) \otimes b(z_2)$ implies that the fields
$D_z^m a$ and $D_z^n b$ are local for all $m,n\in\ZZ_+$.
\end{proof}

\subsection{$(n+\NN)$-products}\label{s2.2} 
In this subsection, following \cite{B}, we define products of local logarithmic fields.
As before, let $\NN$ be a braiding map on $\LF(V)$ (see \eqref{logf14-1}--\eqref{logf18e}),
and let $\V$ be an $\NN$-local subspace of $\LF(V)$, as in \deref{d2.5}.
Observe that, since $a,b\in\V$ are $\NN$-local, both sides of \eqref{logf11} are linear maps from $V$ to $V[\![\z_1,\z_2]\!] [\z_1^{-1}, \z_2^{-1}, \ze_1,\ze_2]$; 
hence the following definition makes sense. 

\begin{definition}\label{d2.10} Let
$a, b\in \V$ be two local logarithmic fields, and $N$ be from \eqref{logf11}.
For $n\in\ZZ$, the \emph{$(n+\NN)$-th product} $a_{(n+\NN)}b$ is defined by ($v\in V$):
\begin{equation}\label{logf16}
\begin{split}
\bigl(a(&z)_{(n+\NN)}b(z)\bigr)v \\
&= D_{\z_1}^{(N-1-n)}\mu \Bigl(\z_{12}^Ne^{\vartheta_{12}\NN}  
a(\z_1,\ze_1)\otimes  b(\z_2,\ze_2) \Bigr)v\Big|_{ \substack{\z_1=\z_2=\z \\ \ze_1=\ze_2=\ze} }\,.
\end{split}
\end{equation}
%for $v\in V$. %and $n\le N-1$. For $n\ge N$, we let $a_{(n+\N)}b=0$. 
From now on, to simplify the notation, we will suppress the dependence on $\ze$ and understand that setting $\z_1=z$ automatically sets $\ze_1=\ze$.
\end{definition}

Note that, by definition, $a_{(n+\NN)}b=0$ for $n\ge N$, since $D_{\z_1}^{(k)} := 0$ for $k<0$.
It is easy to check that $a_{(n+\NN)}b$ is again a logarithmic field with parity $p(a_{(n+\NN)}b)=p(a)+p(b)$, and it does not depend on the choice of $N$ satisfying \eqref{logf11}. 

More generally, for any local element $\al\in\V\otimes\V$ satisfying \eqref{logf3p}, $n\in\ZZ$ and $v\in V$, we let
\begin{equation}\label{logf17-al}
(\hat{\mu}_{(n)}\al)(z)v =  D_{\z_1}^{(N-1-n)}\mu \Bigl(\z_{12}^Ne^{\vartheta_{12}\NN} \al(\z_1,\z_2) \Bigr)v\Big|_{\z_1=\z_2=\z}\,.
\end{equation}
This defines an even linear map
\begin{equation}\label{logf17-mu}
\hat{\mu}_{(n)} \colon\V\otimes \V\rightarrow \LF(V)\,, \qquad 
\hat{\mu}_{(n)}(a\otimes b) = a_{(n+\NN)}b\,.
\end{equation}
In particular, we will use the notation
\begin{equation}\label{logf17}
\NN^{i}a_{(n+\NN)}b = \hat{\mu}_{(n)}\bigl(\NN^{i}(a\otimes b)\bigr),
\qquad i\in\ZZ_+ \,, \;\; n\in\ZZ \,.
\end{equation}
Notice that, since $\NN$ is locally nilpotent, we have $\NN^r(a\otimes b)=0$ for some $r\ge0$ and
\begin{equation}\label{logf17z}
\NN^{i} a(z)_{(n+\NN)}b(z) = 0 \qquad\text{if \;$i\ge r$ \; or \;$n\ge N$} \,.
\end{equation}

%Then $\NN^{i}a_{(n+\NN)}b\,$  is again a logarithmic field with parity $p(a)+p(b)$.
%Hence, from a local pair $(a,b)_{\NN}$ we can construct the family of fields 
%\begin{equation*}
%a(z)_{(n+\NN)}b(z), \; \NN a(z)_{(n+\NN)}b(z),  \; \NN^{2} a(z)_{(n+\NN)}b(z),\dots \quad (n\in\mathbb{Z}) \,.
%\end{equation*}

It is easy to check that the $(n+\NN)$-th products satisfy the following properties:
\begin{align}
\label{2.29}
a_{(n+\NN)} I &= 0\,, & n&\ge0\,,\\
\label{2.30}
a_{(-n-1+\NN)} I &= D_z^{(n)} a \,, & n&\ge0\,,\\
\label{2.31}
I_{(m+\NN)} a &= \de_{m,-1} a\,, & m&\in\ZZ \,.
\end{align}
Furthermore, using the Leibniz rule and \eqref{logf9}, we obtain: 
\begin{align}
(D_{z}a)_{(n+\NN)}b&=-(n+\NN\,)a_{(n-1+\NN)}b\, ,\\
D_{z}(a_{(n+\NN)}b)&=(D_{z}a)_{(n+\NN)}b+a_{(n+\NN)}(D_{z}b)\,,\\
\d_{\ze}(a_{(n+\NN)}b)&=(\d_{\ze}a)_{(n+\NN)}b+a_{(n+\NN)}(\d_{\ze}b)+\NN a_{(n+\NN)}b\,.
\end{align}

\subsection{Dong's Lemma}\label{sdong}
In this subsection, we generalize Dong's Lemma (see, e.g., \cite{K}) to the case of logarithmic fields.
We continue using the notation from the previous subsection.

Let $\V$ be an $\NN$-local subspace of $\LF(V)$ such that the identity field $I \in \V$. 
We denote by $\V^{1}$ the minimal subspace of $\LF(V)$ containing $\hat{\mu}_{(n)}(\V\otimes \V)\subset \LF(V)$ for all $n\in \ZZ$, i.e.,
\begin{equation}\label{logf-v1}
\V^{1}=\sum_{n\in \ZZ}\hat{\mu}_{(n)}(\V\otimes \V)\, .
\end{equation}
Note that $\V_{D} \subset \V^{1}$ because of $I\in \V$ and \eqref{2.30}.
 %we have $a = a_{(-1+\NN)} I$ and $D_z a = a_{(-2+\NN)} I$.

\begin{lemma}[Dong's Lemma]\label{l2.12b}
Let\/ $\NN$ be a braiding map on\/ $\LF(V)$, and\/ $\V$ be an\/ $\NN$-local subspace of\/ $\LF(V)$.
Suppose that for all\/ $n\in\ZZ$, we have 
\begin{equation}\label{logf18b}
\NN (\hat{\mu}_{(n)}\otimes I) = (\hat{\mu}_{(n)}\otimes I) (\NN_{13}+\NN_{23}) \quad \text{on} \quad  \V^{\otimes 3}\,.
\end{equation}
Then the space\/ ${\V^{1}}$ is\/ $\NN$-local.
 \end{lemma}

\begin{proof}
We will check each of the conditions in Definition \ref{d2.5}.

\medskip
\eqref{2.5-1} 
In components, condition \eqref{logf18b} is equivalent to the fact that all $\Phi_i$ are derivations of the $(n+\NN)$-th products:
\begin{equation}\label{logf18d}
\Phi_{i}(a_{(n+\NN)}b)=\Phi_{i}(a)_{(n+\NN)}b+a_{(n+\NN)}\Phi_{i}(b)\,,
\qquad a,b\in\V \,.
\end{equation}
for $a,b\in\V$. Then the symmetry of $\NN$ implies
\begin{equation}\label{logf18dps}
\Psi_{i}(a_{(n+\NN)}b)=\Psi_{i}(a)_{(n+\NN)}b+a_{(n+\NN)}\Psi_{i}(b)\,,
\qquad a,b\in\V \,.
\end{equation}
Now, since $\V$ is component-wise $\NN$-invariant (see \eqref{logf-ls3}), the above equations
\eqref{logf18d} and \eqref{logf18dps} imply that $\V^1$ is component-wise $\NN$-invariant.

\medskip
\eqref{2.5-2} 
Similarly to \eqref{logf18b}, we have that \eqref{logf18dps} is equivalent to 
\begin{equation*}%\label{logf18h}
\NN (I\otimes \hat{\mu}_{(n)}) = ( I\otimes \hat{\mu}_{(n)}) (\NN_{12}+\NN_{13}) \qquad \text{on} \quad \V^{\otimes 3}\,.
\end{equation*}
Together with \eqref{logf18b} (with $m$ in place of $n$), this gives
\begin{equation*}
\NN(\hat{\mu}_{(m)}\otimes \hat{\mu}_{(n)})=(\hat{\mu}_{(m)}\otimes \hat{\mu}_{(n)})(\NN_{13}+\NN_{14}+\NN_{23}+\NN_{24}) \qquad \text{on} \quad  \V^{\otimes 4}\,,
\end{equation*}
for all $m,n\in\ZZ$. From here, we obtain
\begin{equation*}
\NN^r(\hat{\mu}_{(m)}\otimes \hat{\mu}_{(n)})=(\hat{\mu}_{(m)}\otimes \hat{\mu}_{(n)})(\NN_{13}+\NN_{14}+\NN_{23}+\NN_{24})^r \qquad \text{on} \quad  \V^{\otimes 4}\,,
\end{equation*}
for all $r\ge1$.
Applying both sides of this equation to $a\otimes b\otimes c\otimes d \in\V^{\otimes 4}$ gives
\begin{equation*}
\NN^r \bigl(a_{(m+\NN)}b \otimes c_{(n+\NN)}d\bigr) =
(\hat{\mu}_{(m)}\otimes \hat{\mu}_{(n)})(\NN_{13}+\NN_{14}+\NN_{23}+\NN_{24})^r (a\otimes b\otimes c\otimes d).
\end{equation*}
Now observe that, by \eqref{logf18e}, the operators $\NN_{13}$, $\NN_{14}$, $\NN_{23}$, $\NN_{24}$ commute with each other.
Since each of them is nilpotent on $a\otimes b\otimes c\otimes d$, it follows that their sum is also nilpotent on $a\otimes b\otimes c\otimes d$.
Thus, $\NN$ is nilpotent on $a_{(m+\NN)}b \otimes c_{(n+\NN)}d$.

\medskip
\eqref{2.5-3} 
Now we will prove that all fields from $\V^{1}$ are local in two steps:

\begin{enumerate}
\smallskip
\item[(i)] %\label{2.5-3i}
For all $a,b\in \V$, the field $a_{(n+\NN)}b$ is local with respect to any $c\in \V$.

\smallskip
\item[(ii)] %\label{2.5-3ii}
For all $a,b,c,d\in \V$, the fields $a_{(n+\NN)}b$ and $c_{(m+\NN)}d$ are local.
\end{enumerate}

\smallskip
To prove claim (i), notice that because $\V$ is a local space we can find a sufficiently large $N$ such that:
\begin{equation}\label{dong1}
\begin{split}
\z_{12}^{N+\NN}a(\z_{1})b(\z_2) &=(-1)^{p(a)p(b)}\z_{21}^{N+\NN}b(\z_{2})a(\z_1) \,, \\
\z_{13}^{N+\NN} \, a(\z_1) c(\z_3) &=(-1)^{p(a)p(c)} \z_{31}^{N+\NN} \, c(\z_3) a(\z_1)\, , \\
 \z_{23}^{N+\NN} \, b(\z_2) c(\z_3) &=(-1)^{p(b)p(c)} \z_{32}^{N+\NN} \, c(\z_3) b(\z_2) \,.
 \end{split}
 \end{equation}
 Here we used notation \eqref{z12N}, namely $\z_{ij}^{N+\NN}=z_{ij}^{N}e^{\vartheta_{ij}\NN}$ and $\z_{ji}^{N+\NN}=z_{ij}^{N}e^{\vartheta_{ji}\NN}$ for $1\le i<j\le 3$. 

As in the proof of \cite[Lemma 3.4]{B}, we compute for $n'=N-1-n\geq0$ and $v\in V$:
\begin{align*}
&z_{23}^{2N+n'+\NN} (a_{(n+\NN)}b)(z_2) \, c(z_3) v \\
&=\mu\Bigl(z_{23}^{2N+n'+\NN}D_{z_{1}}^{(n')}\mu\bigl(z_{12}^{N+\NN}a(z_{1})\otimes b(z_{2})\bigr)\Big|_{z_{1}=z_{2}}
\otimes c(z_{3})\Bigr) v \\
&=\mu(\mu\otimes I)\Bigl(z_{23}^{2N+n'}D_{z_{1}}^{(n')}e^{\vartheta_{23}\NN_{13}}e^{\vartheta_{23}\NN_{23}}z_{12}^{N+\NN_{12}}a(z_{1})\otimes b(z_{2})\otimes c(z_{3})\Bigr)\Big|_{z_{1}=z_{2}} v \,.
\end{align*}
For the last equality we used that, by Definition \ref{lm}\eqref{2.4-3} and \eqref{logf18b}, we have
\begin{equation*}
e^{\ze \NN}(\mu\otimes I)=(\mu\otimes I)e^{\ze \NN_{13}}e^{\ze\NN_{23}} \,.
\end{equation*}
Moreover, $\vartheta_{13} = \vartheta_{23}$ after setting $z_1=z_2$ and $\ze_1=\ze_2$. 
Similarly, we can replace $z_{23}^{2N+n'}$ with $z_{13}^{N+n'} z_{23}^{N}$ and obtain:
\begin{equation*}
z_{13}^{N+n'+\NN_{13}} z_{23}^{N+\NN_{23}} D_{z_{1}}^{(n')} z_{12}^{N+\NN_{12}} a(z_{1}) b(z_{2}) c(z_{3}) v \big|_{z_{1}=z_{2}} \,.
\end{equation*}
%using again the shorthand notation \eqref{z12N}.
Next, we apply the Leibniz rule in the form
\begin{equation*}
f(z,\ze) D_z^{(n)} g(z,\ze) = \sum_{i=0}^n (-1)^{n-i} D^{(i)}_z \Bigl(\bigl( D_z^{(n-i)} f(z,\ze) \bigr) g(z,\ze)\Bigr),
\end{equation*}
which holds for arbitrary functions $f$ and $g$.
Using that $D_{z_{1}}^{(k)} z_{12}^{m+\NN} = \binom{m+\NN}{k} z_{12}^{m-k+\NN}$,
we obtain
\begin{align*}
&z_{23}^{2N+n'+\NN} (a_{(n+\NN)}b)(z_2) \, c(z_3) v \\
&\;\;=\sum_{i=0}^{n'} (-1)^{n'-i} \binom{N+n'+\NN_{13}}{n'-i} \\
&\qquad\times D^{(i)}_{z_{1}} \Bigl( 
z_{13}^{N+i+\NN_{13}} z_{23}^{N+\NN_{23}} z_{12}^{N+\NN_{12}} a(z_{1}) b(z_{2}) c(z_{3}) v 
\Bigr)\Big|_{z_{1}=z_{2}} \,.
 \end{align*}
 Repeating the same calculation, we also get
 \begin{align*}
&z_{32}^{2N+n'+\NN} c(z_3) \, (a_{(n+\NN)}b)(z_2) v \\
&\;\;=\sum_{i=0}^{n'} (-1)^{n'-i} \binom{N+n'+\NN_{13}}{n'-i} \\
&\qquad\times D^{(i)}_{z_{1}} \Bigl( 
z_{31}^{N+i+\NN_{12}} z_{32}^{N+\NN_{13}} z_{12}^{N+\NN_{23}} c(z_{3}) a(z_{1}) b(z_{2}) v 
\Bigr)\Big|_{z_{1}=z_{2}} \,.
 \end{align*}
Thus, the locality of $a_{(n+\NN)}b$ and $c$ will follow from the equality
\begin{equation}\label{dong2}
\begin{split}
z_{13}^{N+\NN_{13}} & z_{23}^{N+\NN_{23}} z_{12}^{\NN_{12}} a(z_{1}) b(z_{2}) c(z_{3}) \\
&=
z_{31}^{N+\NN_{12}} z_{32}^{N+\NN_{13}} z_{12}^{\NN_{23}} c(z_{3}) a(z_{1}) b(z_{2}) \,.
\end{split}
\end{equation}

We can derive \eqref{dong2} from \eqref{dong1}, but the derivation is subtle because we first
need to apply the operator $z_{12}^{\NN} = e^{\vartheta_{12} \NN}$ to $a(z_{1}) \otimes b(z_{2})$ before we can use the locality
of the pairs $a,c$ and $b,c$. This is where we use that $\V$ is a local space. Suppose that $\NN^r (a \otimes b) = 0$,
then we get from \eqref{logf14}:
\begin{equation}\label{Nexp}
\begin{split}
z_{12}^{\NN} \, & a(z_{1}) b(z_{2}) \\
&= \sum_{s=0}^{r-1} \vartheta_{12}^{(s)} \!\sum_{\substack{1\le i_1,\dots,i_s \le L \\ 1\le j_1,\dots,j_s \le L}}
\bigl(\Phi_{i_1} \cdots \Phi_{i_s} a(z_{1})\bigr) \bigl(\Psi_{j_1} \cdots \Psi_{j_s} b(z_{2})\bigr) \,.
\end{split}
\end{equation}
We have similar expressions for $z_{13}^{\NN} a(z_{1}) c(z_{3})$ and $z_{23}^{\NN} b(z_{2}) c(z_{3})$.
Now we can use the locality of all $\Psi_{j_1} \cdots \Psi_{j_s} b\in \V$
with $c$, increasing if necessary $N$ so that equations \eqref{dong1} hold for all of them.
This implies
\begin{equation*}
z_{13}^{N+\NN_{13}} z_{23}^{N+\NN_{23}} z_{12}^{\NN_{12}} a(z_{1}) b(z_{2}) c(z_{3}) \\
=
z_{13}^{N+\NN_{12}} z_{32}^{N+\NN_{23}} z_{12}^{\NN_{13}} a(z_{1}) c(z_{3}) b(z_{2}) \,.
\end{equation*}
Then we use the locality of all $\Phi_{i_1} \cdots \Phi_{i_s} a\in \V$ with all $\Psi_{k_1} \cdots \Psi_{k_t} c\in\V$,
again increasing $N$ as needed. We obtain
\begin{equation*}
z_{13}^{N+\NN_{12}} z_{32}^{N+\NN_{23}} z_{12}^{\NN_{13}} a(z_{1}) c(z_{3}) b(z_{2}) \\
=
z_{31}^{N+\NN_{12}} z_{32}^{N+\NN_{13}} z_{12}^{\NN_{23}} c(z_{3}) a(z_{1}) b(z_{2}) \,,
\end{equation*}
which proves \eqref{dong2} and completes the proof that $a_{(n+\NN)}b$ and $c$ are $\NN$-local.
 
Finally, we prove step (ii). By step (i) above, we have that any field in $\V$ is local with respect to $c_{(m+\NN)}d$.
In particular, $a$ and $c_{(m+\NN)}d$ are local, and $b$ and $c_{(m+\NN)}d$ are local. Using step (i) again with $c_{(m+\NN)}d$
in place of $c$, we see now that $a_{(n+\NN)}b$ is local with $c_{(m+\NN)}d$. This completes the proof of the lemma.
\end{proof}

\begin{remark}
It is not hard to prove that, if $\NN$ satisfies \eqref{logf18b} and $[\N_{12},\N_{23}]$ $=[\N_{13},\N_{23}]=[\N_{12},\N_{13}]=0$ on $\V^{\otimes 3}$, then the latter holds also on
$(\V^{1})^{\otimes 3}$. However, in all our examples, the last condition is satisfied on the whole $\LF(V)^{\otimes 3}$.
\end{remark}
%\begin{remark}
%If we expand $\NN$ as in \eqref{logf14}, then condition \eqref{logf18b} is equivalent to the fact that all $\Phi_i$ are derivations of the $(n+\NN)$-th products:
%\begin{equation}\label{logf18d}
%\Phi_{i}(a_{(n+\NN)}b)=\Phi_{i}(a)_{(n+\NN)}b+a_{(n+\NN)}\Phi_{i}(b)\,,
%\end{equation}
%for $a,b\in\V$. Then the symmetry of $\NN$ implies that the same is true for $\Psi_i$, which is now equivalent to \eqref{logf18h},
%thus giving another proof of \eqref{logf18h}.
%\end{remark}
%\begin{remark}
%An alternative formulation of \leref{l2.12b} is to assume that $\NN$ is symmetric only on $\V\otimes\V$ but satisfies both \eqref{logf18b} and \eqref{logf18h} on $\V^{\otimes 3}$.
%Then one can show that $\NN$ is symmetric on $\V^1\otimes\V^1$, and the rest of the proof works as before.
%In practice, all three conditions: \eqref{logf18b}, \eqref{logf18h}, and the symmetry of $\NN$, 
%are easy to check in examples.
%\end{remark}

In the next lemma, we make a stronger assumption on $\NN$, which is nevertheless sufficient for our purposes.

\begin{lemma}\label{l2.12}
Let\/ $\NN$ be a braiding map on\/ $\LF(V)$, 
such that\/ \eqref{logf18b} holds for any\/ $\NN$-local subspace\/ $\V\subset\LF(V)$.
Then, for any\/ $\NN$-local subspace\/ $\V$ containing $I$,
there exists an\/ $\NN$-local subspace\/ $\overline{\V} \subset\LF(V)$ containing\/ $\V$, 
which is closed under all\/ $(n+\NN)$-th products $(n\in\ZZ)$.
\end{lemma}
 \begin{proof}
Starting from the local space $\V$, let $\V^1\supset\V$ be given by \eqref{logf-v1}. Then $\V^{1}$ is again a local space, by \leref{l2.12b}. 
By assumption, \eqref{logf18b} now holds for $\V^1$ in place of $\V$; hence, we can apply \leref{l2.12b} for $\V^1$ and construct the local space $\V^{2}=(\V^{1})^{1}$.
Continuing by induction, we let $\V^{m+1} = (\V^{m})^1$, and we obtain a sequence of $\NN$-local subspaces
\[\V=\V^0 \subset \V^{1}\subset \V^2 \subset\cdots \,.\]
%Then the union
%\[ \overline{\V} = \bigcup_{m\geq 0}\V^{m} \]
Their union $\overline{\V}$ is $\NN$-local and closed under all $(n+\NN)$-th products.
 \end{proof}

Notice that the space $\overline{\V}$ constructed in the proof of the above lemma is the unique minimal such subspace.

\subsection{Braiding maps compatible with an algebra}\label{slocmap}

The results of Lemmas \ref{lem2.6}, \ref{l2.9}, \ref{l2.12b} suggest the following notion.

\begin{definition}\label{def1}
Let $U$ be a unital differential algebra, with a (not necessarily associative) product $\mu\colon U\otimes U\to U$,
a unit element $1\in U$, and a derivation $D\in\Der(U)$.
We say that a braiding map $\N$ on $U$ is \emph{compatible} with the algebra structure if it has the following three properties:

\begin{enumerate}
\item\label{2.16-1}
$\N(a \otimes 1) = 0$ for all $a\in U$;

\smallskip
\item\label{2.16-2}
$[D \otimes I, \N]=0$;

\smallskip
\item\label{2.16-3}
$\N (\mu\otimes I) = (\mu\otimes I) (\N_{13}+\N_{23})$.
\end{enumerate}
Sometimes we will denote a unital differential algebra with a compatible braiding map as a quintuple $(U,1,D,\mu,\N)$.
\end{definition}

As before, if we write
\begin{equation}\label{logf-u1}
\N=\sum_{i=1}^{L}\Phi_{i}\otimes  \Psi_{i}\,, \qquad \Phi_{i}, \Psi_{i}\in \End(U) \,, \quad p(\Phi_i)=p(\Psi_i) \,,
\end{equation}
then, combined with the symmetry of $\N$, conditions \eqref{2.16-1}--\eqref{2.16-3} above are equivalent to the following:
\begin{align}
\label{logf-u2}
\Phi_i 1 = \Psi_i 1 = 0 &\,, \\
\label{logf-u3}
[D,\Phi_i] = [D,\Psi_i] = 0 &\,, \\
\label{logf-u4}
\Phi_i, \Psi_i \in\Der(U) &\,, \qquad 1\le i\le L \,.
\end{align}

\begin{remark}
The identities in Definition \ref{def1} are similar to the identities satisfied by \emph{braiding maps} in quantum vertex algebras \cite{EK, FR,DGK}.
We describe the similarities between logarithmic vertex algebras and quantum vertex algebras in Remark \ref{re3.8} below.
\end{remark}

Lemma \ref{l2.12} provides examples of algebras with compatible braiding maps as follows.
 
\begin{corollary}\label{cor3.10-1}  
Let\/ $\NN$ be a braiding map on\/ $\LF(V)$, 
such that\/ \eqref{logf18b} holds for any\/ $\NN$-local subspace\/ $\V\subset\LF(V)$.
Then, for every fixed $n\in\ZZ$, we have the unital differential algebra with a compatible braiding map\/
 $(\overline{\V}, I, D_{z}, \hat{\mu}_{n},\NN)$.
\end{corollary}

Now we will provide a construction of such compatible braiding maps, which appears in all examples from Section \ref{s4} below.
Fix a vector superspace $V$, and consider the associative algebra $U=\End(V)$ with the product $\mu$ given by composition,
the unit element $I$ the identity operator, and the trivial derivation $D=0$. Let $\NN$ be a braiding map on $U$, compatible
with the algebra structure. 
We extend the action of $\Phi_i$ and $\Psi_i$ to $\LF(V)$ by acting only on the coefficients in front
of powers of $z^{\pm1}$ and $\ze$ in the logarithmic fields. Then this action clearly commutes with the derivative $D_z$.
The remaining properties are also easy to check, leading to the following result.

\begin{lemma}\label{loccomp}
Let\/ $(\End(V),I,0,\mu,\NN)$ be a compatible structure as above, and extend the action of\/ $\NN$ to\/ $\LF(V)\otimes\LF(V)$
by acting on the coefficients of logarithmic fields. 
Then $\NN$ is a braiding map on\/ $\LF(V)$, which satisfies\/ \eqref{1.2}, \eqref{logfdn}, and \eqref{logf18b} for any\/ $\NN$-local subspace\/ $\V\subset\LF(V)$.
%Consequently, 
%Then, for every fixed $n\in\ZZ$ and an\/ $\NN$-local subspace\/ $\V\subset\LF(V)$, we have the compatible structure\/ $(\overline{\V}, I, D_{z}, \hat{\mu}_{n},\NN)$.
\end{lemma}

In particular, such braiding maps can be constructed as in \eqref{eq2.34}.

\begin{corollary}\label{cor3.10-2}  
Let\/ $V$ be a vector superspace, $\N$ be a braiding map on\/ $V$, and let\/ $\NN=(\ad\otimes\ad)(\N);$ cf.\ \eqref{logf3}, \eqref{logf-n}, \eqref{eq2.34}.
Then\/ $\NN$ is a braiding map on\/ $\LF(V)$ satisfying\/ \eqref{1.2}, \eqref{logfdn}, and \eqref{logf18b} for any\/ $\NN$-local subspace\/ $\V\subset\LF(V)$.
\end{corollary}

\subsection{Operator product expansion}\label{OPE}

As before, we denote by $\z_{1},\z_{2},\dots$ and $\ze_{1},\ze_{2},\dots$ commuting even formal variables, where we think of $\ze_i$ as $\log z_i$,
and we continue to use the notation $z_{ij} = z_i-z_j$. We introduce additional commuting even formal variables $\ze_{ij}=\ze_{ji}$, which are
thought of as $\log |z_{ij}|$.

In order to define the operator product expansion of logarithmic fields, we need the following definitions motivated by the geometric and logarithmic series expansions in complex analysis.
Let
\begin{equation}\label{iota}
\iota_{z_1,z_2}z_{12}^{-1}= D_{z_1} \vartheta_{12} = \sum_{j\geq 0}z_{1}^{-j-1}z_{2}^{j}
\end{equation}
be the series expansion of $z_{12}^{-1} = 1/(z_1-z_2)$ in the domain $|z_1|>|z_2|$, 
and similarly
\begin{equation}\label{iota2}
\iota_{z_1,z_2}\ze_{12} = \vartheta_{12} = \ze_{1}-\sum_{j\geq 1}\frac{1}{j}z_{1}^{-j}z_{2}^{j}
\end{equation}
be the expansion of $\log z_{12}$ in the same domain
(cf.\ \eqref{logf7} and \reref{r-ze}).
The map $\iota_{z_1,z_2}$ can be extended uniquely to a homomorphism of associative algebras
\begin{equation*}
\iota_{z_1,z_2} \colon V[\![z_{1},z_{2}]\!][z_{1}^{-1}, z_{2}^{-1}, z_{12}^{-1},  \ze_{1}, \ze_{2}, \ze_{12}]\rightarrow V(\!(z_1)\!)(\!(z_2)\!)[\ze_1,\ze_2] \,,
\end{equation*}
by defining it on the generators $z_{12}^{-1}$ and $\ze_{12}$ as above and as the identity on all others:
\begin{equation}\label{iota3}
\iota_{z_1,z_2}(z_i^{\pm1}) = z_i^{\pm1} \,, \qquad \iota_{z_1,z_2}(\ze_i) = \ze_i \,, \qquad i=1,2 \,.
\end{equation}
Analogously, we define the map 
\begin{equation*}
\iota_{z_2,z_1} \colon V[\![z_{1},z_{2}]\!][z_{1}^{-1}, z_{2}^{-1}, z_{12}^{-1},  \ze_{1}, \ze_{2}, \ze_{12}]\rightarrow V(\!(z_2)\!)(\!(z_1)\!)[\ze_1,\ze_2]\,,
\end{equation*}
which acts as the expansion in the domain $|z_2|>|z_1|$. Explicitly, it is given by (cf.\ \eqref{logf8}):
\begin{align}
\label{iota4}
\iota_{z_2,z_1}\ze_{12} &=\vartheta_{21} =\ze_{2}-\sum_{j\geq 1}\frac{1}{j} z_{1}^{j} z_{2}^{-j} \,, \\
\label{iota5}
\iota_{z_2,z_1}z_{12}^{-1} &= D_{z_1} \vartheta_{21} = -D_{z_2} \vartheta_{21} = -\sum_{j\geq 0} z_{1}^{j} z_{2}^{-j-1} \, , \\
\label{iota6}
\iota_{z_2,z_1}(z_i^{\pm1}) &= z_i^{\pm1} \,, \qquad \iota_{z_2,z_1}(\ze_i) = \ze_i \,, \qquad i=1,2 \,.
\end{align}
Observe that here $\iota_{z_2,z_1}\ze_{12}=\iota_{z_2,z_1}\ze_{21}$ is the expansion of $\log z_{21}$ for $|z_2|>|z_1|$. 
Finally, define the map
\begin{equation*}
\begin{split}
\iota_{z_2,z_{12}} \colon V[\![z_{1},z_{2}]\!][z_{1}^{-1}, z_{2}^{-1}, z_{12}^{-1},  \ze_{1}, \ze_{2}, \ze_{12}]\rightarrow V(\!(z_2)\!)(\!(z_{12})\!)[\ze_{2},\ze_{12}]
\end{split}
\end{equation*}
as the expansion in the domain $|z_2|>|z_{12}|$, where we use the relation $z_1=z_2+z_{12}$ and we think of $\ze_1$ as $\log z_1 = \log(z_2+z_{12})$.
Namely,
\begin{align}
\label{iota7}
\iota_{z_2,z_{12}}\ze_{1} &=\ze_{2}-\sum_{j\geq 1 }\frac{(-1)^{j}}{j} z_{12}^{j} z_{2}^{-j} \,,\\
\label{iota8}
\iota_{z_2,z_{12}}z^{-1}_{1} &=D_{z_1} \iota_{z_2,z_{12}}\ze_{1}  = \sum_{j\geq 0}(-1)^{j} z_{12}^{j} z_{2}^{-j-1}\, , 
\end{align}
and $\iota_{z_2,z_{12}}$ is the identity on the remaining generators $z_1$, $z_{2}$, $z_2^{-1}$, $z_{12}^{-1}$, $\ze_2$, $\ze_{12}$.
Note that, by construction, all $\iota$-maps are homomorphisms of associative algebras and commute with the actions of the derivatives
$D_{z_1}$, $D_{z_2}$ and with the multiplication by elements of $\CC[z_1,z_2]$.

With the above notation, we can formulate the following theorem, which is similar to corresponding results for ordinary vertex algebras and generalized vertex algebras
(see \cite{DL,FB,LL,BK2}).

\begin{theorem}\label{th2.16}
Let\/ $V$ be a superspace and\/ $\V\subset\LF(V)$ be a space of local logarithmic fields for some braiding map\/ $\NN$.
Then, for every fixed $a,b\in \V$ and\/ $v\in V$, there exists a formal power series 
\begin{equation}\label{iota9}
\Y_{a,b,v}(z_1,z_2)\in V[\![z_{1},z_{2}]\!][z_{1}^{-1}, z_{2}^{-1}, z_{12}^{-1},  \ze_{1}, \ze_{2}, \ze_{12}] \,,
\end{equation}
with the following properties$:$
\begin{align}
\label{iota10}
a(z_1)b(z_2)v&=\iota_{z_1,z_2}\Y_{a,b,v}(z_1,z_2)\,, \\
b(z_2)a(z_1)v&=(-1)^{p(a)p(b)} \iota_{z_2,z_1}\Y_{a,b,v}(z_1,z_2)\,,
\label{iota11}
\end{align}
and
\begin{equation}\label{iota12}
\sum_{i\in\ZZ_{+}, \, n\in\ZZ }\frac{(-1)^{i}}{i!}\, \ze_{12}^{i}z_{12}^{-n-1}\bigl(\NN^{i }a_{(n+\NN)}b\bigr)(z_{2})v=\iota_{z_2,z_{12}}\Y_{a,b,v}(z_1,z_2)\, .
\end{equation}
Furthermore, there exists\/ $N\in\ZZ_+$, depending only on\/ $a$ and\/ $b$, such that
\begin{equation}\label{iota13}
z_{12}^N \Y_{a,b,v}(z_1,z_2)\in V[\![z_{1},z_{2}]\!][z_{1}^{-1}, z_{2}^{-1}, \ze_{1}, \ze_{2}, \ze_{12}] \,, \qquad v\in V\,.
\end{equation}
\end{theorem}
\begin{proof}
Let $N$ be a positive integer large enough so that \eqref{logf15} is satisfied for all $i\in \ZZ_{+}$. Then we define
\begin{equation}\label{yabv}
\begin{split}
\Y_{a,b,v}(z_1,z_2)
&= z_{12}^{-N} \mu\bigl(z_{12}^{N}e^{\vartheta_{12} \NN} e^{-\ze_{12}\NN} a(z_1)\otimes b(z_2)\bigr)v \\
&=\sum_{i\geq 0}\frac{(-1)^{i}}{i!}\ze_{12}^{i}z_{12}^{-N} \mu\bigl(z_{12}^{N}e^{\vartheta_{12} \NN}\NN^{i}a(z_1)\otimes b(z_2)\bigr)v\, .
\end{split}
\end{equation}
Notice that the above sum is in fact finite, because $\NN$ is nilpotent on $a\otimes b$. Moreover, from locality \eqref{logf15}, we have
\[ \mu\bigl(z_{12}^{N}e^{\vartheta_{12} \NN}\NN^{i}a(z_1)\otimes b(z_2)\bigr)v\in V[\![z_1,z_2]\!][z_1^{-1},z_2^{-1}, \ze_1, \ze_2] \,,\]
which proves \eqref{iota13}.
Then using \eqref{iota} and \eqref{iota2} we get 
\begin{equation*}%\label{m}
\begin{split}
a(z_1)b(z_2)v
&=z_{12}^{-N} \mu\bigl(z_{12}^{N} e^{\vartheta_{12}\NN}e^{-\vartheta_{12}\NN}a(z_1)\otimes b(z_2)\bigr)v\\
&=\iota_{z_1,z_2} z_{12}^{-N} \mu\bigl(z_{12}^{N} e^{\vartheta_{12}\NN}e^{-\ze_{12}\NN}a(z_1)\otimes b(z_2)\bigr)v \\
&=\iota_{z_1,z_2} \Y_{a,b,v}(z_1,z_2)\,.
%& =\sum_{i\geq 0}\iota_{z_1,z_2}\frac{(-1)^{i}}{i!}\ze_{12}^{i}z_{12}^{-N}\mu\bigl(z_{12}^{N}e^{\vartheta_{12}\NN}\NN^{i}a(z_1)\otimes b(z_2)\bigr)v\, .
\end{split}
\end{equation*}
Similarly, from \eqref{logf15}, \eqref{iota4} and \eqref{iota5} we have
\begin{align*}
b(z_2)a(z_1)v
&= z_{12}^{-N} \mu\bigl(z_{12}^{N} e^{\vartheta_{21}\NN}e^{-\vartheta_{21}\NN} b(z_2)\otimes a(z_1)\bigr)v\\
&= \iota_{z_2,z_1} z_{12}^{-N} \mu\bigl(z_{12}^{N} e^{\vartheta_{21}\NN}e^{-\ze_{12}\NN}b(z_2)\otimes a(z_1)\bigr)v \\
&=(-1)^{p(a)p(b)} \iota_{z_2,z_1} z_{12}^{-N} \mu\bigl(z_{12}^{N} e^{\vartheta_{12}\NN}e^{-\ze_{12}\NN}a(z_1)\otimes b(z_2)\bigr)v \\
&=(-1)^{p(a)p(b)} \iota_{z_2,z_1} \Y_{a,b,v}(z_1,z_2)\,.
%& =\sum_{i\geq 0}\iota_{z_1,z_2}\frac{(-1)^{i}}{i!}\ze_{12}^{i}z_{12}^{-N}\mu\bigl(z_{12}^{N}e^{\vartheta_{12}\NN}\NN^{i}a(z_1)\otimes b(z_2)\bigr)v\, .
\end{align*}
Finally, from \eqref{logf17-al}, \eqref{logf17} and Taylor's formula, we have for every $i\in\ZZ_+$:
\begin{align*}
\sum_{n\in\ZZ} &z_{12}^{-n-1}(\NN^{i }a_{(n+\NN)}b)(z_{2})v \\
&= \sum_{n\in\ZZ} z_{12}^{-n-1} D_{\z_1}^{(N-1-n)}\mu \Bigl(\z_{12}^Ne^{\vartheta_{12}\NN} \NN^{i } a(z_1) \otimes b(z_2) \Bigr)v\Big|_{\z_1=\z_2} \\
&= z_{12}^{-N} \iota_{z_2,z_{12}} \mu \bigl(\z_{12}^Ne^{\vartheta_{12}\NN} \NN^{i } a(z_1) \otimes b(z_2) \bigr) \,.
\end{align*}
Multiplying this by $(-1)^{i} \ze_{12}^{i} / i!$ and summing over $i\in\ZZ_+$ proves \eqref{iota12}.
%This completes the proof of the theorem.
\end{proof}

\begin{remark}\label{re2.17}
In the setting of \thref{th2.16} and its proof, consider the linear map
\begin{equation*}
\Y\colon\V\otimes\V \to \Hom\bigl(V, V[\![z_{1},z_{2}]\!][z_{1}^{-1}, z_{2}^{-1}, z_{12}^{-1},  \ze_{1}, \ze_{2}, \ze_{12}]\bigr) \,,
\end{equation*}
defined by
\begin{equation*}
\Y(a\otimes b)v =  \Y_{a,b,v}(z_1,z_2)\,.
\end{equation*}
Then $\Y$ has the property that
\begin{equation*}
\Y\bigl(e^{\ze_{12}\NN}(a\otimes b)\bigr) = z_{12}^{-N} \mu\bigl(z_{12}^{N}e^{\vartheta_{12} \NN} a(z_1)\otimes b(z_2)\bigr)
\end{equation*}
is independent of $\ze_{12}$. Conversely, this property together with \eqref{iota10}, \eqref{iota11} and \eqref{iota13}
implies the locality of $\V$.
\end{remark}

Let again $\V\subset\LF(V)$ be a local space for some braiding map $\N$.
For $a,b\in\V$, their \emph{operator product expansion} (OPE) is defined as the formal series
\begin{equation*}
\sum_{ i\in\ZZ_{+}, \, n\in\ZZ } \frac{(-1)^{i}}{i!}\, \ze_{12}^{i}z_{12}^{-n-1} (\NN^{i }a_{(n+\NN)}b)(z_{2})\, .
\end{equation*}
Note that this is not equal to $a(z_1)b(z_2)$ as it is an expansion in a different domain; compare \eqref{iota10} and \eqref{iota12}.
%The \emph{singular part} of the OPE %(sometimes called just OPE for short) 
%is the sum over only $n\ge0$, and is denoted as
The OPE is denoted as
\begin{equation*}%\label{ope1}
a(z_1)b(z_2)\sim \sum_{ i\in\ZZ_{+}, \, n\in\ZZ } \frac{(-1)^{i}}{i!}\, \ze_{12}^{i}z_{12}^{-n-1} (\NN^{i }a_{(n+\NN)}b)(z_{2})\, .
\end{equation*}
In the physics literature, it is customary to write the OPE as follows:
\begin{equation*}%\label{ope2}
a(z_1)b(z_2)\sim \sum_{ i\in\ZZ_{+}, \, n\in\ZZ } \frac{\log^{i}(z_1-z_2)}{(z_{1}-z_{2})^{n+1}}c_{i,n}(z_{2})\, , 
\end{equation*}
where 
\begin{equation*}%\label{ope3}
c_{i,n}(z)=\frac{(-1)^i}{i!} (\NN^{i }a_{(n+\NN)}b)(z) \in \LF(V) \,,
\end{equation*}
or at $z_2=0$ in the form
\begin{equation*}%\label{ope4}
a(z)b(0)\sim \sum_{ i\in\ZZ_{+}, \, n\in\ZZ } \frac{\log ^iz}{z^{n+1}}c_{i,n}(0) \,;
\end{equation*}
see also \eqref{ope5}, \eqref{ope6} below.
% (see \seref{e4.7} below).

\subsection{Translation covariant logarithmic fields}\label{s2.3}
In this subsection, as with ordinary vertex algebras, we will assume that our superspace $V$ is equipped with an even vector $\vac\in V_{\bar 0}$ (called the \emph{vacuum vector})
and an even linear operator $T\in\End(V)_{\bar 0}$ (called the \emph{translation operator}) such that $T\vac=0$. 
As before, let $\NN$ be a braiding map on $\LF(V)$ satisfying \eqref{1.2} and \eqref{logfdn}.

\begin{definition}\label{dtrcov}
A logarithmic field $a\in\LF(V)$ is called \emph{translation covariant} if 
 \begin{equation}\label{logf20}
 [T,a(z)]=D_{z}a(z)\, .
 \end{equation}  
 We denote by $\LFT(V)$ the space of translation covariant logarithmic fields.  
 \end{definition}
 
Note that $\LFT(V)$ is a $D_{z}$-invariant subspace of $\LF(V)$. To ensure it is invariant under the operators $\Phi_i$ and $\Psi_i$ from \eqref{logf14},
we will impose the condition that
 \begin{equation}\label{logf21}
 \bigl[\ad(T)\otimes I, \NN\bigr] = \bigl[I\otimes \ad(T), \NN\bigr]=0\, . 
 \end{equation}
Notice that the second equality follows from the first and the symmetry of $\NN$.
As a consequence of \eqref{logfdn} and \eqref{logf21}, $\NN$ restricts to a braiding map on $\LFT(V)$.

\begin{lemma}\label{l2.16} 
Let\/ $\NN$ be a braiding map on\/ $\LF(V)$ satisfying\/ \eqref{logfdn} and\/ \eqref{logf21}. Then for every\/ local subspace\/ $\V\subset\LFT(V)$ and any\/ $a,b\in\V$,
we have\/ $a_{(n+\NN)}b \in\LFT(V)$ for all\/ $n\in\ZZ$.
\end{lemma}
\begin{proof} 
Since $\ad(T)$ is a derivation of the product $\mu$ in $\End(V)$, we have:
\begin{align*}
\ad(T) \, & \mu\bigl( e^{\vartheta_{12} \NN} a(z_1)\otimes b(z_2) \bigr) \\
&=\mu \bigl(\ad(T\otimes I+I\otimes T) e^{\vartheta_{12} \NN} a(z_1)\otimes b(z_2)\bigr)\\
&=(D_{z_1}+D_{z_2})\mu\bigl( e^{\vartheta_{12} \NN} a(z_1)\otimes b(z_2)\bigr),
\end{align*}
where in the last identity we used the translation covariance of $a,b$ and $(D_{z_1}+D_{z_2})\vartheta_{12}=0$. Note that for an arbitrary function $f$ in two variables,
\[D_{z}f(z,z)=(D_{z_1}+D_{z_2})f(z_1,z_2)\big|_{z_1=z_2=z}\, .\]
Applying Definition \ref{d2.10}, we thus obtain  
\begin{equation*}
\ad(T) (a(z)_{(n+\NN)}b(z))=D_{z}(a(z)_{(n+\NN)}b(z))\,,
\end{equation*}
as claimed.
\end{proof}

\begin{lemma}\label{ltrcvac} 
For every translation covariant field\/ $a(z)\in\LFT(V)$, we have\/ $a(z)\vac\in V[\![z]\!]$.
%\begin{equation}\label{logf22}
%a(z)\vac\in V[\![z]\!]\, .
%\end{equation}
\end{lemma}
\begin{proof} 
From \eqref{logf20} and $T\vac=0$, we derive $D_{z}a(z)\vac = T a(z)\vac$. Let us write
\begin{equation*}
a(z)\vac = \sum_{n\in\ZZ} a_n(\ze) \z^{-n-1} \,, \qquad a_n(\ze) \in V[\ze] \,,
\end{equation*}
where $a_n(\ze) = 0$ for $n\gg0$. Then we obtain the equations
\begin{equation*}
(\partial_\ze-n) a_{n-1}(\ze) = T a_n(\ze) \,, \qquad n\in\ZZ \,.
\end{equation*}
Pick the largest $N\in\ZZ$ such that $a_{N-1}(\ze) \ne 0$ and $a_n(\ze) = 0$ for $n\ge N$. 
Hence,
\begin{equation*}
(\partial_\ze-N) a_{N-1}(\ze) = 0 \,,
\end{equation*}
which is impossible for a polynomial, unless $N=0$ and $a_{-1}(\ze) \in V$ is a constant.
Thus, $a_n(\ze)=0$ for $n\ge0$.
Next, from
\begin{equation*}
(\partial_\ze+1) a_{-2}(\ze) = T a_{-1}(\ze) \in V \,,
\end{equation*}
we see that $a_{-2}(\ze) \in V$. In the same way, by induction, we get $a_n(\ze) \in V$ for all $n<0$.
\end{proof}

As a consequence of \leref{ltrcvac}, we get an even linear map from $\LFT(V)$ to $V$ defined by
\begin{equation}\label{logf22b}
\Theta\colon \LFT(V) \to V \,, \qquad
\Theta(a) = a(z)\vac\big|_{z=0} \,.
\end{equation}
Then from the equation $D_{z}a(z)\vac = T a(z)\vac$, we obtain
\begin{equation}\label{logf22c}
a(z)\vac = e^{zT} \Theta(a) \,,\qquad a\in\LFT(V) \,,
\end{equation}
like for ordinary vertex algebras.

\begin{lemma}\label{l2.18} 
Let\/ $\NN$ be a braiding map on\/ $\LF(V)$ satisfying\/ \eqref{logfdn} and\/ \eqref{logf21}. 
Then for every\/ local subspace\/ $\V\subset\LFT(V)$ and\/ $a,b\in\V$, we have
\begin{equation}\label{aThetab}
a(z) \Theta(b) = \sum_{i\in \ZZ_{+}, \, n\in \ZZ}\frac{(-1)^{i}}{i!}\ze^{i} z^{-n-1} \Theta\bigl( \NN^{i}a_{(n+\NN)}b \bigr).
\end{equation}
Furthermore, if\/ $\Theta(a)=0$, then\/ $a(z) \Theta(b) = 0$ for all\/ $b\in\V$.
\end{lemma}
\begin{proof} 
We will use \thref{th2.16}. 
Let again $N$ be a positive integer large enough so that \eqref{logf15} is satisfied for all $i\in \ZZ_{+}$. 
If we expand $z_{12}^{\NN} \NN^i a(z_1)b(z_2)$ as in the right-hand side of \eqref{Nexp}, it follows from \leref{ltrcvac} that
\begin{equation*}
\bigl( z_{12}^{N+\NN} \NN^i a(z_1)b(z_2) \bigr)\vac \in V[\![z_{1},z_{2}]\!][z_{1}^{-1}, \ze_{1}]
\end{equation*}
is independent of $\ze_2$ and has only non-negative powers of $z_2$. Due to locality \eqref{logf15}, we conclude that
\begin{equation*}
%\mu\bigl(z_{12}^{N}e^{\vartheta_{12} \NN}\NN^{i}a(z_1)\otimes b(z_2)\bigr) \vac = 
c_i(z_1,z_2) := \bigl( z_{12}^{N+\NN} \NN^i a(z_1)b(z_2) \bigr)\vac \in V[\![z_{1},z_{2}]\!] \,.
\end{equation*}
Then from \eqref{yabv} we find
\begin{equation*}
\Y_{a,b,\vac}(z_1,z_2)
=\sum_{i\geq 0}\frac{(-1)^{i}}{i!}\ze_{12}^{i}z_{12}^{-N} c_i(z_1,z_2) \, .
\end{equation*}
By \eqref{iota10}, we have
\begin{equation*}
a(z_1) \Theta(b) = \iota_{z_1,z_2} \Y_{a,b,\vac}(z_1,z_2) \big|_{z_2=0}
=\sum_{i\geq 0}\frac{(-1)^{i}}{i!}\ze_1^i z_{1}^{-N} c_i(z_1,0) \,,
\end{equation*}
using that $\iota_{z_1,z_2} \ze_{12}|_{z_2=0} = \ze_1$ and $\iota_{z_1,z_2} z_{12}^{-N}|_{z_2=0} = z_1^{-N}$
by \eqref{iota}, \eqref{iota2}.
By definition, the coefficient in front of $z_1^{N-n-1}$ in the Taylor expansion of $c_i(z_1,0)$ is
\begin{align*}
D_{z_1}^{(N-n-1)} c_i(z_1,0) \big|_{z_1=0}
&= \bigl( D_{z_1}^{(N-n-1)} c_i(z_1,z_2) \big|_{z_1=z_2=z} \bigr) \big|_{z=0} \\
&= \Theta\bigl( \NN^{i}a_{(n+\NN)}b \bigr),
\end{align*}
which proves \eqref{aThetab}.

For the second claim in the lemma, suppose that $\Theta(a)=0$. Then $a(z)\vac=0$ by \eqref{logf22c}. 
Applying \eqref{iota11} and using \eqref{iota4}, we get
\begin{align*}
0 &= (-1)^{p(a)p(b)} z_{12}^{N} b(z_2) a(z_1)\vac \\
&= \iota_{z_2,z_1} z_{12}^{N} \Y_{a,b,\vac}(z_1,z_2) 
=\sum_{i\geq 0}\frac{(-1)^{i}}{i!} \vartheta_{21}^i c_i(z_1,z_2) \,.
\end{align*}
This sum is finite because $\NN$ is nilpotent on $a\otimes b$.
We want to show that all $c_i(z_1,z_2) = 0$. Assume, on the contrary, that some $c_r \ne 0$
and $c_i=0$ for all $i>r$. Then the highest power of $\ze_2$ in the above sum is $\ze_2^r$
and its coefficient is non-zero, a contradiction. Therefore, all $c_i=0$, and as a consequence
$a(z) \Theta(b) = 0$.
\end{proof}

For any field $a\in \LF(V)$, we have the sequence of linear maps $a_{(n+\N)}\in\End(V)$, $n\in\ZZ$, given by
\begin{equation}\label{logf24}
a(z)\big|_{\ze=0}=\sum_{n\in \ZZ} z^{-n-1} a_{(n+\N)} \,.
 \end{equation}
We remark that the notation $a_{(n+\NN)}b\in \LF(V)$ means the $(n+\NN)$-th product of two logarithmic fields $a,b$ (Definition \ref{d2.10}), while the notation $a_{(n+\N)}v\in V$ 
means the vector obtained from the evaluation of the linear operator $a_{(n+\N)}$ on $v\in V$. In the case of logarithmic vertex algebras, these two notions will be related in the next section.
We can already see this relationship in the following corollary of \eqref{aThetab}.

\begin{corollary}\label{l2.18b} 
In the setting of \leref{l2.18}, we have
\begin{equation}\label{Thanb}
\Theta(a_{(n+\NN)}b) = a_{(n+\N)} \Theta(b) \,, \qquad a,b\in\V \,, \;\; n\in\ZZ \,.
\end{equation}
\end{corollary}

The following version of the Kac Existence Theorem (see \cite{K,DK}) will be useful to generate examples of logarithmic vertex algebras. 
In it we use the definition of the space $\overline{\V}$ from Lemma \ref{l2.12}. In addition, $\NN$ will be a braiding map on $\LF(V)$ satisfying 
\eqref{1.2}, \eqref{logfdn}, and \eqref{logf18b} for any $\NN$-local subspace $\V\subset\LF(V)$. For simplicity, we will assume that
$\NN=(\ad\otimes\ad)(\N)$ is constructed as in \coref{cor3.10-2}, where $\N$ is a braiding map on $V$ as in \eqref{logf-n}.
Then conditions \eqref{logf21} hold provided that $[T,\phi_i]=[T,\psi_i]=0$ for all $i$.
As we will see in the next section, for a logarithmic vertex algebra, the map $\NN$ always has this form.

\begin{theorem}[Existence Theorem]\label{l2.19} 
Let\/ $V$ be a superspace equipped with an even vector\/ $\vac \in V_{\bar 0}$ and an even linear operator\/ $T\in \End(V)_{\bar 0}$ such that\/ $T\vac=0$,
and\/ $\N$ be a braiding map on\/ $V$ as in\/ \eqref{logf-n} such that
\begin{equation}\label{Tphipsi}
[T,\phi_i]=[T,\psi_i]=0\,, \qquad i=1,\dots,L \,. 
\end{equation}
Set\/ $\NN=(\ad\otimes\ad)(\N)$ and extend
it to a braiding map on\/ $\LF(V)$ by acting only on the coefficients in front of powers of\/ $z^{\pm1}$ and\/ $\ze$ in the logarithmic fields. 
Fix an\/ $\NN$-local subspace
\begin{equation*}
\V=\Span\bigl\{I,a^{i}(z) \,\big|\, i\in J \bigr\} \subset \LFT(V)
\end{equation*}
of translation covariant logarithmic fields, where\/ $J$ is an index set.
Suppose that\/ $\V$ is \emph{complete} in the sense that
\begin{equation*}
V=\Span\bigl\{\vac, \, a^{i_{1}}_{(n_1+\N)}\cdots a^{i_{k}}_{(n_k+\N)}\vac \,\big|\, k\ge1, \, i_{1},\dots, i_{k}\in J, \, n_{1},\dots ,n_{k}\in\ZZ \bigr\} \,.
\end{equation*}
Then\/ $\overline{\V} \subset \LFT(V)$ is an\/ $\NN$-local space containing\/ $\V$, closed under all\/ $(n+\NN)$-th products,
and for every\/ $v \in V$ there exists a unique logarithmic field\/ $Y (v, z)\in \overline{\V}$ such that\/ $Y (v, z)\vac \big|_{z=0}=v$.
Furthermore, $\overline{\V}$ is the maximal\/ $\NN$-local subspace of\/ $\LFT(V)$ containing\/ $\V$.
\end{theorem}
\begin{proof} 
As discussed before the theorem, $\NN$ is a braiding map on $\LF(V)$ satisfying \eqref{1.2}, \eqref{logfdn}, and \eqref{logf18b} for any $\NN$-local subspace $\V\subset\LF(V)$.
Hence, by \leref{l2.12}, we can construct the minimal local space $\overline{\V}$ containing $\V$ and closed under all $(n+\NN)$-th products, $n\in\ZZ$.
Moreover, \eqref{Tphipsi} implies $[\ad(T),\ad(\phi_i)]=[\ad(T),\ad(\psi_i)]=0$ for all $i$, so the map $\NN$ satisfies \eqref{logf21}.
Then \leref{l2.16} ensures that $\overline{\V} \subset \LFT(V)$.

The linear map $\Theta$, given by \eqref{logf22b}, satisfies $\Theta(I)=\vac$ and $\Theta(a^i)=a^i_{(-1+\N)}\vac$; see \eqref{logf24}.
More generally, by \eqref{Thanb}, we have
\begin{align*}
\Theta\bigl(a^{i_{1}}_{(n_1+\NN)} I\bigr) &= a^{i_{1}}_{(n_1+\N)} \vac \,, \\
\Theta\bigl(a^{i_{1}}_{(n_1+\NN)} (a^{i_{2}}_{(n_2+\NN)} I)\bigr) &= a^{i_{1}}_{(n_1+\N)} a^{i_{2}}_{(n_2+\N)} \vac \,, \quad\text{etc.}
\end{align*}  
By induction, we get
\begin{equation*}
\Theta\bigl(a^{i_{1}}_{(n_1+\NN)} (a^{i_{2}}_{(n_2+\NN)}(\cdots (a^{i_{k}}_{(n_k+\NN)} I)\cdots) \bigr)
= a^{i_{1}}_{(n_1+\N)}\cdots a^{i_{k}}_{(n_k+\N)}\vac \,.
\end{equation*}  
Therefore, the map $\Theta$ is surjective, due to the completeness of $\V$.

To prove that $\Theta$ is injective, suppose that $a\in \overline{\V}$ is such that $\Theta(a)=0$. Applying \leref{l2.18} for the space $\overline{\V}$, we get $a(z)\Theta(b)=0$ for all $b\in \overline{\V}$,
which implies $a(z)=0$ because $\Theta$ is surjective. Hence $\Theta$ is bijective, and we define $Y\colon V\to\overline{\V}$ as $\Theta^{-1}$.

Finally, to show that $\overline{\V}$ is maximal, consider any local subspace $\U\subset\LFT(V)$ such that $\V\subset\U$. Then applying the above results for $\U$ instead of $\V$, we obtain the subspace
$\overline{\U}$ and the bijection $\Theta\colon\overline{\U} \to V$. However, $\overline{\V} \subset \overline{\U}$ and the restriction of $\Theta$ to $\overline{\V}$ is also bijective. 
Hence, $\overline{\V} = \overline{\U}$, which implies $\U\subset\overline{\V}$.
\end{proof}

Notice that the space $\overline{\V}$ constructed in the above theorem is both minimal and maximal; thus is the unique space with these properties.
This uniqueness is related to Goddard's Uniqueness Theorem for vertex algebras (see \cite{K,DK}).
We will have another version of it for logarithmic vertex algebras in the next section.

\section{Logarithmic vertex algebras}\label{s3}
In this section, we introduce two equivalent definitions of the notion of a logarithmic vertex algebra.  We define conformal logarithmic vertex algebras and express the braiding map $\N$ in terms of the Virasoro $L_0$-operator. 
We derive an Existence Theorem, Uniqueness Theorem, $(n+\N)$-th product identity, OPEs, skew-symmetry, associativity, and a version of the Borcherds identity for logarithmic vertex algebras.

\subsection{$\NN$-logarithmic vertex algebras}\label{s3.1}
Our first definition is given in terms of a braiding map $\NN$ on the space $\LF(V)$ of logarithmic fields, and it encodes the properties of the space $\overline\V$ from \thref{l2.19}.

\begin{definition}\label{d3.1}
An $\NN$-\emph{logarithmic vertex algebra} (abbreviated $\NN$-logVA) is a vector superspace $V$ (space of states), equipped with an even vector $\vac\in V_{\bar 0}$ (vacuum vector), an even endomorphism $T\in \End(V)_{\bar 0}$ (translation operator),  an even linear map (state-field correspondence)
\[Y_z\colon  V\to \LF(V)\,, \qquad a\mapsto Y_z(a)=Y(a,z)\,, \]
and a braiding map $\NN$ on the space $Y_{z}(V)$ (see \deref{lm}),
which are subject to the following axioms:

\medskip
(\emph{vacuum})\; $Y(\vac,z)=I$, \; $Y(a,z)\vac\in V[\![z]\!]$, \; $Y(a,z)\vac\big|_{z=0} = a$, \; $T\vac=0$.

\medskip
(\emph{covariance})\; $[\ad(T)\otimes I, \NN]=[D_z\otimes I, \NN]=0$, \;  $[T,Y(a,z)] = D_z Y(a,z)$.

\medskip
(\emph{locality})\;  
%For $a,b\in V$ we have that
%\begin{equation}\label{locality}
%\begin{split}
%z_{12}^{N+\NN}&Y(a,\z_{1})Y(b,\z_2)c\\
%&\quad =(-1)^{p(a)p(b)}\z_{21}^{N+\NN}Y(b,\z_{2})Y(a,\z_1)c \, ,
%\end{split}
%\end{equation}
%for some integer $N\in 2\ZZ_+$ i.e. $Y_{z}(V)$ is a local space.
The space $Y_{z}(V)$ is $\NN$-local (see \deref{d2.5}).

\medskip
(\emph{hexagon})\; For all $n\in\ZZ$, we have
\begin{equation} \label{hexagon}
\NN (\hat{\mu}_{(n)}\otimes I) = (\hat{\mu}_{(n)}\otimes I) (\NN_{13}+\NN_{23})  \quad \text{on} \quad  Y_{z}(V)^{\otimes 3}\,,
\end{equation}
where $\hat{\mu}_{(n)}$ is the $(n+\NN)$-th product of logarithmic fields defined by \eqref{logf16} and \eqref{logf17-mu}.

\smallskip
We denote an $\NN$-logVA by $(V,\vac,T, Y,\NN)$ or just $V$ for short.
\end{definition}

We will construct examples of logarithmic vertex algebras in \seref{s4} below. The main tool for that is the following version of the Kac Existence Theorem (see \cite{K,DK}), which 
allows us to generate an $\NN$-logVA from an $\NN$-local space of logarithmic fields.
%is an immediate consequence of \thref{l2.19}.

\begin{theorem}[Existence Theorem]\label{t3.14} 
Under the assumptions of \thref{l2.19}, we have the\/ $\NN$-logVA\/ $(V,\vac,T,Y,\NN)$.
\end{theorem}
\begin{proof}
Follows immediately from \thref{l2.19}.
\end{proof}

Let $V$ be an $\NN$-logVA, and denote by $\V$ the local space $Y_{z}(V)$. Then, by \thref{l2.19}, we have $\V=\overline\V$.
This implies that $\V$ is closed under all $(n+\NN)$-th products ($n\in\ZZ$), and in particular under $D_z$ because $I\in\V$ (see \eqref{2.30}).
Hence we have the following:

\begin{remark}\label{rem3.10} 
Let $V$ be an $\NN$-logVA. Then, for every fixed $n\in\ZZ$, we have the structure $(Y_{z}(V), I, D_{z}, \hat{\mu}_{(n)},\NN)$, which is compatible in the sense of \deref{def1}.
\end{remark}

Moreover, again by \thref{l2.19}, the space $\V=Y_{z}(V)$ is maximal. As a consequence, we obtain the following version of
Goddard's Uniqueness Theorem (cf.\ \cite{K}).

\begin{theorem}[Uniqueness Theorem]\label{t3.7} 
Let\/ $V$ be an\/ $\NN$-logVA. If\/ $\U\subset\LFT(V)$ is an\/ $\NN$-local space of translation covariant logarithmic fields containing\/ $Y_{z}(V)$,
then\/ $\U=Y_{z}(V)$. In particular, for any\/ $a(z)\in\U$, we have\/ $a(z)=Y(v,z)$ where $v=a(z)\vac|_{z=0}$.
\end{theorem}

As in \cite{B}, we will use the notation (cf.\ \eqref{logf24}):
\begin{equation}\label{logva1} 
X(a,z):=Y(a,z)\big|_{\ze=0} = \sum_{n\in \ZZ} a_{(n+\N)} \, z^{-n-1} \,, \qquad a\in V\,.
\end{equation}
The linear operators $a_{(n+\N)} \in\End(V)$ are called the \emph{modes} of $a$.
For example, the vacuum and covariance axioms imply that
\begin{equation}\label{logva-12} 
a_{(-1+\N)} \vac = a \,, \qquad  a_{(-2+\N)} \vac = Ta \,,
\end{equation}
just as for ordinary vertex algebras. Notice that, by definition, for every fixed $a,b\in V$ we have
\begin{equation}\label{modes-0} 
a_{(n+\N)} b = 0 \,, \qquad  n\gg0 \,.
\end{equation}
%As an application of Theorem \ref{t3.7}, 
We have the following analog of the $n$-th product identity (cf.\ \cite{K}).

\begin{proposition}[$(n+\N)$-th product identity]\label{pro3.12} 
For any\/ $\NN$-logVA\/ $V$ and\/ $a, b\in V$, we have  
\begin{equation}\label{nprod-1}
Y(a,z)_{(n+\NN)}Y(b,z)=Y(a_{(n+\N)}b,z)\,,
\qquad %a, b\in V \,, \;\; 
n\in\ZZ \,.
\end{equation}
In particular,
%\begin{equation*}
$Y(Ta,\z)=D_{z}Y(a,z)$.
%\end{equation*}
\end{proposition}
\begin{proof}
As discussed before \reref{rem3.10}, the space $Y_z(V)$ is closed under all $(n+\NN)$-th products.
Hence, for fixed $a,b\in V$ and $n\in\ZZ$, we have $Y(a,z)_{(n+\NN)}Y(b,z)=Y(v,z)$ for some $v\in V$.
Recall that, by definition, 
\begin{equation*}
\Theta\bigl(Y(v,z)\bigr) = Y(v,z)\vac\big|_{z=0}=v \,.
\end{equation*}
But from \coref{l2.18b},
\begin{equation}\label{nprod-2}
\bigl(Y(a,z)_{(n+\NN)}Y(b,z)\bigr)\vac\big|_{z=0}=a_{(n+\N)}b \,;
\end{equation}
thus, $v=a_{(n+\N)}b$.
\end{proof}

More generally, using \eqref{aThetab}, the same proof as above gives the full expansion of a logarithmic field 
in an $\NN$-logVA:
\begin{equation}\label{logva1b}
 Y(a,z)b = \sum_{i\in \ZZ_{+}, \, n\in \ZZ}\frac{(-1)^{i}}{i!}\ze^{i} z^{-n-1} \Bigl( \NN^{i} Y(a,z_1)_{(n+\NN)} Y(b,z_1) \Bigr)\vac\Big|_{z_1=0} \,,
\end{equation}
for all $a,b \in V$, where we use the notation \eqref{logf17}.
In order to understand this formula better, let us rewrite the coefficients as
\begin{equation}\label{logva1c}
\begin{split}
\Bigl( \NN^{i} Y(a,z)_{(n+\NN)} Y(b,z) \Bigr)\vac\Big|_{z=0}
&= \bigl( \Theta \, \hat{\mu}_{(n)} \, \NN^{i} \bigr) \bigl( Y(a,z) \otimes Y(b,z) \bigr) \\
&= \bigl( \Theta \, \hat{\mu}_{(n)} \, \NN^{i} (Y_z \otimes Y_z)\bigr) (a\otimes b) \,.
\end{split}
\end{equation}
where, as before, we use the notation \eqref{logf17-mu} and \eqref{logf22b}.
Similarly to $\hat{\mu}_{(n)}$, we introduce the linear maps
\begin{equation}\label{eq3.7a}
\mu_{(n)} \colon V\otimes V\rightarrow V\,, \qquad {\mu}_{(n)}(a\otimes b)= a_{(n+\N)}b\, .
\end{equation}
Then \eqref{nprod-1} and \eqref{nprod-2} can be reformulated as the following identities on $V\otimes V$:
\begin{align}
\label{nprod-3}
\hat{\mu}_{(n)} \, (Y_z \otimes Y_z) &= Y_z \, \mu_{(n)} \,,\\
\label{nprod-4}
\Theta \, \hat{\mu}_{(n)} \, (Y_z \otimes Y_z) &= \mu_{(n)} \,.
\end{align}
These two identities are equivalent because the maps $\Theta$ and $Y_z$ are inverse to each other.
Using the isomorphisms $\Theta$ and $Y_z$, we will
transfer the action of $\NN$ on $Y_z(V) \otimes Y_z(V)$ to an action on $V\otimes V$.

\begin{definition}\label{l3.8b} 
For any $\NN$-logVA $V$, we have an even linear operator $\N$ on $V \otimes V$ defined by
\begin{equation}\label{eq3.11}
\N (a\otimes b)= \NN \bigl( Y(a,\z_{1})\otimes Y(b,\z_{2}) \bigr)(\vac\otimes \vac)\big|_{z_{1}=z_{2}=0}\, .
\end{equation}
\end{definition}

As in \eqref{nprod-3}, \eqref{nprod-4}, equation \eqref{eq3.11} is equivalent to the identity
\begin{equation}\label{nprod-5}
\N = (\Theta\otimes\Theta) \, \NN \, (Y_{z_1} \otimes Y_{z_2}) \,,
\end{equation}
%\begin{equation*}
%(Y_{z_1}\otimes Y_{z_2})(a\otimes b):=Y(a,z_1)\otimes Y(b,z_2)\,.
%\end{equation*}
and by applying $Y_{z_1} \otimes Y_{z_2}$ to it we get
\begin{equation}\label{nprod-6}
(Y_{z_1} \otimes Y_{z_2}) \, \N = \NN \, (Y_{z_1} \otimes Y_{z_2}) \,.
\end{equation}
We will study the properties of $\N$ in the next subsection. In particular, it is clear that $\N$ is locally nilpotent.
Similarly to \eqref{logf17}, let us introduce the notation
\begin{equation}\label{eq3.7}
\N^{i}a_{(n+\N)}b := \mu_{(n)} \bigl(\N^{i} (a\otimes b)\bigr) \,, \qquad a,b\in V \, .
\end{equation}
As an application of the above identities, we derive the explicit expansion of logarithmic fields in an $\NN$-logVA,
which explains our choice of notation for the modes (cf.\ \cite{B}).

\begin{proposition}\label{p-modes}
In any\/ $\NN$-logVA\/ $V$, we have for\/ $a,b\in V$,
\begin{equation}\label{nprod-7}
\begin{split}
Y(a,z)b &= \sum_{i\in \ZZ_{+}, \, n\in \ZZ}\frac{(-1)^{i}}{i!}\ze^{i} z^{-n-1} \N^{i} a_{(n+\N)}b \\
&= \sum_{n\in \ZZ} z^{-n-1-\N} a_{(n+\N)}b \,, %\qquad a,b\in V \,,
\end{split}
\end{equation} 
where, as before, $z^{-\N}:=e^{-\ze\N}$.
\end{proposition} 
\begin{proof} 
We apply \eqref{logva1b}, where the coefficients are given by \eqref{logva1c}:
\begin{equation*}
\bigl( \Theta \, \hat{\mu}_{(n)} \, \NN^{i} \, (Y_z \otimes Y_z)\bigr) (a\otimes b) \,.
\end{equation*}
Then using \eqref{nprod-6} and \eqref{nprod-4}, we find
\begin{equation*}
\Theta \, \hat{\mu}_{(n)} \, \NN^{i} \, (Y_z \otimes Y_z)
= \Theta \, \hat{\mu}_{(n)} \, (Y_z \otimes Y_z) \, \N^{i}
= \mu_{(n)} \, \N^{i} \,,
\end{equation*}
thus completing the proof.
\end{proof}

From Propositions \ref{pro3.12} and \ref{p-modes}, we derive an expression for translation covariance in terms of modes:
\begin{equation}\label{Tanb}
(Ta)_{(n+\N)}b=-(n+\N)a_{(n-1+\N)}b \,, \qquad a,b\in V\,.
\end{equation} 

In the future, it will also be convenient to use the linear maps $Y(z) \colon V\otimes V \to V(\!(z)\!)[\ze]$ 
and $X(z) \colon V\otimes V \to V(\!(z)\!)$ defined by
\begin{equation}\label{logva0} 
Y(z)(a\otimes b):=Y(a,z)b\, ,\qquad X(z)(a\otimes b):=X(a,z)b = Y(a,z)b \big|_{\ze=0} \, .
\end{equation}
Then equation \eqref{nprod-7} can be restated as follows (cf.\ \cite{B}):
\begin{equation}\label{logva3}
Y(z) = X(z) \, e^{-\ze \N} \,, \qquad X(z) = Y(z) \, e^{\ze \N} \,.
\end{equation} 
%\begin{align}
% &X(a,z)b=Y(z)e^{\ze \N}(a\otimes b)=\sum_{n\in\ZZ} \mu_{(n)}(a\otimes b)z^{-n-1}\in V(\!(z)\!)\, ,\label{logva3}\\
%&Y(a,z)b=X(z)e^{-\ze \N}(a\otimes b)=\sum_{n\in\ZZ} \mu_{(n)}(z^{-n-1-\N}a\otimes b)\in V(\!(z)\!)[\ze]\, .\label{logva4}
%\end{align}

From \eqref{logf16}, \eqref{nprod-6} and \prref{pro3.12}, we obtain another version of the $(n+\N)$-th product identity.

\begin{corollary}\label{cor-n-prod}
For any\/ $\NN$-logVA\/ $V$ and\/ $a,b,c\in V$, $n\in\ZZ$, we have
\begin{equation}\label{eq-n-prod}
\begin{split}
Y(&a_{(n+\N)} b,z_2) c \\
&= D_{\z_1}^{(N-1-n)} \Bigl(\z_{12}^N Y(z_1) (I \otimes  Y(z_2))
e^{\vartheta_{12}\N_{12}} (a \otimes b \otimes c) \Bigr)\Big|_{\z_1=\z_2} \,,
\end{split}
\end{equation}
for a sufficiently large\/ $N$ depending only on\/ $a$ and\/ $b$.
\end{corollary}

We finish this subsection with another consequence of \prref{pro3.12}, which is also its generalization.

\begin{corollary}\label{c3.13} 
For any\/ $\NN$-logVA\/ $V$, elements\/ $a, b\in V$ and\/ $n\in\ZZ$, $i\in\ZZ_{+}$, we have  
\begin{equation*}
\NN^{i}Y(a,z)_{(n+\NN)}Y(b,z)=Y(\N^{i}a_{(n+\N)}b,z)\, .
\end{equation*}
\end{corollary}
\begin{proof} 
Follows by applying both sides of \eqref{nprod-3} to $\N^i(a\otimes b)$ and using \eqref{nprod-6}.
\end{proof}

\subsection{Properties of the operator $\N$}\label{s3.2}
%We continue to use the notation from the previous subsection. 
Let again $V$ be an $\NN$-logVA, and define the linear operator $\N$ on $V\otimes V$ by \eqref{eq3.11}.
In this subsection, we study the properties of $\N$ and list them in \prref{l3.8cc} below.
But first, we need the following definition.

\begin{definition}\label{d3.4} 
A \emph{derivation} of an $\NN$-logVA $V$ is a linear operator $D\in \End(V)$ such that
\begin{equation}\label{der1}
D \bigl( Y(a,z)b \bigr) =Y(D a,z)b+(-1)^{p(a)p(D)}Y(a,z)D b\,, \qquad a,b\in V\, .
\end{equation}
We denote by $\Der(V)$ the set of all derivations of $V$.
\end{definition}

As usual, $\Der(V)$ is a Lie superalgebra under the supercommutator bracket. Condition \eqref{der1} is equivalent to $D$ being a derivation of the
$(n+\N)$-th product $\mu_{(n)}$ for every $n\in\ZZ$, and can be restated as
\begin{equation}\label{der2}
D \, \mu_{(n)} = \mu_{(n)} \, (D \otimes I + I \otimes D) \,.
\end{equation} 
Equivalently, using notation \eqref{logva0}, we have on $V\otimes V$:
\begin{equation}\label{der3}
D \, Y(z) = Y(z) \, (D \otimes I + I \otimes D) \,.
\end{equation} 
It is easy to derive from \eqref{logva-12} that every derivation $D\in\Der(V)$ satisfies
\begin{equation}\label{der4}
D\vac=0 \,, \qquad [D,T]=0 \,.
\end{equation} 
Notice that $T\in\Der(V)$, by the covariance axiom and \prref{pro3.12}.

\begin{proposition}\label{l3.8cc} 
Let\/ $V$ be an\/ $\NN$-logVA. Then the linear operator\/ $\N$ defined by\/ \eqref{eq3.11} has the following properties$:$
\begin{enumerate}
\medskip
\item\label{3.9-1}
$\N$ is a braiding map on\/ $V$. 

\medskip
\item\label{3.9-2}
$\N$ is locally nilpotent on\/ $V\otimes V$. 

\medskip
\item\label{3.9-3}
$(V, \vac, T, \mu_{(n)}, \N)$ is compatible for every\/ $n\in\ZZ$.

\medskip
\item\label{3.9-4}
$\N (Y(z)\otimes I) = (Y(z)\otimes I) (\N_{13}+\N_{23})\,  $.

\medskip
\item\label{3.9-5}
$\N\in \Der(V)\otimes \Der(V)\,$.

\medskip
\item\label{3.9-6}
$\NN=(\ad\otimes\ad)(\N)$\; on \;$Y_z(V) \otimes Y_z(V)$.
\end{enumerate}
\end{proposition}

\begin{proof}
\eqref{3.9-1} If we write $\NN$ as in \eqref{logf14}, then we obtain from \eqref{nprod-5} that $\N$ has the form \eqref{logf-n} with
\begin{equation*}
\phi_i = \Theta\,\Phi_i\, Y_z = \Theta\,\Phi_i\,\Theta^{-1} \,, \qquad
\psi_i = \Theta\,\Psi_i\, Y_z = \Theta\,\Psi_i\,\Theta^{-1} \,.
\end{equation*}
In particular, this implies $\N\in\End(V)\otimes\End(V)$. The other conditions for a braiding map (see \deref{lm}) follow immediately from the corresponding properties of $\NN$.

\eqref{3.9-2} follows from the local nilpotency of $\NN$, because from \eqref{nprod-5} we have
\begin{equation*}
\N^i = (\Theta\otimes\Theta)\,\NN^i\, (Y_z\otimes Y_z) \,.
\end{equation*}

\eqref{3.9-3} and \eqref{3.9-4} are equivalent to each other and to $\N\in \Der(V)\otimes\End(V)$, due to \eqref{der1}--\eqref{der4}. Since $\N$ is symmetric, they are also equivalent to \eqref{3.9-5}.
To prove them, we recall that $\NN$ satisfies \eqref{hexagon}, and using \eqref{nprod-3}, \eqref{nprod-5}, we derive
\begin{equation*}
\N ({\mu}_{(n)}\otimes I) = ({\mu}_{(n)}\otimes I) (\N_{13}+\N_{23}) \,, \qquad n\in\ZZ\,.
\end{equation*}

\eqref{3.9-6} We find for $v\in V$:
\begin{equation*}
\Phi_i\bigl(Y(v,z)\bigr) = (Y_z\,\phi_i\,\Theta) Y(v,z) = Y_z(\phi_i v) = Y(\phi_i v,z) \,.
\end{equation*}
But since $\phi_i \in\Der(V)$, we have $Y(\phi_i v,z) = [\phi_i,Y(v,z)]$. Therefore, $\Phi_i=\ad(\phi_i)$ on $Y_z(V)$. Similarly, we have $\Psi_i=\ad(\psi_i)$ on $Y_z(V)$.
\end{proof}

\begin{remark}
Recall that our braiding map $\NN$ is defined on all of $\LF(V) \otimes \LF(V)$. However, in principle, the identity $\NN=(\ad\otimes\ad)(\N)$ holds only for the restrictions to the fields from
the $\NN$-logVA $V$. In practice, we start with the operator $\N$ and then define $\NN$ everywhere using that identity.
\end{remark}
\begin{remark}
We can start from an arbitrary locally nilpotent braiding map $\N' = \sum \phi_i' \otimes \psi_i'$ on $V\otimes V$ such that $[T,\phi_i']=[T,\psi_i']=0$ for all $i$. Then, as discussed
in \thref{l2.19}, the map $\NN=(\ad\otimes\ad)(\N')$ will have all the desired properties. We can still define $\N$ by \eqref{eq3.11}, and we will have $\NN=(\ad\otimes\ad)(\N)$,
but in general $\N\ne\N'$. For example, we can add different multiples of the identity operator to $\phi_i'$ and $\psi_i'$.
\end{remark}

%The following lemma is useful to establish identities, a proof of this lemma is given in \cite[Lemma 4.1]{K}.
%\begin{lemma}\label{l3.6} Let U a vector space and let $R(z)\in \End (U)[\![z]\!]$. Then the differential equation $\d_{z}f(z)=R(z)f(z)$ has a unique solution of the form 
%\begin{equation*}
%f(z)=\sum_{n\in\mathbb{Z}_{+}}f_{n}z^{n},\qquad f_{n}\in U\, , 
%\end{equation*}
%with the given initial data $f_{0}$.
%\end{lemma}

\subsection{$\N$-logarithmic vertex algebras}\label{s3.2b}
Now we express the definition of a logarithmic vertex algebra using the endomorphism $\N$ instead of $\NN$. In Proposition \ref{d3.1b} below, we prove that this definition is equivalent to Definition \ref{d3.1}.

\begin{definition} \label{d3.11}
An \emph{$\N$-logarithmic vertex algebra} (abbreviated $\N$-logVA) 
is a vector superspace $V$ (space of states), equipped with an even vector $\vac\in V_{\bar 0}$ (vacuum vector), an even endomorphism $T\in \End(V)_{\bar 0}$ (translation operator),  an even linear map (state-field correspondence)
\[Y_z\colon  V\to \LF(V)\,, \qquad a\mapsto Y_z(a)=Y(a,z)\,, \]
and a braiding map $\N$ on $V$ (see \deref{lm}),
which are subject to the following axioms:

\medskip
(\emph{vacuum})\; $Y(\vac,z)=I$, \; $Y(a,z)\vac\in V[\![z]\!]$, \; $Y(a,z)\vac\big|_{z=0} = a$, \; $T\vac=0$.

\medskip
(\emph{covariance})\; %$[T\otimes I, \N]=0$, \;  
$[T,Y(a,z)] = D_z Y(a,z)$.

\medskip
(\emph{nilpotence})\; $\N$ is locally nilpotent on $V\otimes V$.

\medskip
(\emph{locality})\;  For every $a,b\in V$, there exists $N\in\ZZ_+$ such that for all $c\in V$, 
\begin{equation}\label{logva5}
\begin{split}
Y(z_1)& (I \otimes Y(z_2)) z_{12}^{N} e^{\vartheta_{12}  \N_{12}} (a\otimes b\otimes c )\\
&=(-1)^{p(a)p(b)} \, Y(z_2) (I \otimes Y(z_1)) z_{12}^{N} e^{\vartheta_{21}  \N_{12}} (b\otimes a\otimes c ),
\end{split}
\end{equation}
where we use the notation \eqref{logva0}.

\medskip
(\emph{hexagon})\; The following identity is satisfied on $V^{\otimes 3}$: 
\begin{equation} \label{hexagonb}
\N (Y(z)\otimes I) = (Y(z)\otimes I) (\N_{13}+\N_{23})\,,
\end{equation}
where we use the notation \eqref{logf-1} and \eqref{logva0}.

\medskip
We denote an $\N$-logVA by $(V,\vac,T, Y,\N)$ or $V$ for short.
When the braiding map $\N$ is fixed, we will call $V$ just a \emph{logarithmic vertex algebra} or a logVA.
\end{definition}

\begin{remark}\label{re3.7}
As discussed in \seref{s3.2}, the hexagon axiom together with the symmetry of $\N$ is equivalent to the condition $\N\in \Der(V)\otimes \Der(V)$.
In particular, this implies that $[T\otimes I,\N]=0$.
\end{remark}

\begin{proposition} \label{d3.1b} The notion of an\/ $\NN$-logVA is equivalent to that of an\/ $\N$-logVA.
\end{proposition}

\begin{proof} 
If we have an $\NN$-logVA $V$, then the operator $\N$, defined by \eqref{eq3.11}, satisfies all properties from \prref{l3.8cc}.
Locality in the form \eqref{logva5} follows from \eqref{nprod-6} and the $\NN$-locality \eqref{logf11} of the fields $a(z)=Y(a,z)$ and $b(z)=Y(b,z)$.
Hence, $V$ is an $\N$-logVA.

Conversely, suppose that $V$ is an $\N$-logVA. Define $\NN=(\ad\otimes\ad)(\N)$ as a linear operator on $\End(V) \otimes \End(V)$
and extend it to $\LF(V) \otimes \LF(V)$ by letting it act on the coefficients of the logarithmic fields.
Since $\N\in \Der(V)\otimes \Der(V)$ (by \reref{re3.7}), from the proof of \prref{l3.8cc}\eqref{3.9-6} we can derive \eqref{nprod-6}.
As before, this implies that $\NN$ is locally nilpotent on $Y_{z}(V)\otimes Y_{z}(V)$, because $\N$ is locally nilpotent on $V\otimes V$.
The remaining properties of $\NN$ and the $\NN$-locality again follow from \eqref{nprod-6}. Hence, $V$ is an $\NN$-logVA.
\end{proof}

%Due to the previous proposition, from now on we refer to $\NN$-logVA and $\N$-logVA just by logVA.
%We will provide examples of logarithmic vertex algebras in \seref{s4} below.

%\begin{remark} 
If $V$ is a logVA with $\N=0$, then from \eqref{logva3} we see that all fields $Y(a,z)$, $a\in V$, are independent of $\zeta$.   
%Additionally, the locality axiom becomes the usual locality for vertex algebras. 
Hence, a logVA with $\N=0$ is the same as an ordinary vertex algebra (see, e.g., \cite[Chapter 4]{K}).
%\end{remark}
More generally, for a logVA $V$, consider the subspace
\begin{equation}\label{V0}
V^0 = \{ a\in V \,|\, \N(a\otimes b)=0 \;\;\text{for all}\;\; b\in V\} \,.
\end{equation}
Then $V^0$ is an ordinary vertex algebra and is a subalgebra of $V$.

\begin{remark}\label{re3.8} The locality \eqref{logva5} can be equivalently expressed as follows:
\begin{equation}\label{3.8}
\begin{split}
z_{12}^{N} \Bigl( Y({z_1}) & (I \otimes Y({z_2}))(e^{(\vartheta_{12}-\eta)\N} (a\otimes b)\otimes c) \Bigr)\Big|_{\eta=\vartheta_{21}}\\
& = (-1)^{p(a)p(b)} \, z_{12}^{N} \, Y(b,z_2) Y(a,z_1) c \,,
\end{split}
\end{equation}
for $a,b,c\in V$ and some integer $N\geq 0$ that depends only on $a$ and $b$.
Note that we cannot evaluate $\eta=\vartheta_{21}$ inside the parentheses, since the product $\vartheta_{12} \vartheta_{21}$ and hence the operator $e^{(\vartheta_{12}-\vartheta_{21})\N}$ are not well defined. 

Note that equation \eqref{3.8} is reminiscent of the definition of locality in \emph{quantum vertex algebras} \cite{FR, EK, DGK}, where the map $e^{\ze\N} \colon V\otimes V\rightarrow (V\otimes V)[\ze]$ plays the role of a braiding map. Additionally, our hexagon identity is analogous to the hexagon identity for quantum vertex algebras, while condition \eqref{2.4-3} in \deref{lm} trivially implies that $\N$ is a solution of the classical Yang--Baxter equation.  
%In quantum vertex algebras the braiding map is defined, in particular, as a Laurent series on a formal variable $z$. We hope in the future to study more general cases of quantum vertex algebras where additionally the braiding map depends on a logarithmic variable.
\end{remark}

%We have the following natural algebraic definition on logVAs. 
%\begin{definition} A \emph{homomorphism} $\rho$ between logVAs $V$ and $V'$ is a linear map $\rho\colon V\rightarrow V'$ such that $\rho(\vac)=\vac'$, $\, T' \rho=\rho  T$, $\, \N' (\rho\otimes\rho)= (\rho\otimes\rho) \N$ and 
%\begin{equation*}
%\rho(Y(a,z)b)=Y'(\rho(a),z)\rho(b)\, \qquad a,b\in V\, .
%\end{equation*}
%\end{definition}
%We also have that the tensor product of logVAs defines a logVA. The proposition below is a consequence of the definition of logVA.  
%\begin{remark}
%Let $(V',\vac',T',Y',\N')$ and $(V'',\vac'',T'',Y'',\N'')$ be two logVAs. Then their \emph{tensor product} is the logVA
%\[(V'\otimes V'', \, \vac'\otimes \vac'', \, T'\otimes I+I\otimes T'', \, Y, \, \N'_{13}+\N''_{24}),\] 
%where $Y(a\otimes b ,z)=Y'(a,z)\otimes Y''(b,z)$ for $a\in V'$ and $b\in V''$.
%\end{remark}

\subsection{Conformal logarithmic vertex algebras}\label{s3.2c}
As we mention in the introduction, the following definition is motivated by logarithmic conformal field theory (see \cite{G1,G2, Ga}).

\begin{definition}\label{d3.10}
A logVA is called \emph{conformal} of central charge $c\in\CC$ if there exists an even vector $\om\in V_{\bar 0}$ (the \emph{conformal vector}),
with the following properties:

\begin{enumerate}

\item\label{3.17-1}
$\N(\om\otimes v) = 0$ for all $v\in V$ (i.e., $\om\in V^0$, see \eqref{V0}).

\medskip
\item\label{3.17-2}
The modes $L_n:=\om_{(n+1+\N)}$
%\begin{equation*}
%Y(\om,z)=\sum_{n\in \ZZ}L_{n} \, z^{-n-2} \,, \qquad L_n\in\End(V) \,,
%\end{equation*}
%where the operators $L_{n}$ $(n\in\ZZ)$ 
satisfy the commutation relations of the Virasoro Lie algebra with central charge $c$.
%\begin{equation*}
%[L_m,L_n] = (m-n)L_{m+n} + \delta_{m,-n} \frac{m^3-m}{12} c \, I \,,
%\end{equation*}

\medskip
\item\label{3.17-3}
$L_{-1}=T$ is the translation operator of $V$.

\medskip
\item\label{3.17-4}
$L_{0}$ is \emph{locally finite} on $V$ (i.e., every $v\in V$ is contained in a finite-dimensional subspace invariant under $L_0$).
\end{enumerate}
\end{definition}

Condition \eqref{3.17-1} above and \eqref{logva3} imply that
\begin{equation}\label{Virfield}
Y(\om,z)=L(z)=\sum_{n\in \ZZ}L_{n} z^{-n-2}
\end{equation}
is a \emph{Virasoro field} (called the stress-energy tensor and usually denoted as $T(z)$ in the physics literature).
Notice that the difference between this definition and that of a conformal vertex algebra is that for the latter the linear operator $L_{0}$ must be semisimple (see \cite{K}).
Since $L_0$ is locally finite, it has a \emph{Jordan--Chevalley decomposition}
\begin{equation}\label{JCD}
L_{0}=L^{(s)}_{0}+L_{0}^{(n)} \,, \qquad  \bigl[L^{(s)}_{0}, L_{0}^{(n)}\bigr] = 0 \,,
\end{equation}
where $L^{(s)}_{0}$ is semisimple and $L^{(n)}_{0}$ is locally nilpotent on $V$. %(cf.\ \cite[Lemma 4.6]{B}).

Because the Virasoro field $Y(\om,z)$ is non-logarithmic, it satisfies the usual commutator formula for vertex algebras: % GIVE REF.
\begin{align*}
[Y(\om,z_1),Y(a,z_2)] &= \sum_{j\in\ZZ_+} Y(\om_{(j+\N)} a, z_2) \, \partial_{z_2}^{(j)} \delta(z_1,z_2)
%&= \sum_{j\in\ZZ_+} Y(L_{j-1} a, z_2) \partial_{z_2}^{(j)} \delta(z_1,z_2)
\end{align*}
(see \coref{corbor} below).
Comparing the coefficients in front of $z_1^{-2}$ gives
\begin{equation}\label{L0}
\begin{split}
[L_{0},Y(a,z)] &=zY(L_{-1}a,z)+ Y(L_{0}a,z)\\
&=zD_{\z}Y(a,z)+ Y(L_{0}a,z)\,,
\end{split}
\end{equation}
just like for ordinary vertex algebras (see \cite{K}).
From the uniqueness of the Jordan--Chevalley decomposition $zD_{z}=z\partial_{z}+\partial_{\zeta}$, we obtain:
\begin{align*}
\bigl[L_{0}^{(s)}, Y(a,z)\bigr]&=z\partial_{z}Y(a,z)+Y\bigl(L_{0}^{(s)}a,z\bigr)\, , \\
\bigl[L_{0}^{(n)}, Y(a,z)\bigr]&=\partial_{\ze}Y(a,z)+Y\bigl(L_{0}^{(n)}a,z\bigr)\,.
\end{align*} 
In particular, we can exponentiate the operator $\partial_{\ze}$ and derive from the second equation that
\begin{equation*}
Y(a,z)=e^{\ze L_{0}^{(n)}}X\bigl(e^{-\ze L_{0}^{(n)}}a,z\bigr)e^{-\ze L_{0}^{(n)}}\, .
\end{equation*}

On the other hand, we know from \eqref{logva3} that the $\ze$-dependence of $Y(a,z)$ is governed by the braiding map $\N$.
The following proposition relates the operators $L_{0}^{(n)}$ and $\N$ in a special case, which appears often in examples.

\begin{proposition}\label{pro3.11} Let\/ $V$ be a  conformal logVA, in which the locally nilpotent part of the Virasoro operator\/ $L_0$ has the form
\begin{equation*}
L_{0}^{(n)}=d+\sum_{i=1}^L f_i g_i %\qquad\text{where}\quad d,f_i,g_i\in \Der(V)\,, \;\; p(f_i)=p(g_i) \,, \;\; p(d)=\bar0\,.
\end{equation*}
for some\/ $d,f_i,g_i\in \Der(V)$ with\/ $p(f_i)=p(g_i)$ and\/ $p(d)=\bar0$.
Then for all\/ $a,b\in V$, we have
\begin{equation*}
\N(a\otimes b)=-\sum_{i=1}^L \Bigl(f_i{\otimes} g_i+(-1)^{p(f_i)}g_i{\otimes} f_i\Bigr)(a\otimes b) \mod \Ker Y(z) \, .
\end{equation*}
\end{proposition}
\begin{proof} 
By \eqref{logva3}, $\partial_{\ze}Y(a,z)b=-Y(z)\N (a\otimes b)$. On the other hand, using that $d$, $f_i$ and $g_i$ are derivations, we find:
\allowdisplaybreaks
\begin{align*}
\partial_{\ze} &Y(a,z)b +Y\bigl((L_{0}^{(n)}-d)a,z\bigr)b \\ 
&=\bigl[L_{0}^{(n)} - d, Y(a,z)\bigr]b \\
&= \sum_{i=1}^L \Bigl( f_i[g_i,Y(a,z)] + (-1)^{p(g_i)p(a)} [f_i,Y(a,z)] g_i \Bigr)b \\
&=\sum_{i=1}^L \Bigl( f_i Y(g_i a,z) + (-1)^{p(g_i)p(a)} Y(f_i a,z) g_i \Bigr)b \\
&=\sum_{i=1}^L \Bigl( Y( f_i g_i a,z) \\
&\qquad\quad+ (-1)^{p(f_i)(p(a)+p(g_i))} Y(g_i a,z) f_i + (-1)^{p(g_i)p(a)} Y(f_i a,z) g_i \Bigr)b \\
&= Y\bigl((L_{0}^{(n)}-d)a,z\bigr)b + \sum_{i=1}^L Y(z) \bigl(f_i{\otimes} g_i+(-1)^{p(f_i)}g_i{\otimes} f_i\bigr)(a\otimes b) \,.
\end{align*}
Thus,
\begin{equation*}
\partial_{\ze} Y(a,z)b = \sum_{i=1}^L Y(z) \bigl(f_i{\otimes} g_i+(-1)^{p(f_i)}g_i{\otimes} f_i\bigr)(a\otimes b) \,,
\end{equation*}
which completes the proof.
\end{proof}

\begin{remark}
Suppose that a conformal logVA $V$ has two different conformal vectors $\om'$ and $\om''$ with corresponding Virasoro operators $L'_n$
and $L''_n$, respectively. Then \eqref{L0} implies that $L'_0-L''_0 \in\Der(V)$. Since the locally nilpotent part of a derivation is a derivation
(see, e.g., \cite[Lemma 4.6]{B}), we obtain that $(L'_0)^{(n)}-(L''_0)^{(n)} \in\Der(V)$.
If we can apply \prref{pro3.11} for the conformal vector $\om'$, then we can also apply it for $\om''$ and we will get the same braiding map.
\end{remark}

\subsection{Skew-symmetry, associativity, and OPE}\label{s3.4} 
Logarithmic vertex algebras satisfy skew-symmetry and associativity properties similar to those of ordinary vertex algebras and generalized vertex algebras
(see \cite{DL,FB,LL,BK2}).
Below we will sometimes specify explicitly the dependence of the logarithmic fields on both $z$ and $\ze$,
writing $Y(a,z,\ze)$ and $Y(z,\ze)$ instead of $Y(a,z)$ and $Y(z)$, respectively.
The next lemma is useful to prove the skew-symmetry relation.

\begin{lemma}\label{l3.15} 
In any logVA, we have\/ $Y(a,\z)\vac=e^{zT}a$ and
\begin{equation*}%\label{logva7}
e^{-z_{2}T}Y(a,\z_{1},\ze_1)e^{z_{2}T}=\iota_{z_{1},z_{2}}Y(a,\z_{12}, \ze_{12})\,,
\end{equation*}
where we use the expansions\/ \eqref{iota}, \eqref{iota2}.
\end{lemma}
\begin{proof} 
The identity $Y(a,\z)\vac=e^{zT}a$ follows from $T\vac=0$ and the translation covariance of $Y(a,z)$; see \eqref{logf22c}.
Again by translation covariance, we have
\begin{equation*}
e^{-z_{2}T}Y(a,\z_{1},\ze_1)e^{z_{2}T} = e^{-z_{2} \ad(T)}Y(a,\z_{1},\ze_1) = e^{-z_{2}D_{z_1}}Y(a,\z_{1},\ze_1) \,.
\end{equation*}
Then the second identity in the lemma follows from Taylor's formula, since $e^{-z_{2}D_{z_1}} z_1^n = \iota_{z_{1},z_{2}} z_{12}^n$
and $e^{-z_{2}D_{z_1}} \ze_1 = \iota_{z_{1},z_{2}} \ze_{12}$ for $n\in\ZZ$.
\end{proof}

\begin{proposition}[Skew-symmetry] \label{pro3.16}
For any logVA\/ $V$ and\/ $a,b\in V$, we have
\begin{equation*}
Y(a,z,\ze) b = (-1)^{p(a)p(b)} e^{zT}Y(b,-z,\ze)a \,.
\end{equation*}
In particular, 
\begin{equation*}
X(a,z) b = (-1)^{p(a)p(b)} e^{zT}X(b,-z)a \,.
\end{equation*}
\end{proposition}
\begin{proof} 
We use the locality relation \eqref{logva5} for $c=\vac$. By Lemma \ref{l3.15} and \eqref{logva3}, 
we can set $z_2=0$ in the left-hand side of \eqref{logva5} and obtain:
\begin{align*}
Y(z_1)&(I \otimes Y(z_2))z_{12}^{N} e^{\vartheta_{12}  \N_{12}} (a\otimes b\otimes \vac)\big|_{z_{2}=0}\\
& =Y(z_1)z_{1}^{N} e^{\ze_1\N}(a\otimes b)=z_{1}^{N}X(z_1)(a\otimes b)\, .
\end{align*}
Again from Lemma \ref{l3.15}, \eqref{logva3} and \eqref{iota2}, we find that the right-hand side of \eqref{logva5} gives:
\begin{align*}
 Y(z_2) &(I \otimes Y(z_1)) z_{12}^{N} e^{\vartheta_{21}  \N_{12} } (b\otimes a \otimes \vac) \\
 &= \mu\bigl((Y_{z_2} \otimes e^{z_1 T}) z_{12}^{N} e^{\vartheta_{21}  \N} (b\otimes a) \bigr) \\
&=e^{z_1 T} \iota_{z_2,z_1} z_{12}^{N} Y(z_{21},\ze_{21}) e^{\ze_{21}  \N} (b\otimes a) \\
&=e^{z_1 T} \iota_{z_2,z_1} z_{12}^{N} X(z_{21}) (b\otimes a) \,.
%&=z_{1}^{N}e^{z_1 T}X(-z_{1}) (b\otimes a) \,  .
\end{align*}
We pick $N$ large enough so that $z^N X(b,z)a$ has only non-negative powers of $z$. Then it makes sense to set $z_2=0$
in the above expression and get
\begin{equation*}
z_{1}^{N}e^{z_1 T}X(-z_{1}) (b\otimes a) \,.
\end{equation*}
Therefore,
\begin{equation*}
X(z) (a\otimes b) = (-1)^{p(a)p(b)} e^{z T}X(-z)(b\otimes a)\,.
\end{equation*}
To finish the proof, we replace in this identity $a\otimes b$ with $e^{-\ze\N}(a\otimes b) \in (V\otimes V)[\ze]$ and use again \eqref{logva3}.
\end{proof}

In the associativity relation below and its proof, it will be convenient to use notation \eqref{logf-1} and \eqref{z12N}.

\begin{proposition}[Associativity]\label{pro3.17}
For any logVA\/ $V$ and\/ $a,b,c\in V$, there exists $N\in \ZZ_{+}$ such that 
\begin{equation}\label{logva9}
\begin{split}
\iota_{z_{1},z_{2}} & Y({z_{12}},\ze_{12}) (I\otimes Y({-z_2},\ze_2)) \z_{12}^N\,e^{\vartheta_{12} \NN_{13}} (a\otimes b\otimes c)\\
&=Y({-z_2},\ze_2)(Y({z_1},\ze_1)\otimes I) \z_{12}^N\,e^{\vartheta_{21} \NN_{13}} (a\otimes b\otimes c)\,,
\end{split}
\end{equation}
where the expansion\/ $\iota_{z_1,z_2}$ is as in\/ \eqref{iota}, \eqref{iota2}.
\end{proposition}
\begin{proof} 
Let us apply $e^{-z_{2}T}$ to both sides of the locality relation \eqref{logva5}.
By  Lemma \ref{l3.15} and  Proposition \ref{pro3.16}, the left-hand side of \eqref{logva5} becomes
\begin{align*}
e^{-z_{2}T} &Y({z_1}) (I \otimes Y({z_2})) z_{12}^{N+\N_{12}} (a\otimes b\otimes c) \\
&=\iota_{z_{1},z_{2}}Y({z_{12}},\ze_{12}) (I \otimes e^{-z_{2}T}Y({z_2})) z_{12}^{N+\N_{12}} (a\otimes b\otimes c)\\
& =(-1)^{p(b)p(c)}\iota_{z_{1},z_{2}}Y(z_{12},\ze_{12}) (I\otimes Y({-z_2},\ze_2)) z_{12}^{N+\N_{13}} (a\otimes c\otimes b)\,,
\end{align*}
while the right-hand side of \eqref{logva5} becomes
\begin{align*}
(-1)^{p(a)p(b)} &e^{-z_{2}T} Y({z_2}) (I\otimes Y({z_1})) z_{21}^{N+\N_{12}} (b\otimes a\otimes c)\\
&=(-1)^{p(b)p(c)}Y({-z_2},\ze_2)(Y({z_1})\otimes I) z_{21}^{N+\N_{13}} (a\otimes c\otimes b)\, .
\end{align*}
This completes the proof.
\end{proof}

%\begin{remark}
If we replace $z_2$ with $-z_2$ in the associativity identity \eqref{logva9}, it becomes equivalent to:
\begin{equation}\label{logva10}
\begin{split}
\iota_{z_{1},z_{2}} & (z_1+z_2)^{N}Y({z_1+z_{2}}, \tilde{\ze}_{12}) (I\otimes Y({z_2})) e^{\tilde{\ze}_{12}\N_{13}} (a\otimes b\otimes c)\\
& = \iota_{z_2,z_1} (z_1+z_2)^{N} Y({z_2}) (Y({z_1})\otimes I) e^{\tilde{\ze}_{21}\N_{13}} (a\otimes b\otimes c)\,,
\end{split}
\end{equation}
where $\tilde{\ze}_{12}$ and $\tilde{\ze}_{21}$ are new formal variables and (cf.\ \eqref{iota7})
\begin{equation}\label{iota14}
\iota_{z_{1},z_{2}} \tilde{\ze}_{12}:=\zeta_{1}-\sum_{j=1}^\infty \frac{(-1)^{j}}{j} z_{1}^{-j} z_{2}^{j} \,.
\end{equation}
%\end{remark}

We also have a version of formal commutativity and associativity, which follows immediately from \thref{th2.16},
\prref{p-modes} and \coref{c3.13}. We use the expansions \eqref{iota14} and \eqref{iota}--\eqref{iota6}.

\begin{proposition}\label{th4.4}
Let\/ $V$ a logVA and\/ $a, b, c\in V$. Then there exists a formal series 
\[\Y_{a,b,c}(z_1,z_2,\ze_1,\ze_2,\ze_3)\in V[\![z_{1},z_{2}]\!][z_{1}^{-1}, z_{2}^{-1}, z_{12}^{-1},  \ze_{1}, \ze_{2}, \ze_3] \]
with the property that 
\begin{equation*}
\begin{split}
Y(a,z_1)Y(b,z_2)c &=\iota_{z_1,z_2}\Y_{a,b,c}(z_1,z_2,\ze_1,\ze_2,\ze_{12})\, ,\\
Y(b,z_2)Y(a,z_1)c &=(-1)^{p(a)p(b)} \iota_{z_2,z_1}\Y_{a,b,c}(z_1,z_2,\ze_1,\ze_2,\ze_{12})\, ,\\
Y(Y(a, z_3)b,z_2)c &=\iota_{z_2,z_3}\Y_{a,b,c}(z_2+z_3,z_2,\tilde{\ze}_{23},\ze_2,\ze_3)\,. 
\end{split}
\end{equation*}
Furthermore, 
\begin{equation*}%\label{iota13}
z_{12}^N \Y_{a,b,c}(z_1,z_2,\ze_1,\ze_2,\ze_3)\in V[\![z_{1},z_{2}]\!][z_{1}^{-1}, z_{2}^{-1}, \ze_{1}, \ze_{2}, \ze_3] \,,
\end{equation*}
for some\/ $N\in\ZZ_+$ that depends only on\/ $a$ and\/ $b$.
\end{proposition}

As in \seref{OPE}, we can use \prref{th4.4} to express the OPEs in a logVA as follows:
\begin{equation*}%\label{logva5b}
Y(a,z_1)Y(b,z_2) \sim \sum_{i \in\ZZ_+, \, n\in\ZZ} \frac{(-1)^{i}}{i!}\ze_{12}^{i}z_{12}^{-n-1}Y(\N^{i}a_{(n+\N)}b,z_2)\, .
\end{equation*}
In the physics literature, this is written as
\begin{equation}\label{ope5}
Y(a,z_1)Y(b,z_2)\sim \sum_{i \in\ZZ_+, \, n\in\ZZ} \frac{(-1)^{i}\log^{i}(z_1-z_2)}{i!(z_{1}-z_{2})^{n+1}}Y(\N^{i}a_{(n+\N)}b,z_2)\, , 
\end{equation}
or after setting $z_2=0$ as
\begin{equation}\label{ope6}
Y(a,z)Y(b,0)\sim \sum_{i \in\ZZ_+, \, n\in\ZZ} \frac{(-1)^{i}\log^{i} z}{i! z^{n+1}}Y(\N^{i}a_{(n+\N)}b,0)\, . 
\end{equation}
Notice that $Y(b,0)$ does not make sense by itself but it does after it is applied to the vacuum vector $\vac$. Then
$Y(b,0)\vac=b$ and the above formula corresponds to the expansion \eqref{nprod-7} of $Y(a,z)b$.
We define the \emph{singular OPE} (or the singular part of the OPE) as the sum over the subset $\{i=0,\,n\ge0\} \cup \{i\ge1,\,n\in\ZZ\}$.

\subsection{Borcherds identity}\label{s3.5}

The formal delta function $\delta(z_1,z_2)$ defined in \eqref{logf10} plays an important role in the theory of vertex algebras (see \cite{FLM,K,FB,LL}).
We will utilize the following generalization introduced in \cite{B}:
\begin{equation}\label{deltaN}
\delta_{\N}(z_1,z_2):= \delta(z_1,z_2) e^{(\ze_{2}-\ze_1)\N} = \sum_{n\in\ZZ} z_{1}^{-n-1-\N} z_{2}^{n+\N}\,,
\end{equation}
which is a linear operator 
\begin{equation*}
\de_\N(z_1,z_2) \colon V\otimes V \to \CC[\![z_1^{\pm1},z_2^{\pm1}]\!] \otimes (V\otimes V)[\ze_1,\ze_2] \,.
\end{equation*}
Notice that
\begin{equation}\label{deltaN1}
(D_{z_1}+D_{z_2}) \de_\N(z_1,z_2) = 0 \,.
\end{equation}
We will not list all other properties of $\de_\N$ but only one that is needed here.

\begin{lemma}\label{ldeN} 
For any two elements $a,b$ in a logVA, one has
\begin{equation*}
Y(z_1)(a\otimes b) \delta(z_1,z_2)= Y(z_2)\delta_{\N}(z_1,z_2) (a\otimes b)\, .
\end{equation*}
\end{lemma}
\begin{proof}
Indeed, from \eqref{logva3} and $X(z_1)\delta(z_1,z_2) = X(z_2)\delta(z_1,z_2)$, we have
\begin{align*}
Y(z_1)(a\otimes b) \delta(z_1,z_2)
&= X(z_1) e^{-\ze_1\N} (a\otimes b) \delta(z_1,z_2) \\
&= X(z_2) e^{-\ze_1\N} (a\otimes b) \delta(z_1,z_2) \\
&= Y(z_2) e^{\ze_2\N} e^{-\ze_1\N} (a\otimes b) \delta(z_1,z_2) \,,
\end{align*}
as claimed.
\end{proof}

Now we can state a version of the Borcherds identity for logarithmic vertex algebras,
which is one of the main results of the paper.

\begin{theorem}[Borcherds identity] \label{t3.20} 
For every three elements\/ $a,b,c$ in a logVA\/ $V$ and\/ $n\in\ZZ$, we have the following identity
\begin{equation}\label{borcherds}
\begin{split}
\iota_{z_1,z_2} & Y(z_1)(I\otimes Y(z_2))z_{12}^n e^{\vartheta_{12}\N_{12}} (a\otimes b \otimes c) \\
&-(-1)^{p(a)p(b)}\iota_{z_2,z_1}Y(z_2)(I\otimes Y(z_1))z_{12}^n e^{\vartheta_{21}\N_{12}} (b\otimes a \otimes c) \\
&=\sum_{j\in\ZZ_{+} }Y(z_2)(\mu_{(n+j)}\otimes I)D_{z_{2}}^{(j)}\delta_{\N_{13}}(z_{1},z_{2}) (a\otimes b\otimes c)\, .
\end{split}
\end{equation}
\end{theorem}
\begin{proof} 
The proof is similar to that of \cite[Theorem 5.2]{B}. 
For a sufficiently large $N$, let us denote
\begin{equation*}
F(z_1,z_2) := Y(z_1)(I\otimes Y(z_2))z_{12}^N e^{\vartheta_{12}\N_{12}} \,,
\end{equation*}
which implicitly also depends on $\ze_1$ and $\ze_2$. It follows from locality \eqref{logva5} and \eqref{logva3} that
\begin{equation*}
F(z_1,z_2) e^{\ze_{1}\N_{13}} (a\otimes b\otimes c) \in V(\!(z_{2})\!)[\ze_{2}](\!(z_{1})\!)\, . 
\end{equation*}
Hence, as in the proof of \leref{ldeN}, we have
\begin{equation*}
F(z_1,z_2)(a\otimes b\otimes c) \de(z_1,z_2)  = F(z_2,z_2) \delta_{\N_{13}}(z_1,z_2) (a\otimes b\otimes c) \,. 
\end{equation*}
Next, recall that
\begin{equation*}
\de(z_1,z_2) = (\iota_{z_1,z_2}-\iota_{z_2,z_1}) z_{12}^{-1} \,,
\end{equation*}
which implies
\begin{equation*}
\partial_{z_2}^{(n)} \de(z_1,z_2) = (\iota_{z_1,z_2}-\iota_{z_2,z_1}) z_{12}^{-n-1} \,, 
\qquad n\in\ZZ \,.
\end{equation*}
Then using locality \eqref{logva5} again, we find that the left-hand side of \eqref{borcherds} is equal to:
\begin{equation*}
F(z_1,z_2)(a\otimes b\otimes c) (\iota_{z_1,z_2}-\iota_{z_2,z_1}) z_{12}^{n-N}
=F(z_1,z_2)(a\otimes b\otimes c) \partial^{(N-n-1)}_{z_2}\delta(z_1,z_2)\, .
\end{equation*}
Applying the Leibniz rule, we obtain:
\begin{align*}
F&(z_1,z_2)(a\otimes b\otimes c) \partial^{(N-n-1)}_{z_2} \de(z_1,z_2) \\
&= D^{(N-n-1)}_{z_2} \bigl( F(z_1,z_3) (a\otimes b\otimes c) \de(z_1,z_2) \bigr)\big|_{z_3=z_2} \\
&= D^{(N-n-1)}_{z_2} \bigl( F(z_2,z_3) \de_{\N_{13}}(z_1,z_2) (a\otimes b\otimes c) \bigr)\big|_{z_3=z_2} \\
&= \sum_{j\ge0} \bigl( D_{z_1}^{(N-n-1-j)} F(z_1,z_2) \bigr)\big|_{z_1=z_2} D_{z_{2}}^{(j)} \de_{\N_{13}}(z_1,z_2) (a\otimes b\otimes c) \,.
\end{align*}
Now to finish the proof, we observe that
\begin{equation*}
D_{z_1}^{(N-n-1-j)} F(z_1,z_2) \big|_{z_1=z_2} = Y(z_2)(\mu_{(n+j)}\otimes I) \,,
\end{equation*}
by the $(n+\N)$-th product identity \eqref{eq-n-prod}.
\end{proof}

\begin{remark}\label{rembor}
Since $\N$ is locally nilpotent, the coefficients of $\delta_{\N_{13}}(z_{1},z_{2}) (a\otimes b\otimes c)$ span a finite-dimensional subspace of $V^{\otimes 3}$.
Hence, there exists $N\in\ZZ_+$ such that $\mu_{(N+j)}\otimes I$ vanishes on this subspace for all $j\ge0$, and the Borcherds identity \eqref{borcherds}
implies the locality \eqref{logva5}, with $N$ possibly depending on not only $a,b$ but $c$ as well. 

Suppose that $V$ has the following property:
for every $a,b\in V$, there is $N\in\ZZ_+$ such that
\begin{equation*}%\label{Phi-nilp}
\bigl(\phi_{i_1} \cdots \phi_{i_k} a\bigr)_{(n+\N)}b = 0
\quad\text{for all}\quad n\ge N \,, \; k\ge 0 \,, \; 1\le i_1,\dots,i_k\le L \,,
\end{equation*}
where the components $\phi_i$ of $\N$ are given by \eqref{logf-n}. Then the Borcherds identity \eqref{borcherds}
implies the locality \eqref{logva5} with $N$ depending only on $a,b$.
\end{remark}

\begin{corollary}\label{corbor}
If\/ $a\in V^0$, i.e., $\N(a\otimes v)=0$ for all\/ $v\in V$, then we have the usual \emph{Borcherds identity}
for every\/ $b\in V$, $n\in\ZZ{:}$
\begin{align*}
\iota_{z_1,z_2} & Y(a,z_1)Y(b,z_2) z_{12}^n
-(-1)^{p(a)p(b)}\iota_{z_2,z_1}Y(b,z_2) Y(a,z_1)z_{12}^n \\
&=\sum_{j\in\ZZ_{+} }Y(a_{(n+j+\N)}b, z_2) \partial_{z_{2}}^{(j)}\delta(z_{1},z_{2}) \,,
\end{align*}
and, in particular, the usual \emph{commutator formula}
\begin{equation*}
[Y(a,z_1),Y(b,z_2)] =\sum_{j\in\ZZ_{+} }Y(a_{(j+\N)}b, z_2) \partial_{z_{2}}^{(j)}\delta(z_{1},z_{2}) \,.
\end{equation*}
\end{corollary}

We can express the Borcherds identity \eqref{borcherds} in terms of the product $\mu_{(n)}$ defined in \eqref{eq3.7a}. 
To do so, notice that  by the definition of $\iota_{z_1,z_2}$ and $\iota_{z_2,z_1}$ (see \eqref{iota}--\eqref{iota6}), 
we have for $n\in\ZZ$:
\begin{align*}
\iota_{z_1,z_2} z_{12}^{n+\N} = \iota_{z_1,z_2} z_{12}^{n} e^{\vartheta_{12}\N}
&=\sum_{j\in\ZZ_{+}}\binom{n+\N}{j} (-1)^j z_{1}^{n-j+\N}z_{2}^{j}\, ,\\
\iota_{z_2,z_1} z_{21}^{n+\N} = \iota_{z_2,z_1} z_{12}^{n} e^{\vartheta_{21}\N}
&=\sum_{j\in\ZZ_{+}}\binom{n+\N}{j} (-1)^{n+j}  z_{1}^{j}z_{2}^{n-j+\N}\, ,
\end{align*}
where
\[\binom{n+\N}{j}=\frac1{j!} (n+\N)(n+\N-1)\cdots (n+\N-j+1)\, .\]
Using these expansions and the mode expansion \eqref{nprod-7} of logarithmic fields, we see that for any fixed $n\in\ZZ$,
equation \eqref{borcherds} is equivalent to the following collection of identities on $V^{\otimes 3}$  for all $k,m\in\ZZ$:
%\begin{equation*}
%\begin{split}
%&\sum_{j\in \ZZ}(-1)^{j}\Big(\mu_{(m+n-j)}z_{1}^{-\N}(I\otimes \mu_{(k+j)})z_{1}^{\N_{12}}z_{2}^{-\N_{23}}\\
%&-(-1)^{p(a)p(b)}\mu_{(n+k-j)}\z_{2}^{-\N}(I\otimes \mu_{(m+j)})z_{1}^{-\N_{23}}z_{2}^{\N_{12}}\Big)\binom{n+\N_{12}}{j}\\
%&=\sum_{j\in \ZZ} \mu_{(m+k-j)}z_{2}^{-\N}( \mu_{(n+j)}\otimes I)z_{2}^{\N_{13}}z_{1}^{-\N_{13}} \binom{m+\N_{13}}{j}
%\end{split}
%\end{equation*}
\begin{equation}\label{borcherds2}
\begin{split}
\sum_{j\in \ZZ_+} &(-1)^{j}\mu_{(m+n-j)}(I\otimes \mu_{(k+j)})\binom{n+\N_{12}}{j}\\
& -\sum_{j\in \ZZ_+}(-1)^{n+j}\mu_{(n+k-j)}(I\otimes \mu_{(m+j)})\binom{n+\N_{12}}{j} (P \otimes I) \\
& =\sum_{j\in \ZZ_+} \mu_{(m+k-j)}( \mu_{(n+j)}\otimes I) \binom{m+\N_{13}}{j}\, .
\end{split}
\end{equation}

\begin{corollary}\label{corbor2}
If\/ $a,b\in V$ are such that\/ $\N(a\otimes b)=0$, then for all\/ $m,k\in\ZZ$ and\/ $v\in V$, we have the \emph{commutator formula}
\begin{equation*}
\bigl[a_{(m+\N)},b_{(k+\N)}\bigr]v =\sum_{j\in\ZZ_{+} }\mu_{(m+k-j)}( \mu_{(j)}\otimes I) \binom{m+\N_{13}}{j} (a\otimes b\otimes v) \,.
\end{equation*}
\end{corollary}
\begin{proof}
Set $n=0$ in \eqref{borcherds2} and apply both sides to $a\otimes b\otimes v$.
\end{proof}

The Borcherds identity from \cite{K} is the special case $\N=0$ of \eqref{borcherds2}. 
We also have the following logarithmic version of the Jacobi identity from \cite{FLM}.

\begin{proposition}%[Jacobi identity]
\label{pro3.21} 
In any logVA\/ $V$, we have the following identity on\/ $V^{\otimes 3}{:}$
\begin{align*}
\iota_{z_1,z_2} &Y(z_1)(I\otimes Y(z_2))\delta_{\N_{12}}(z_{3},z_{12}) \\
&-\iota_{z_2,z_1}Y(z_2)(I\otimes Y(z_1))\delta_{\N_{12}}(z_{3},z_{12})(P\otimes I)\\
&=\iota_{z_1,z_3}Y(z_2)(Y(z_3)\otimes I)\delta_{\N_{13}}(z_{13},z_{2}) \, . 
\end{align*}
\end{proposition}
\begin{proof}
After it is applied to $a\otimes b\otimes c$, %the Jacobi 
this identity is equivalent to the collection of Borcherds identities \eqref{borcherds} for all $n\in\ZZ$. 
Indeed, we can express each of the terms in \eqref{borcherds} as a residue:
\begin{align*}
\iota_{z_1,z_2} & Y(z_1)(I\otimes Y(z_2))z_{12}^{n} e^{\vartheta_{12}\N_{12}} (a\otimes b \otimes c) \\
&=\Res_{z_{3}}\iota_{z_1,z_2}\Bigl(Y(z_1)(I\otimes Y(z_2)) z_{3}^{n}\delta(z_{3},z_{12}) e^{\vartheta_{12}\N_{12}} (a\otimes b\otimes c) \Bigr)\\
&=\Res_{z_{3}}\iota_{z_1,z_2}\Bigl(Y(z_1)(I\otimes Y(z_2))z_3^{n+\N_{12}}\delta_{\N_{12}}(z_{3},z_{12}) (a\otimes b\otimes c) \Bigr) ;
\end{align*}
similarly, 
\begin{align*}
\iota_{z_2,z_1} &Y(z_2)(I\otimes Y(z_1))z_{12}^n e^{\vartheta_{21}\N_{12}}(b\otimes a)\otimes c\\
& =\Res_{z_{3}}\iota_{z_2,z_1}\Bigl(Y(z_2)(I\otimes Y(z_1))z_3^{n+\N_{12}}\delta_{\N_{12}}(z_{3},z_{12})(b\otimes a \otimes c)\Bigr)
\end{align*}
and 
\begin{align*}
&\sum_{j\geq 0}Y(z_2)(\mu_{(n+j)}\otimes I)D_{z_{2}}^{(j)}\delta_{\N_{13}}(z_{1},z_{2}) (a\otimes b\otimes c)\\
&\qquad=\sum_{j\geq 0}Y(z_2)(\mu_{(n+j)}\otimes I)(-1)^{j}D_{z_{1}}^{(j)}\delta_{\N_{13}}(z_{1},z_{2}) (a\otimes b\otimes c)\\
&\qquad=\Res_{z_3}\Bigl(Y(z_2)(X(z_3)\otimes I)z_3^{n}e^{-z_3D_{z_1}}\delta_{\N_{13}}(z_{1},z_{2}) (a\otimes b\otimes c) \Bigr)\\
&\qquad=\Res_{z_{3}}\iota_{z_1,z_3}\Bigl(Y(z_2)(Y(z_3)\otimes I)z_3^{n+\N_{12}}\delta_{\N_{13}}(z_{13},z_{2})(a\otimes b\otimes c)\Bigr).
\end{align*}
This completes the proof.
\end{proof}

\subsection{Relation to twisted logarithmic modules}\label{ex4.3}
%The following example was mentioned in the introduction.
Let $U$ be an ordinary vertex algebra with a locally-nilpotent derivation $\mathcal{N}\in \Der(U)$. Then $\varphi=e^{-2\pi\ii \mathcal{N}}$ is an automorphism of $U$,
which is not semisimple unless $\mathcal{N}=0$. Consider a \emph{$\varphi$-twisted module} $W$ in the sense of \cite{B}.

We will introduce a logVA structure on the vector space direct sum $V=U\oplus W$, so that it becomes an \emph{abelian extension} of $U$ by $W$.
Namely, $Y(a,z)b$ is defined as in $U$ for $a,b\in U$, and $Y(a,z)b=0$ for $a,b\in W$. For $a\in U$ and $b\in W$, we define $Y(a,z)b$ using the $U$-action on $W$,
while $Y(b,z)a$ is determined from the skew-symmetry relation (see \prref{pro3.16}).
The braiding map on $V$ is given by
\begin{equation*}
\N = (\mathcal{N} \, \pi_U) \otimes \pi_W + \pi_W \otimes (\mathcal{N} \, \pi_U) \in \End(V)\otimes\End (V)\, ,
\end{equation*}
where $\pi_U\colon V\to U$ and $\pi_W\colon V\to W$ denote the projections.
Then $V$ is a logVA. Notice that, by construction,
\begin{equation*}
\N(a \otimes b) = 0 \,, \quad \N(a \otimes v) = (\mathcal{N} a) \otimes v \,, \qquad a,b\in U \,, \; v\in W\,.
\end{equation*}
Hence, the mode expansion \eqref{nprod-7} reduces to the following formula from \cite[(5.8)]{B}:
\begin{equation*}
Y(a,z)v = \sum_{n\in\ZZ} \bigl( z^{-n-1-\mathcal{N}} a\bigr)_{(n+\N)} v \,, \qquad a\in U \,, \; v\in W\,,
\end{equation*}
where $z^{-\mathcal{N}} := e^{-\ze\mathcal{N}}$.
Moreover, in the Borcherds identity \eqref{borcherds} for $a,b\in U$, $c\in W$, we can replace $\N_{12}=0$ and $\N_{13}=\mathcal{N} \otimes I \otimes I$,
and it becomes exactly the Borcherds identity \cite[(5.8)]{B}.
Similarly, identity \eqref{borcherds2} applied to $U\otimes U\otimes W$ corresponds to \cite[(5.3)]{B}.

Several examples of twisted logarithmic modules were provided in \cite{B,BS1,BS2}; they all give logarithmic vertex algebras due to the above construction.
We discuss additional examples in the next section.

\section{Examples of logarithmic vertex algebras}\label{s4}
In this section, we describe in detail several examples of conformal logVAs, including the well-known examples of symplectic fermions \cite{G1,Kau1,Kau2} and free bosons, 
a new example of a logarithmic vertex algebra associated to an integral lattice, and a logVA corresponding to Gurarie--Ludwig's logarithmic conformal field theory \cite{GL1,GL2,G2}.
In the last example, the explicit relations among the generators were previously unknown.

\subsection{Symplectic fermions}\label{ex4.2} 
As far as we know, symplectic fermions were historically the first example of logarithmic conformal field theory \cite{G1,Kau1,Kau2}.
Here, we interpret it as a conformal logVA.

Let $V$ be the Grassmann algebra $\bigwedge\bigl( \xi_{n},\chi_{n} \bigr)_{n=0,1,2,\dots}$ generated by odd variables $\xi_{n}$, $\chi_n$. 
On this vector superspace we consider the odd logarithmic fields given by:
\begin{align*}
\xi(z) %= Y(\xi_{0},z)  
&= \sum_{n=0}^\infty \xi_n z^n + \zeta \partial_{\chi_0} + \sum_{n=1}^\infty \partial_{\chi_n} \frac{z^{-n}}{-n}\,  , \\
\chi(z) %= Y(\chi_{0},z)  
&= \sum_{n=0}^\infty \chi_n z^n - \zeta \partial_{\xi_0} - \sum_{n=1}^\infty \partial_{\xi_n} \frac{z^{-n}}{-n}\,  . 
\end{align*}
The fields $\xi(z)$ and $\chi(z)$ satisfy
\begin{equation*}
\xi(z_1) \chi(z_2) = {:} \xi(z_1) \chi(z_2) {:} + \vartheta_{12}\,, \quad
[\xi(z_1), \xi(z_2)] = [\chi(z_1) \chi(z_2)] = 0\,.
\end{equation*}
Their derivatives
\begin{align*}
D_z\xi(z) &= \sum_{n=1}^\infty n\xi_n z^{n-1} + \sum_{n=0}^\infty \partial_{\chi_n} z^{-n-1} \,  , \\
D_z\chi(z) &= \sum_{n=1}^\infty n\chi_n z^{n-1} - \sum_{n=0}^\infty \partial_{\xi_n} z^{-n-1}
\end{align*}
are non-logarithmic and have the following singular OPEs:
\begin{align*}
\bigl(D_{z_1} \xi(z_1)\bigr) \bigl(D_{z_2} \chi(z_2)\bigr) &\sim \frac1{(z_1-z_2)^2} \,, &
\bigl(D_{z_1} \chi(z_1)\bigr) \bigl(D_{z_2} \xi(z_2)\bigr) &\sim \frac{-1}{(z_1-z_2)^2} \,, \\[6pt]
\bigl(D_{z_1} \xi(z_1)\bigr) \bigl(D_{z_2} \xi(z_2)\bigr) &\sim 0 \,, &
\bigl(D_{z_1} \chi(z_1)\bigr) \bigl(D_{z_2} \chi(z_2)\bigr) &\sim 0 \,,
\end{align*}
i.e., they form a pair of \emph{symplectic fermions}.

We let the vacuum vector $\vac=1$.
The translation operator $T\in \End(V)_{\bar0}$ is uniquely defined by the conditions $T\vac=0$ and
\begin{equation*}
[T,\xi_n]=(n+1)\xi_{n+1}\, ,\quad [T,\chi_n]=(n+1)\chi_{n+1}\, , \qquad n\ge0 \,.
\end{equation*}
Then it satisfies
\begin{equation*}
[T,\partial_{\chi_n}]=-n\partial_{\chi_{n-1}} \, ,\quad [T,\partial_{\xi_n}]=-n\partial_{\xi_{n-1}} \, , \qquad n\ge0 \,,
\end{equation*}
which implies that the fields $\xi(z)$ and $\chi(z)$ are translation covariant.
It is easy to check that the linear operator
\begin{equation*}
\N = \partial_{\xi_0} \otimes \partial_{\chi_0} - \partial_{\chi_0} \otimes \partial_{\xi_0}
\end{equation*}
is a braiding map on $V$ and commutes with $T\otimes I$.
Then, as in \thref{l2.19}, the operator 
\begin{equation*}
\NN = (\ad\otimes\ad)(\N) = \ad(\partial_{\xi_0}) \otimes \ad(\partial_{\chi_0}) - \ad(\partial_{\chi_0}) \otimes \ad(\partial_{\xi_0})
\end{equation*}
is a braiding map on $\LF(V)$ and also on $\LFT(V)$.

We claim that the subspace $\V=\Span\{\xi(z),\chi(z),I\} \subset\LFT(V)$ is $\NN$-local.
Indeed, conditions \eqref{2.5-1} and \eqref{2.5-2} of \deref{d2.5} follow immediately from the fact that
\begin{equation*}
[\partial_{\xi_0}, \xi(z)] = [\partial_{\chi_0}, \chi(z)] = I \,, \qquad
[\partial_{\xi_0}, \chi(z)] = [\partial_{\chi_0}, \xi(z)] = 0 \,,
\end{equation*}
which implies
\begin{equation*}
\NN(\xi\otimes\chi) = -I \otimes I \,, \quad \NN(\chi\otimes\xi) = I \otimes I \,, \quad \NN(\xi\otimes\xi) = \NN(\chi\otimes\chi) = 0 \,.
\end{equation*}
Hence,
\begin{align*}
e^{\vartheta_{12}\NN} \xi(z_{1})\chi(z_{2})
&=\xi(z_{1})\chi(z_{2})-\vartheta_{12} 
={:}\xi(z_{1})\chi(z_{2}){:} \\
&= -{:}\chi(z_{1})\xi(z_{2}){:}
=-e^{\vartheta_{21}\NN} \chi(z_{2})\xi(z_{1})\,,
\end{align*}
and the locality of all other pairs of fields is obvious.
Additionally, the space $\V$ is complete by construction (see \thref{l2.19}).

Therefore, by Theorems \ref{l2.19} and \ref{t3.14}, we obtain the structure of a logVA on $V$.
The state-field correspondence $Y$ is defined so that $Y(a,z)\vac|_{z=0} = a$ for all $a\in V$.
For instance, 
\begin{align*}
Y(\xi_0,z)&=\xi(z) \,, & Y(\chi_0,z)&=\chi(z) \,, \\ 
Y(\xi_1,z)&=D_z\xi(z) \,, & Y(\chi_1,z)&=D_z\chi(z) \,. 
\end{align*}
According to \eqref{ope5}, the generating fields $\xi(z)$ and $\chi(z)$ have the following singular OPEs:
\begin{align*}
\xi(z_1)\chi(z_2) &\sim \log(z_1-z_2) \,, &
\chi(z_1)\xi(z_2) &\sim -\log(z_1-z_2)\,, \\
\xi(z_1)\xi(z_2) &\sim 0 \,, &
\chi(z_1)\chi(z_2) &\sim 0 \,.
\end{align*}

Finally, we note that $V$ is a conformal logVA with central charge $c=-2$ (see Definition \ref{d3.10}), where the conformal vector is given by 
\begin{equation*}
\om = \chi_{1}\xi_{1}\, ,  \qquad  Y(\om,z)={:} \bigl(D_z\chi(z)\bigr)\bigl(D_z\xi(z)\bigr) {:}\, .
\end{equation*}
Then $L_{0}=L_{0}^{(s)}+L_{0}^{(n)}$ with
\begin{equation*}
 L^{(s)}_{0}=\sum_{n=1}^\infty \bigl( n\chi_{n}\partial_{\chi_{n}}+n\xi_{n}\partial_{\xi_{n}} \bigr), \qquad 
 L^{(n)}_{0}=\partial_{\chi_{0}}\partial_{\xi_{0}}\,,
\end{equation*}
which agrees with \prref{pro3.11}. 

It is clear that if we take $r$ anti-commuting pairs of fields $(\xi^i(z),\chi^i(z))$, where each pair is as $(\xi(z),\chi(z))$,
they will generate a logVA that is an $r$-th tensor power of the logVA $V$ constructed above. The tensor product
of two logarithmic vertex algebras is described in more detail in the following remark.

\begin{remark}\label{rem-tp}
Let $(V',\vac',T',Y',\N')$ and $(V'',\vac'',T'',Y'',\N'')$ be two logVAs. Then their \emph{tensor product} is the logVA
\begin{equation*}
(V'\otimes V'', \, \vac'\otimes \vac'', \, T'\otimes I+I\otimes T'', \, Y, \, \N'_{13}+\N''_{24}) \,,
\end{equation*}
where 
\begin{equation*}
Y(a\otimes b ,z)=Y'(a,z)\otimes Y''(b,z) \,, \qquad a\in V' \,, \; b\in V'' \,.
\end{equation*}
If $V'$ and $V''$ are conformal with conformal vectors $\om'$ and $\om''$, respectively, then $V'\otimes V''$ is also
conformal with 
\begin{equation*}
\om=\om'\otimes \vac''+\vac'\otimes \om'' 
\end{equation*}
and central charge equal the sum of central changes of $V'$ and $V''$.
\end{remark}

\subsection{Free bosons}\label{ex4.1}  
This example is a continuation of Example \ref{ex2.5}, but for completeness we present it again from scratch.
It is similar to the example of symplectic fermions from \seref{ex4.2}.

Let $V=\mathbb{C}[x_{0},x_{1},x_2,\dots]$ be the space of polynomials in infinitely many even variables. 
We consider the logarithmic field on $V$ given by  
\begin{equation*}
x(z) = \sum_{n=0}^\infty x_n z^n + \partial_{x_0} \ze + \sum_{n=1}^\infty \partial_{x_n} \frac{z^{-n}}{-n} \,,
\end{equation*}
which satisfies
\begin{equation*}
x(z_1)x(z_2)= {:}x(z_1)x(z_2){:} +\vartheta_{12}\, .
\end{equation*}
The derivative
\begin{equation*}
D_z x(z) = \sum_{n=1}^\infty n x_n z^{n-1} + \sum_{n=0}^\infty \partial_{x_n} z^{-n-1}
\end{equation*}
is a \emph{free boson} field with the singular OPE
\begin{equation*}
\bigl(D_{z_1} x(z_1)\bigr) \bigl(D_{z_2} x(z_2)\bigr) \sim \frac1{(z_1-z_2)^2} \,.
\end{equation*}

We let the vacuum vector $\vac=1$, and define the translation operator $T\in \End(V)$ by $T\vac=0$ and
\begin{equation*}
[T,x_n]=(n+1)x_{n+1}\, , \qquad n\ge0 \,.
\end{equation*}
Then one checks that
\begin{equation*}
[T,\partial_{x_n}]=-n\partial_{x_{n-1}} \, , \qquad n\ge0 \,,
\end{equation*}
so the field $x(z)$ is translation covariant.
We define a braiding map on $V$ by
\begin{equation*}
\N = -\partial_{x_0} \otimes \partial_{x_0} \,.
\end{equation*}
Then $[\N,T\otimes I]=0$; hence
\begin{equation*}
\NN = (\ad\otimes\ad)(\N) = -\ad(\partial_{x_0})\otimes \ad(\partial_{x_0})
\end{equation*}
is a braiding map on $\LFT(V)$ satisfying the conditions of \thref{l2.19}.

The subspace $\V=\Span\{x(z), I\} \subset\LFT(V)$ is $\NN$-local because
\begin{equation*}
\NN (x \otimes x) = -I \otimes I
\end{equation*}
and
\begin{align*}
e^{\vartheta_{12}\NN} x(z_1)x(z_2)
&=x(z_1)x(z_2)-\vartheta_{12} 
={:}x(z_1)x(z_2){:} \\
&={:}x(z_2)x(z_1){:} 
=e^{\vartheta_{21}\NN}x(z_2)x(z_1)\,.
\end{align*}
We can apply Theorems \ref{l2.19} and \ref{t3.14} to conclude that $V$ has the structure of a logVA.
The state-field correspondence $Y$ satisfies
\begin{equation*}
Y(x_0,z)=x(z) \,, \qquad Y(x_1,z)=D_z x(z) \,. 
\end{equation*}
The generating field $x(z)$ has the singular OPE
\begin{equation*}
x(z_1)x(z_2) \sim \log(z_1-z_2) \,.
\end{equation*}

The logVA $V$ is conformal with central charge $c=1$ and
\begin{equation*}
\om=\frac12 (x_1)^2 \,, \qquad Y(\om,z)=\frac12 {:} \bigl(D_z x(z)\bigr)^2 {:} \,. 
\end{equation*}
In particular,  $L_{0}=L_{0}^{(s)}+L_{0}^{(n)}$ where
\begin{equation*}
L^{(s)}_{0}=\sum_{n=1}^\infty nx_{n}\partial_{x_{n}}\, , \qquad L^{(n)}_{0}=\frac{1}{2}\partial_{x_{0}}^2 \,,
\end{equation*}
in agreement with \prref{pro3.11}.

%\subsection{Free boson logVA}\label{ex4.1h}  
Now consider $r$ commuting logarithmic fields $x^1(z),\dots,x^r(z)$ where each of them is a logarithmic free boson.
Then
\begin{equation*}
x^i(z_1)x^j(z_2)= {:}x^i(z_1)x^j(z_2){:} + \de_{i,j} \vartheta_{12}\,,
\end{equation*}
and they generate a logVA isomorphic to the $r$-th tensor power of the free boson logVA (see \reref{rem-tp}).
Explicitly, it is identified with the space of polynomials in even variables $x^i_n$ ($1\le i\le r$, $0\le n$).
Here we describe this example in a basis-independent way.

Start with an $r$-dimensional vector space $\h$ equipped with a non-degenerate symmetric bilinear form $(\cdot|\cdot)$. 
Recall that the \emph{free boson vertex algebra} $B_\h$ is defined as follows (see, e.g., \cite{K}).
As a vector space, $B_\h$ is the symmetric algebra $S(t^{-1}\h[t^{-1}])$. This is a highest-weight representation of the 
Heisenberg Lie algebra $\hat\h=\h[t,t^{-1}]\oplus\CC K$ with the Lie brackets
\begin{equation*}
[a_m,b_n] = m\de_{m,-n} (a|b) K \,, \quad [K,\hat\h]=0 \,, \qquad a,b\in\h \,, \; m,n\in\ZZ\,,
\end{equation*}
where $K$ acts as $I$ and we use the shorthand notation $a_m = at^m$.
The vertex algebra $B_\h$ is generated by the free boson fields
\begin{equation*}
a(z) = \sum_{m\in\ZZ} a_m z^{-m-1} %\,, \qquad a_m = at^m \,,
\end{equation*}
which have singular OPEs
\begin{equation*}
a(z_1)b(z_2) \sim \frac{(a|b)}{(z_1-z_2)^2} \,, \qquad a,b\in\h \,.
\end{equation*}
The vacuum vector is the highest-weight vector $1$.

In order to construct a logarithmic extension of $B_\h$, we first extend the
Heisenberg Lie algebra $\hat\h$ to the Lie algebra $\hat\h^{\log}$ defined as 
the abelian extension of $\hat\h$ by its module $\tilde\h\simeq\h$.
More explicitly, $\hat\h^{\log} = \hat\h\oplus\tilde\h$ as a vector space, 
the brackets in $\hat\h$ are as before, and
\begin{equation*}
[a_m,\tilde b] = \de_{m,0} (a|b) K \,, \quad [\tilde a,\tilde b] = 0 \,, \qquad a,b\in\h \,, \; m\in\ZZ\,,
\end{equation*}
where the elements in $\tilde\h$ are denoted as $\tilde h$ for $h\in\h$.
Now we introduce the \emph{logarithmic extension} of $B_\h$:
\begin{equation*}
B_\h^{\log} =  B_\h \otimes S(\tilde\h) \,.
\end{equation*}
Since $\tilde\h$ is an $\hat\h$-module, $S(\tilde\h)$ and hence $B_\h^{\log}$ are also $\hat\h$-modules.
If we let the abelian Lie algebra $\tilde\h$ act trivially on $B_\h$ and as multiplication on $S(\tilde\h)$,
then $B_\h^{\log}$ becomes an $\hat\h^{\log}$-module.

For every $a\in\h$, we define a logarithmic field
\begin{equation*}
\tilde{a}(z) = \tilde{a} + \ze a_0 + \sum_{n\ne0} a_n \frac{z^{-n}}{-n} \,,
\end{equation*}
so that
\begin{equation*}
D_z\tilde{a}(z) = a(z) \,, \qquad a\in\h \,.
\end{equation*}
These fields satisfy $(a,b\in\h)$:
\begin{equation*}
\tilde{a}(z_1) \tilde{b}(z_2) = {:}a(z_1)b(z_2){:} + (a|b) \vartheta_{12}\,,
\end{equation*}
which means that their singular OPEs are
\begin{equation*}
\tilde{a}(z_1) \tilde{b}(z_2) \sim (a|b) \log(z_1-z_2) \,.
\end{equation*}
By taking $D_{z_1}$ and $D_{z_2}$, we also get
\begin{equation*}
a(z_1) \tilde{b}(z_2) \sim \frac{(a|b)}{z_1-z_2}  \,, \qquad
\tilde{a}(z_1) b(z_2) \sim -\frac{(a|b)}{z_1-z_2}  \,.
\end{equation*}

Let $\{h^1,\dots,h^r\}$ be an orthonormal basis for $\h$ with respect to the bilinear form $(\cdot|\cdot)$.
Then the fields $\tilde{h}^i(z)$ can be identified with the above $x^i(z)$, and 
$B_\h^{\log} \simeq \CC[x^i_n]_{1\le i\le r,\, 0\le n}$ so that
\begin{equation*}
\tilde{h}^i = x^i_0 \,, \qquad h^i_n = \partial_{x^i_n} \,, \qquad h^i_{-m} = m x^i_m
\qquad (n\ge0 \,, \; m\ge1) \,.
\end{equation*}
The vacuum vector is $\vac=1\otimes1$, the translation operator $T$ on $B_\h^{\log}$ is defined by $T\vac=0$ and
\begin{equation*}
[T,\tilde{a}]=a_{-1} \,, \quad
[T,a_n]=-n a_{n-1}\, , \qquad a\in\h, \;\; n\in\ZZ \,,
\end{equation*}
and the braiding map is given by
\begin{equation*}
\N %= - \sum_{i=1}^r \partial_{x^i_0} \otimes \partial_{x^i_0} 
= - \sum_{i=1}^r h^i_0 \otimes h^i_0  \in \End(B_\h^{\log}) \otimes \End(B_\h^{\log}) \,.
\end{equation*}
As usual, we also have a conformal vector with central charge $r=\dim\h$,
\begin{equation*}
\om %= \frac12 \sum_{i=1}^r (x^i_1)^2 
= \frac12 \sum_{i=1}^r (h^i_{-1})^2 \otimes 1 \in B_\h^{\log} \,,
\end{equation*}
the same as the conformal vector in $B_\h$. However, unlike its action on $B_\h$,
the action of $L_0$ on $B_\h^{\log}$ is not semisimple. Its locally nilpotent part is
\begin{equation*}
L_0^{(n)} = \frac12 \sum_{i=1}^r (h^i_{0})^2 \,,
\end{equation*}
which is consistent with \prref{pro3.11}. 
In conclusion, $B_\h^{\log}$ is a conformal logVA such that
\begin{equation*}
Y(1\otimes\tilde{a}, z) = \tilde{a}(z) \,, \quad
Y(a_{-1} \otimes 1,z) = a(z) \,, \qquad a\in\h \,.
\end{equation*}
Observe that both $\N$ and $\om$ are independent of the choice of basis for $\h$.

\subsection{Lattice logVA}\label{ex4.5} 

Let $Q$ be an integral lattice, i.e., a free abelian group of rank $r$
equipped with a non-degenerate symmetric bilinear form $(\cdot|\cdot) \colon Q\times Q\to\ZZ$. 
We extend the bilinear form to the complex vector space $\h=\CC\otimes_\ZZ Q$ and identify 
$Q$ as a subset of $\h$.
We will continue to use the notation of \seref{ex4.1}. 

First, recall the definition of the \emph{lattice vertex algebra} $V_Q$ (see, e.g., \cite{K}).
There is a $2$-cocycle $\ep\colon Q \times Q \to \{\pm1\}$ such that
\begin{equation*}%\label{lat2}
\ep(\al,\al) = (-1)^{|\al|^2(|\al|^2+1)/2}  \,, \qquad |\al|^2 = (\al|\al) \,, \;\; \al\in Q \,.
\end{equation*}
It gives rise to the associative algebra $\CC_\ep[Q]$ with a basis
$\{ e^\al \}_{\al\in Q}$ and product
\begin{equation*}%\label{lat1}
%\CC_\ep[Q]=\Span_\CC \{ e^\al | \al\in Q \} \,, \qquad
e^\al e^\be = \ep(\al,\be) e^{\al+\be} \,,
\end{equation*}
which is a central extension of the multiplicative group algebra $\CC[Q]$ of the additive group $Q$.
Such a $2$-cocycle $\ep$ is unique up to equivalence and can be chosen
to be bimultiplicative. 
%It satifies
%\begin{equation*}%\label{lat22}
%\ep(\al,\be) \ep(\be,\al) = (-1)^{(\al|\be) + |\al|^2 |\be|^2}  \,, \qquad \al,\be\in Q \,.
%\end{equation*}

Then $V_Q$ is defined as the superspace $V_Q=B_\h\otimes\CC_\ep[Q]$,
where $B_\h$ is even and the parity of $e^\al$ is $|\al|^2$ mod $2\ZZ$.
The vacuum vector is $\vac=1\otimes e^0$.
The translation operator $T$ on $V_Q$ is given by 
\begin{equation*}
T(v \otimes e^\be) = (Tv) \otimes e^\be + (\be_{-1}v) \otimes e^\be \,, \qquad
v\in  B_\h \,, \; \be\in Q \,,
\end{equation*}
where $T$ acts on $B_\h$ as in \seref{ex4.1}.
We let the Heisenberg algebra act on $V_Q$ by
\begin{equation*}%\label{lat3}
a_n (v \otimes e^\be) = (a_n v) \otimes e^\be + \de_{n,0} (a|\be) v\otimes e^\be
\,, \qquad v\in  B_\h \,, \; \be\in Q \,.
\end{equation*}
Then the state-field correspondence on $V_Q$ is uniquely determined by the
generating fields:
\begin{align*}%\label{lat4}
Y(a_{-1}\otimes 1,z) &= \sum_{n\in\ZZ} a_{n} z^{-n-1} \,, \qquad a\in\lieh \,,
\\ %\label{lat5}
Y(1\otimes e^\al,z) &= 
\exp\Bigl( \sum_{n<0} \al_{n} \frac{z^{-n}}{-n} \Bigr) 
\exp\Bigl( \sum_{n>0} \al_{n} \frac{z^{-n}}{-n} \Bigr) \otimes e^\al z^{\al_0} \,,
\end{align*}
where $e^\al$ acts as multiplication on $\CC_\ep[Q]$ and $z^{\al_0} e^\be = z^{(\al|\be)} e^\be$.

Now we introduce the \emph{logarithmic extension} of $V_Q$:
\begin{equation*}
V_Q^{\log} =  B_\h^{\log} \otimes \CC_\ep[Q] = B_\h \otimes S(\tilde\h) \otimes \CC_\ep[Q] \,,
\end{equation*}
which we claim is a conformal logVA.
The vacuum vector in $V_Q^{\log}$ is $\vac=1\otimes 1\otimes e^0$, and the translation operator $T$ is given by 
\begin{equation*}
T(v \otimes e^\be) = (Tv) \otimes e^\be + (\be_{-1}v) \otimes e^\be \,, \qquad
v\in  B_\h^{\log} \,, \; \be\in Q \,.
\end{equation*}
Next, we let $\tilde\h$ act trivially on $\CC_\ep[Q]$, and thus obtain the structure of an $\hat\h^{\log}$-module $V_Q^{\log}$.
The (logarithmic) free boson fields 
\begin{equation*}
Y(a_{-1}\otimes 1 \otimes e^0,z) = a(z) \,, \quad
Y(1\otimes\tilde{a}\otimes e^0, z) = \tilde{a}(z) \,, \qquad a\in\h \,,
\end{equation*}
are as in \seref{ex4.1}. The vertex operators $Y(1\otimes 1\otimes e^\al,z)$ are also defined as above, but one has to take into account that
$\al_0$ now acts not only on the factor $\CC_\ep[Q]$ but also on $S(\tilde\h)$. Its action on $S(\tilde\h)$ is locally nilpotent;
hence $z^{\al_0}$ should be interpreted as $e^{\ze\al_0}$ when acting there (cf.\ \cite{BS2}). Explicitly,
\begin{align*}
Y(1&\otimes 1\otimes e^\al,z) \\
&= \Gamma_\al(z) := 
\exp\Bigl( \sum_{n<0} \al_{n} \frac{z^{-n}}{-n} \Bigr) 
\exp\Bigl( \sum_{n>0} \al_{n} \frac{z^{-n}}{-n} \Bigr) 
\otimes e^{\ze\al_0} \otimes e^\al z^{\al_0} \,.
\end{align*}
The logVA $V_Q^{\log}$ is generated by the fields $\tilde{a}(z)$ and $\Gamma_\al(z)$.
The singular OPEs between them are as in \cite{K} and in \seref{ex4.1}, except for the following new OPE:
\begin{equation*}
\tilde{a}(z_1) \Gamma_\be(z_2) \sim (a|\be) \Gamma_\be(z_2) \log(z_1-z_2) \,, \qquad a\in\h \,, \; \be\in Q\,.
\end{equation*}

The logVA $V_Q^{\log}$ is conformal of central charge $r=\rank Q$ with the same conformal vector as $B_\h$ (or $V_Q$ as well).
Note, however, that the linear operator $a_0$ for $a\in\h$ is locally nilpotent on $S(\tilde\h)$ and semisimple on $\CC_\ep[Q]$
(and zero on $B_\h$).
Hence,
\begin{equation*}
a_0^{(n)} = I \otimes a_0 \otimes I \,, \qquad a_0^{(s)} = I \otimes  I \otimes a_0 \,.
\end{equation*}
Let, as before, $\{h^1,\dots,h^r\}$ be an orthonormal basis for $\h$ with respect to the bilinear form $(\cdot|\cdot)$.
Then the locally nilpotent part of $L_0$ is
\begin{equation*}
L_0^{(n)} = \frac12 \sum_{i=1}^r \Bigl( (h^i_{0})^2 - \bigl(h_{0}^{i\,(s)}\bigr)^2 \Bigr)
= \frac12 \sum_{i=1}^r \Bigl( h^{i\,(n)}_{0}  h^{i\,(n)}_{0} + 2 h^{i\,(n)}_{0} h^{i\,(s)}_{0} \Bigr).
\end{equation*}
It follows from \prref{pro3.11} that the braiding map on $V_Q^{\log}$ is given by
\begin{equation*}
\N = -\sum_{i=1}^r \Bigl( h^{i\,(n)}_{0} \otimes h^{i\,(n)}_{0} + h^{i\,(n)}_{0} \otimes h^{i\,(s)}_{0} + h^{i\,(s)}_{0} \otimes h^{i\,(n)}_{0} \Bigr).
\end{equation*}
This is consistent with the above OPEs and the general OPE formula \eqref{ope5}.

\subsection{Gurarie--Ludwig's LCFT}\label{e4.7}

In this subsection, we interpret an example of a logarithmic conformal field theory (LCFT) due to Gurarie--Ludwig \cite{GL1,GL2,G2} as an example of a conformal logVA $V$.
We will first present the LCFT and derive from it the properties that $V$ should satisfy; then we will take these properties as a definition and prove rigorously that
indeed we have a logVA. In the process, we will discover relations among the generators of $V$, which were previously unknown.

%In \cite{GL1,GL2,G2}, the authors studied an LCFT, in which the 
The superspace of states $V$ is a representation of
the Virasoro Lie algebra with central charge $c=0$. As usual, the conformal vector $\om\in V_{\bar0}$ corresponds to the Virasoro field 
\begin{equation}\label{logex0}
L(z) = Y(\om,z) = \sum_{n\in\ZZ} L_n z^{-n-2}
\end{equation}
(see \eqref{Virfield}),
which is denoted as $T(z)$ in \cite{G2}.
It is assumed that $V$ contains odd vectors $\xi,\bar\xi \in V_{\bar1}$ such that the corresponding fermionic fields
\begin{equation*}
\xi(z) = Y(\xi,z) \,, \qquad \bar\xi(z) = Y(\bar\xi,z)
\end{equation*}
are \emph{primary} of conformal weight $2$, i.e.,
\begin{equation}\label{logex1a}
L_{0}\xi=2\xi\, , \quad L_{0}\bar{\xi}=2\bar{\xi}\, ,  \quad L_{n}\xi=L_{n}\bar{\xi}=0\,  , \qquad n>0 \,.
\end{equation}
Notice that $L(z)$ is also a primary field because $c=0$:
\begin{equation}\label{logex1b}
L_{0}\om=2\om\, , \quad L_{n}\om=0\,  , \qquad n>0 \,.
\end{equation}
Finally, it is assumed that $\om$ has a \emph{logarithmic partner} $\ell\in V_{\bar0}$ %(denoted $t$ in \cite{G2}) 
such that
\begin{equation}\label{logex1c}
L_{0}\ell=2\ell+\om\, , \quad L_{2}\ell=\be\vac \,, \quad L_1\ell=L_{n}\ell=0\,  , \qquad n>2 \,,
\end{equation}
for some $\be\in\CC$, which is a parameter of the theory (in \cite{G2}, $\ell$ is denoted as $t$ and $\be$ as $b$).
Here, as before, $\vac$ is the vacuum vector.
The above equations imply that $L_0$ is not semisimple on the subspace $\Span\{\om,\ell,\xi,\bar\xi\} \subset V$,
but instead $L_{0}=L_{0}^{(s)}+L_{0}^{(n)}$ with
\begin{equation*}
L_{0}^{(s)}=2I \,, \qquad L_{0}^{(n)} \ell=\om \,, \qquad L_{0}^{(n)} \om=L_{0}^{(n)}\xi = L_{0}^{(n)} \bar\xi = 0\,.
\end{equation*}

As usual, the above action of the Virasoro Lie algebra on the states is equivalent to specifying the singular OPEs between $L(z)$ and the fields corresponding to the states:
\begin{equation}\label{logex-ope1}
\begin{split}
L(z)L(0) &\sim \frac{2{L}(0)}{z^{2}}+\frac{T {L}(0)}{z}\, , \\[6pt]
L(z)\ell(0) &\sim \frac{\be}{z^{4}}+\frac{2\ell(0)+L(0)}{z^{2}}+\frac{T\ell(0)}{z}\, , \\[6pt]
L(z){\xi}(0) &\sim \frac{2{\xi}(0)}{z^{2}}+\frac{T {\xi}(0)}{z}\, , \qquad
L(z)\bar{\xi}(0) \sim \frac{2\bar{\xi}(0)}{z^{2}}+\frac{T \bar{\xi}(0)}{z}\,.
\end{split}
\end{equation}
Here, as in \cite{G2}, we are using the shorthand form \eqref{ope6} of the OPE \eqref{ope5}.
Its rigorous meaning is that, after applied to the vacuum vector, we have $L(0)=\om$, $\xi(0)=\xi$, etc.
We also use our previous notation $T=L_{-1}$ for the translation operator, and set $\ell(z)=Y(\ell,z)$.

The remaining singular OPEs were determined in \cite{GL1,GL2,G2} to be as follows:
\begin{equation}\label{logex-ope2}
\begin{split}
\xi(z)\bar{\xi}(0) &\sim \frac{\be}{2z^{4}}+\frac{\ell(0)}{z^{2}}+ \frac{T\ell(0)}{2z} + \frac18 L(z)L(0) + \frac12 L(z)L(0) \log(z) \,,\\[6pt]
\ell(z){\xi}(0) &\sim \frac{T\xi(0)}{2z} + \frac14 L(z)\xi(0) - L(z)\xi(0)\log(z)\, ,\\[6pt]
\ell(z)\bar{\xi}(0) &\sim \frac{T\bar{\xi}(0)}{2z} + \frac14 L(z)\bar{\xi}(0) - L(z)\bar{\xi}(0) \log(z) \, ,\\[6pt]
\ell(z)\ell(0) &\sim \frac{-2\be\log(z)}{z^{4}}+\frac{\ell(0)(1-4\log(z))-L(0)(\log(z)+2\log^{2}(z))}{z^{2}} \\[6pt]
&+\frac{T\ell(0)(1-4\log(z))-TL(0)(\log(z)+2\log^{2}(z))}{2z}\, .
\end{split}
\end{equation}
The OPEs of opposite products are derived from the above using skew-symmetry (\prref{pro3.16}), and any other singular OPEs are assumed zero.
In the right-hand sides of the first three equations, the products $L(z)L(0)$, $L(z)\xi(0)$ and $L(z)\bar{\xi}(0)$ should be further expanded according to their OPEs,
which gives the following expressions:
\begin{align*}
\xi(z)\bar{\xi}(0) &\sim \frac{\be}{2z^{4}}+\frac{4\ell(0)+L(0)+4L(0)\log(z)}{4z^{2}} \\[6pt]
&+ \frac{4T\ell(0)+TL(0)+4TL(0)\log(z)}{8z} \,,\\[6pt]
\ell(z){\xi}(0) &\sim \frac{\xi(0)}{2z^2} + \frac{3T\xi(0)}{4z}  - \frac{2{\xi}(0) \log(z)}{z^{2}}+\frac{T {\xi}(0) \log(z)}{z} \,,\\[6pt]
\ell(z)\bar{\xi}(0) &\sim \frac{\bar\xi(0)}{2z^2} + \frac{3T\bar\xi(0)}{4z}  - \frac{2{\bar\xi}(0) \log(z)}{z^{2}}+\frac{T {\bar\xi}(0) \log(z)}{z} \,.
\end{align*}
Conversely, the above OPE of $\ell(z)\ell(0)$ is already in expanded form, while in compact form it can be written as:
\begin{equation*}
\ell(z)\ell(0) \sim \frac{\ell(0)}{z^{2}} +\frac{T\ell(0)}{2z} -4\xi(z)\bar{\xi}(0)\log(z) + L(z)L(0) \log^{2}(z) \, .
\end{equation*}

\begin{remark}\label{rem-ss}
Let us show, for example, how to apply the skew-symmetry relation from \prref{pro3.16} to find the singular OPE of $\xi(z) \bar{\xi}(0)$. It is given by the singular part of
\begin{align*}
-\frac{e^{zT} \be}{2(-z)^{4}}&-\frac{e^{zT} \bigl(4\ell(0)+L(0)+4L(0)\log(z)\bigr)}{4(-z)^{2}} \\[6pt]
&- \frac{e^{zT} \bigl(4T\ell(0)+TL(0)+4TL(0)\log(z)\bigr)}{8(-z)} \,,
\end{align*}
which simplifies to $\bar{\xi}(z)\xi(0) \sim -\xi(z)\bar{\xi}(0)$.
\end{remark}

Let us assume that $V$ is a logVA generated by the logarithmic fields $L(z)$, $\ell(z)$, $\xi(z)$ and $\bar\xi(z)$.
Since all these fields are of conformal weight $2$, we will label their modes similarly to those of $L(z)$:
\begin{equation}\label{logex01}
\begin{split}
X(\xi,z)&=\xi(z)\big|_{\ze=0}=\sum_{n\in \ZZ}\xi_{n}z^{-n-2}\,, \\
X(\bar\xi,z)&=\bar{\xi}(z)\big|_{\ze=0}=\sum_{n\in \ZZ}\bar{\xi}_{n}z^{-n-2}\,,\\
X(\ell,z)&=\ell(z)\big|_{\ze=0}=\sum_{n\in \ZZ}\ell_{n}z^{-n-2} \,.
\end{split}
\end{equation}
This means that we have a shift
\begin{equation}\label{logex2}
\om_{(n+\N)} = L_{n-1} \,, \quad
a_{(n+\N)} = a_{n-1} \,, \qquad n\in\ZZ \,, \;\; a\in\{\ell,\xi,\bar\xi\} \, .
\end{equation}
According to \eqref{ope6}, the above singular OPEs encode the action of the braiding map $\N$ and the $(n+\N)$-th products of the generators of $V$.
(However, $\N$ is not determined uniquely from the OPEs.)
Setting $\log(z)=0$ in the OPEs, we obtain the action of the modes:
\begin{equation}\label{logex3a}
\begin{aligned}
\xi_2\bar\xi &= \frac{\be}2\vac \,, &  \xi_1\bar\xi &= 0 \,, &  \xi_0\bar\xi &= \frac14(4\ell+\om) \,, & \xi_{-1} \bar\xi &= \frac{T}8(4\ell+\om) \,, \\[6pt]
%\bar\xi_2\xi &= -\frac{\be}2\vac \,, & \bar\xi_0\xi &= -\frac14(4\ell+\om) \,, & \bar\xi_{-1} \xi &= -\frac{T}8(4\ell+\om) \,, 
\ell_0\xi &= \frac12 \xi \,, & \ell_{-1}\xi &= \frac34 T\xi \,, &  \ell_0\bar\xi &=\frac12 \bar\xi \,, & \ell_{-1}\bar\xi &= \frac34 T\bar\xi \,, \\[6pt]
\ell_0\ell &=  \ell \,, & \ell_{-1}\ell &= \frac12 T\ell \,,
\end{aligned}
\end{equation}
and
\begin{equation}\label{logex3b}
\xi_m\bar\xi = 0 \,, \quad \ell_n a = 0 \quad\text{for all}\quad m>2 \,, \; n>0 \,, \; a\in\{\ell,\xi,\bar\xi\} \,.
\end{equation}
All other products $a_nb$ with $a,b\in\{\ell,\xi,\bar\xi\}$, $n\ge-1$, either follow from the skew-symmetry relation or are zero.

We will find the braiding map $\N$ on $V$ by utilizing \prref{pro3.11}. It was observed in \cite{GL1,GL2,G2} that the above OPEs of generating fields admit odd derivations
$\eta$ and $\bar\eta$ if we define their action on $\om$, $\ell$, $\xi$, $\bar\xi$ according to the diagram
\begin{equation}\label{logex4}
\begin{gathered}
\xymatrix{
                           & \ell \ar[ld]_{-{\eta}/{2}} \ar[dr]^{{{\bar{\eta}}/{2}}} &                       \\ 
\xi \ar[dr]_{\bar{\eta}} &                               & \bar{\xi} \ar[dl]^{\eta} \quad.\\
                           & \om
}
\end{gathered}
\end{equation}
Thus, we obtain odd derivations $\eta,\bar\eta\in\Der(V)$ of the logVA $V$ satisfying
%Moreover, we have %$[\eta,\bar\eta]=0$, i.e., $\eta\bar\eta=-\bar\eta\eta$.
\begin{equation}\label{etabareta}
\eta^2=\bar\eta^2=0 \,, \qquad \eta\bar\eta=-\bar\eta\eta \,.
\end{equation}

\begin{remark}
The action of $\eta$ and $\bar\eta$ can be extended to an action of the 
Lie superalgebra $\mathfrak{gl}(1|1)$, which has a basis $\{E,N,\eta,\bar{\eta}\}$ and the non-vanishing brackets 
\begin{equation*}
[N,\eta]=\eta \, ,\qquad  [N,\bar{\eta}]=-\bar{\eta}\, ,\qquad [\eta,\bar{\eta}]=E\, .
\end{equation*}
In our case, $N\om=N\ell=0$, $N\xi=\xi$, $N\bar\xi=-\bar\xi$, and $E=0$.
\end{remark}

Now observe that $L^{(n)}_{0}=\frac{1}{2} \eta\bar\eta$. Hence, by \prref{pro3.11},
\begin{equation}\label{logex5}
\N = \frac{1}{2}(\bar{\eta} \otimes \eta -  \eta \otimes \bar\eta) \,.
\end{equation}
From here, using \eqref{logf-nact}, we compute the action of $\N$ on tensor products of generators and obtain:
\begin{align*}
\N(\bar\xi\otimes\xi) &= - \N(\xi\otimes\bar\xi) = \frac12 \om\otimes\om \,, &
\N(\ell\otimes\xi) &= \N(\xi\otimes\ell) = \xi\otimes\om \,, \\
\N(\ell\otimes\ell) &= 2\xi\otimes\bar\xi-2\bar\xi\otimes\xi \,, &
\N(\ell\otimes\bar\xi) &= \N(\bar\xi\otimes\ell) = \bar\xi\otimes\om \,.
\end{align*}
%Here, we have used the fact that $\N$ is symmetric. 
The action on all other tensors not listed here is zero.
In particular, $\N(\om\otimes a)=0$ for all $a\in\{\om,\ell,\xi,\bar\xi\}$, as it should by \deref{d3.10}\eqref{3.17-1}.
We also have
\begin{equation*}
\N^2(\ell\otimes\ell) = -2 \om\otimes\om \,,
\end{equation*}
and $\N^2=0$ on all other tensors.

\begin{lemma}\label{lem4.4}
Given the above\/ $(n+\N)$-th products\/ \eqref{logex2}--\eqref{logex3b} of generators\/ $\ell$, $\xi$, $\bar\xi$ and action of\/ $\N$ by\/ \eqref{logex5}, the singular OPEs \eqref{logex-ope1}, \eqref{logex-ope2}
of generating fields agree with the general formula \eqref{ope6}. 
\end{lemma}
\begin{proof}
Since the terms of the OPE of $Y(a,z)Y(b,0)$ given by \eqref{ope6} correspond to the action $Y(a,z)b$ given by \eqref{nprod-7}, it is enough to compute the latter.
Using \eqref{logva3}, we find:
\begin{align*}
\xi(z)\bar{\xi} &= X(\xi,z)\bar\xi + \frac{\ze}2 L(z)\om \,,\\
\ell(z){\xi} &= X(\ell,z)\xi - \ze X(\xi,z)\om \, , \qquad \ell(z){\bar\xi} = X(\ell,z)\bar\xi - \ze X(\bar\xi,z)\om \, , \\
\ell(z)\ell &= X(\ell,z)\ell -2\ze\bigl(X(\xi,z)\bar{\xi}-X(\bar{\xi},z)\xi\bigr) - \ze^2 L(z)\om \, .
\end{align*}
Due to skew-symmetry (\prref{pro3.16}), we have as in \reref{rem-ss},
\begin{equation*}
X(\xi,z)\om \sim L(z)\xi \,, \qquad X(\bar\xi,z)\om \sim L(z)\bar\xi \,, \qquad X(\bar{\xi},z)\xi \sim -X(\xi,z)\bar{\xi} \,,
\end{equation*}
where $\sim$ here means equality of the singular parts that involve negative powers of $z$. Therefore,
\begin{align*}
\ell(z){\xi} &\sim X(\ell,z)\xi - \ze L(z)\xi \, , \qquad \ell(z){\bar\xi} \sim X(\ell,z)\bar\xi - \ze L(z)\bar\xi \, , \\
\ell(z)\ell &\sim X(\ell,z)\ell -4\ze X(\xi,z)\bar{\xi} - \ze^2 L(z)\om \, .
\end{align*}
The rest of the proof follows by construction.
\end{proof}

As in the proof of \leref{lem4.4}, from \eqref{logf-nact}, \eqref{logva3}, \eqref{logex4} and \eqref{logex5}, we can derive the following expansions of logarithmic fields:
\begin{equation}\label{logex6}
\begin{split}
\xi(z) &= X(\xi,z) + \frac{\ze}2 L(z)\eta \,,\qquad
\bar{\xi}(z) = X(\bar\xi,z) - \frac{\ze}2 L(z)\bar\eta \,,\\
\ell(z) &= X(\ell,z) -\zeta X(\xi,z)\bar{\eta}-\zeta X(\bar{\xi},z)\eta+\frac{\ze^2}{2} L(z) \bar{\eta}\eta \, .
\end{split}
\end{equation}

Our next goal is to determine the relations satisfied by the modes of generating fields. We summarize the answer in the next lemma.

\begin{lemma}\label{lem4.5}
Let\/ $V$ be a logVA with elements\/ $\om,\ell\in V_{\bar0}$ and\/ $\xi,\bar\xi\in V_{\bar1}$, whose\/ $(n+\N)$-th products are given by\/
\eqref{logex1a}--\eqref{logex1c}, \eqref{logex2}--\eqref{logex3b}. Suppose that there are\/ $\eta,\bar\eta\in\Der(V)_{\bar1}$ satisfying\/
\eqref{etabareta} and acting according to\/
\eqref{logex4}, such that the braiding map\/ $\N$ is given by\/ \eqref{logex5}. Then the following relations hold
for all\/ $m,k\in\ZZ{:}$
\allowdisplaybreaks
\begin{align*}
[\eta, a_{m}]&=(\eta a)_{m}\,, \quad [\bar{\eta}, a_{m}]=(\bar{\eta} a)_{m}\,,  \qquad a\in \{L,\ell,,\xi,\bar{\xi}\} \, , \\[3pt]
[L_m,L_k] &= (m-k)L_{m+k} \,, \\[3pt]
[L_{m}, {\xi}_{k}]&=(m-k){\xi}_{m+k}+\frac{1}{2}L_{m+k}\eta\, ,\quad
[L_{m}, \bar{\xi}_{k}]=(m-k)\bar{\xi}_{m+k}-\frac{1}{2}L_{m+k}\bar{\eta}\, ,\\[3pt]
[L_{m},\ell_{k}]&=(m-k)\ell_{m+k} \!+\! (m+1)L_{m+k} \!-\! \xi_{m+k}\bar{\eta} \!-\! \bar{\xi}_{m+k}\eta \!+\! \delta_{m,-k} \frac{m^{3}-m}{6} \be\, , \\[3pt]
[\xi_{m},\xi_{k}]&=\xi_{m+k}\eta\, ,\qquad [\bar{\xi}_{m},\bar{\xi}_{k}]=-\bar{\xi}_{m+k}\bar{\eta}\, ,\\[3pt]
[\xi_{m},\bar{\xi}_{k}]&=(m-k)\Bigl(\frac{1}{8}L_{m+k}+\frac{1}{2}\ell_{m+k}\Bigr)+\delta_{m,-k}\frac{m^{3}-m}{12} \be \, \\[3pt]
&+\frac{1}{2}\bar{\xi}_{m+k}\eta-\frac{1}{2}\xi_{m+k}\bar{\eta} -\frac{1}{2}\sum_{j\geq 1}\frac{1}{j} \bigl(L_{m-j}L_{k+j}-L_{k-j}L_{m+j}\bigr) ,\\[3pt]
[\ell_m , \xi_k ]&=-\frac{1}{4}(3k+m+4) \xi_{m+k} + \Bigl(\ell_{m+k}+\frac{5}{8}L_{m+k}\Bigr)\eta + \delta_{m,-k}\frac{m}{4} \be\eta \\[3pt]
& +\sum_{j\geq 1}\frac{1}{j} \bigl(\xi_{m-j}L_{k+j}-L_{k-j}\xi_{m+j}\bigr) ,\\[3pt]
[\ell_{m}, \bar{\xi}_{k}] &= -\frac{1}{4}(3k+m+4) \bar{\xi}_{m+k} - \Bigl(\ell_{m+k}+\frac{5}{8}L_{m+k}\Bigr)\bar\eta - \delta_{m,-k}\frac{m}{4} \be\bar{\eta} \\[3pt]
& +\sum_{j\geq 1}\frac{1}{j} \bigl(\bar{\xi}_{m-j}L_{k+j}-L_{k-j}\bar{\xi}_{m+j}\bigr) ,\\[3pt]
[\ell_{m}, \ell_{k}] &= \frac{1}{2}(m-k) \ell_{m+k} + \delta_{m,-k}\frac{m}{2} \be\bar{\eta}\eta\\[3pt]
& +2\sum_{j\geq 1}\frac{1}{j} \bigl(\xi_{m-j}\bar\xi_{k+j}+\bar{\xi}_{k-j}\xi_{m+j} -\bar{\xi}_{m-j}\xi_{k+j}-\xi_{k-j}\bar\xi_{m+j} \bigr)\\[3pt]
& +2\sum_{j\geq 2}\frac{h(j-1)}{j}\bigl( L_{m-j}L_{k+j}-L_{k-j}L_{m+j}\bigr),
\end{align*}
where\/ $h(j):=1+\frac{1}{2}+\cdots +\frac{1}{j}$ are the harmonic numbers.
\end{lemma}
\begin{proof}
The identities for $[\eta, a_{m}]$ and $[\bar\eta, a_{m}]$ simply express the fact that $\eta$ and $\bar\eta$ are derivations.
All other formulas follow from the Borcherds identity \eqref{borcherds2} with $n=0$ applied to $a\otimes b\otimes v$,
where $a,b\in\{\om,\ell,\xi,\bar\xi\}$ and $v\in V$ is arbitrary. 
In the case when $\N(a\otimes b)=0$, the left-hand side of \eqref{borcherds2} reduces to the commutator formula from \coref{corbor2}.

Let us consider, for example, $a=b=\xi$. From \coref{corbor2}, we get
\begin{equation*}
\bigl[\xi_{(m+\N)},\xi_{(k+\N)}\bigr] = \sum_{j\ge0 }\mu_{(m+k-j)}( \mu_{(j)}\otimes I) \binom{m+\N_{13}}{j} (\xi\otimes\xi\otimes v) \,,
\end{equation*}
and from \eqref{logf-nact}, \eqref{logex4} and \eqref{logex5}, we have
\begin{equation*}
\N_{13} (\xi\otimes\xi\otimes v) = \frac12 \om\otimes\xi\otimes\eta v \,, \qquad
\N_{13}^2 (\xi\otimes\xi\otimes v) = 0 \,.
\end{equation*}
Then we apply $\mu_{(j)}\otimes I$ to either $\xi\otimes\xi\otimes v$ or $\om\otimes\xi\otimes\eta v$, using the $(j+\N)$-th products
given by \eqref{logex1a}--\eqref{logex1c}, \eqref{logex2}--\eqref{logex3b}. Observe that $\xi_{(j+\N)}\xi=0$ for $j\ge0$, and
$\om_{(j+\N)}\xi=0$ for $j\ge2$, while $\om_{(1+\N)}\xi=2\xi$. Hence, the above sum over $j$ reduces to the terms with $j=0$ and $j=1$:
\begin{align*}
\mu_{(m+k)}( &\mu_{(0)}\otimes I)(\xi\otimes\xi\otimes v)
+\mu_{(m+k-1)}( \mu_{(1)}\otimes I) (m+\N_{13}) (\xi\otimes\xi\otimes v) \\
&= \frac12 \mu_{(m+k-1)}( \mu_{(1)}\otimes I) (\om\otimes\xi\otimes\eta v) \\
&= \frac12 \mu_{(m+k-1)} (2\xi\otimes\eta v)
= \xi_{(m+k-1+\N)} (\eta v) \,.
\end{align*}
This proves that $[\xi_{m},\xi_{k}]=\xi_{m+k}\eta$ after we make the shift \eqref{logex2}.

As another example, let us consider $a=b=\ell$. In this case, we have:
\begin{align*}
\N_{12} (\ell\otimes\ell\otimes v) &= 2\xi\otimes\bar\xi\otimes v - 2\bar\xi\otimes\xi\otimes v \,, \\
\N_{12}^2 (\ell\otimes\ell\otimes v) &= -2 \om\otimes\om\otimes v \,, \\
\N_{13} (\ell\otimes\ell\otimes v) &= \xi\otimes\ell\otimes\bar\eta v + \bar\xi\otimes\ell\otimes\eta v  \,, \\
\N_{13}^2 (\ell\otimes\ell\otimes v) &= \om\otimes\ell\otimes\bar\eta\eta v \,,
\end{align*}
and the higher powers are zero. Also, notice that for $j\ge1$,
\begin{equation*}
(-1)^j \binom{\N_{12}}{j} %= \frac{(-1)^j}{j!} \N_{12} (\N_{12}-1) \cdots (\N_{12}-j+1) 
= -\frac1j \N_{12} + \frac{h(j-1)}{j} \N_{12}^2 +\cdots
\end{equation*}
where the dots signify higher powers of $\N_{12}$ and we let $h(0):=0$.
Hence, the first sum in the left-hand side of \eqref{borcherds2} with $n=0$ applied to $\ell\otimes\ell\otimes v$ is
\begin{align*}
\mu_{(m)} &(I\otimes \mu_{(k)}) (\ell\otimes\ell\otimes v)
-\sum_{j\ge1} \frac1j \mu_{(m-j)}(I\otimes \mu_{(k+j)}) \N_{12} (\ell\otimes\ell\otimes v) \\
&+\sum_{j\ge2} \frac{h(j-1)}{j} \mu_{(m-j)}(I\otimes \mu_{(k+j)}) \N_{12}^2 (\ell\otimes\ell\otimes v) \\
&= \ell_{(m+\N)} (\ell_{(k+\N)} v) -2 \sum_{j\ge1} \frac1j \Bigl( \xi_{(m+\N)} (\bar\xi_{(k+\N)} v) - \bar\xi_{(m+\N)} (\xi_{(k+\N)} v) \Bigr) \\
&-2\sum_{j\ge2} \frac{h(j-1)}{j} \om_{(m+\N)} (\om_{(k+\N)} v) \,.
\end{align*}
Since $P(\ell\otimes\ell)=\ell\otimes\ell$, the second sum in the left-hand side of \eqref{borcherds2} is obtained from the first sum by swapping
$k$ and $m$. 

To compute the right-hand side of \eqref{borcherds2}, we first observe that $\om_{(j+\N)}\ell=0$ for $j\ge4$
due to \eqref{logex1c}. Hence, we only need to find the coefficient of $\N_{13}^2$ in the binomials
$\binom{m+\N_{13}}{j}$ for $j=0,1,2,3$, which gives $0$, $0$, $1/2$ and $(m-1)/2$, respectively.
However, $\om_{(2+\N)}\ell=0$, so only $j=3$ contributes.
Second, to find the contributions involving $\N_{13}^1$, we need to determine the products 
$\xi_{(j+\N)}\ell$ and $\bar\xi_{(j+\N)}\ell$. For this, as before, we use the skew-symmetry relation from \prref{pro3.16}:
\begin{align*}
X(\xi,z)\ell &= e^{zT} X(\ell,-z)\xi \sim e^{zT} \Bigl( \frac12 \xi (-z)^{-2} + \frac34 T\xi (-z)^{-1} \Bigr) \\
&\sim \frac12 \xi z^{-2} - \frac14 T\xi z^{-1}
\end{align*}
where $\sim$ denotes equality modulo terms with non-negative powers of $z$.
The same calculation applies for $\bar\xi$. We obtain
\begin{align*}
\xi_{(0+\N)} \ell = - \frac14 T\xi \,, \quad \xi_{(1+\N)} \ell = \frac12 \xi \,, \quad
\bar\xi_{(0+\N)} \ell = - \frac14 T\bar\xi \,, \quad \bar\xi_{(1+\N)} \ell = \frac12 \bar\xi \,,
\end{align*}
and the higher products are zero. We see that we may have a contribution $\N_{13}^1$ from $\binom{m+\N_{13}}{j}$
only for $j=1$. 

Putting these together, we find that the right-hand side of \eqref{borcherds2} is:
\begin{align*}
\sum_{j\ge0} \binom{m}{j} &\mu_{(m+k-j)}( \mu_{(j)}\otimes I) (\ell\otimes\ell\otimes v) \\[6pt]
&+ \mu_{(m+k-1)}( \mu_{(1)}\otimes I) \N_{13} (\ell\otimes\ell\otimes v) \\[6pt]
&+ \frac{m-1}2 \mu_{(m+k-3)}( \mu_{(3)}\otimes I) \N_{13}^2 (\ell\otimes\ell\otimes v) \,.
\end{align*}
Now we compute each of the terms in this sum. By \eqref{logex2}--\eqref{logex3b} and \eqref{Tanb}, the first one is
\begin{align*}
\sum_{j\ge0} &\binom{m}{j} (\ell_{(j+\N)}\ell)_{(m+k-j+\N)} v 
= \frac12 (T\ell)_{(m+k+\N)} v + m \ell_{(m+k-1+\N)} v \\
&= -\frac12 (-m+k+\N)\ell_{(m+k-1+\N)} v \\
&= \frac12 (m-k)\ell_{(m+k-1+\N)} v -\frac12 \xi_{(m+k-1+\N)} (\bar\eta v) -\frac12 \bar\xi_{(m+k-1+\N)} (\eta v) \,.
\end{align*}
The remaining two terms give
\begin{align*}
\mu_{(m+k-1)} &( \mu_{(1)}\otimes I) (\xi\otimes\ell\otimes\bar\eta v + \bar\xi\otimes\ell\otimes\eta v) \\
&= (\xi_{(1+\N)} \ell)_{(m+k-1+\N)} (\bar\eta v) + (\bar\xi_{(1+\N)} \ell)_{(m+k-1+\N)} (\eta v) \\
&= \frac12 \xi_{(m+k-1+\N)} (\bar\eta v) + \frac12 \bar\xi_{(m+k-1+\N)} (\eta v) \,,
\end{align*}
and
\begin{align*}
\frac{m-1}2 &\mu_{(m+k-3)}( \mu_{(3)}\otimes I) (\om\otimes\ell\otimes\bar\eta\eta v) 
= \frac{m-1}2 (\om_{(3)}\ell)_{(m+k-3+\N)} (\bar\eta\eta v) \\
&=\frac{m-1}2 \be \vac_{(m+k-3+\N)} (\bar\eta\eta v)
= \delta_{m+k,2} \frac{m-1}2 \be\bar\eta\eta v \,.
\end{align*}
Taking into account the shift \eqref{logex2}, we obtain the formula for $[\ell_{m}, \ell_{k}]$.
This completes the proof of the lemma.
\end{proof}

\begin{remark}\label{remnew}
We would like to stress that, even though the singular OPE $\xi(z)\xi(0) \sim 0$ is trivial, the modes of $\xi$ do not commute with each other.
This is a new phenomenon in the logarithmic case, which does not occur for ordinary vertex algebras. Similarly, although the singular OPE of $\ell(z)\ell(0)$
contains only the fields $L(z)$ and $\ell(z)$, the commutator $[\ell_{m}, \ell_{k}]$ involves the modes of $\xi$ and $\bar\xi$ as well.
In \cite{G2}, it is claimed that $\om$ and $\ell$ generate an LCFT; however, we were unable to construct a logVA generated only by them.
\end{remark}

After we have derived the properties of a hypothetical logVA $V$ corresponding to Gurarie--Ludwig's LCFT, we now reverse course and use these properties to define $V$
and prove that it is indeed a conformal logVA. 
To this end, we introduce a unital associative superalgebra $\A$ with a set of generators $L_{n}$, $\ell_{n}$, $\xi_{n}$, $\bar{\xi}_{n}$, $\eta$, $\bar{\eta}$ ($n\in \ZZ$) subject to the
relations from \leref{lem4.5}. However, we need to make sense of the infinite sums in the relations. 

Fix $\beta\in\CC$.
Let $U$ be a vector superspace with a basis
$\{L_{n},\ell_{n},\xi_{n},\bar{\xi}_{n}, \eta,\bar{\eta} \}_{n\in \ZZ}$
where $L_n$ and $\ell_n$ are even and the other basis vectors are odd.
We denote by $\T=\bigoplus_{j\ge0} U^{\otimes j}$ 
the tensor algebra over $U$.
Both $U$ and $\T$ are $\ZZ$-graded by setting 
\begin{equation*}
\deg L_n = \deg \ell_n = \deg\xi_{n} = \deg\bar{\xi}_{n}= -n \,, \quad \deg 1=\deg\eta=\deg\bar{\eta}=0 \,.
\end{equation*}
We consider the completion $\overline\T$ of $\T$, which extends it by allowing infinite sums 
of the form
\begin{equation*}
\sum_{j\ge N} t_j \otimes u_j \quad\text{where}\quad 
t_j \in\T \,, \; u_j\in U \,,\; \deg t_j = m+j \,, \; \deg u_j =- j\,,
\end{equation*}
for some fixed $m,N\in\ZZ$. 

Clearly, $\overline\T$ is no longer an associative superalgebra under $\otimes$,
but its quotient $\A$ will be.
We define $\A$ as the vector superspace quotient $\overline\T / (\T\otimes\mathcal{R}\otimes\T)$,
where $\mathcal{R}\subset\overline\T$ is spanned by all relations from \leref{lem4.5} together with the relations \eqref{etabareta}.
The tensor product in $\overline\T$ induces a well-defined product in $\A$, because one can use the relations
in $\mathcal{R}$ to rewrite the product of two elements of $\overline\T$ as an element of $\overline\T$ mod
$\T\otimes\mathcal{R}\otimes\T$.

Therefore, $\A$ is a unital associative superalgebra, and it inherits a $\ZZ$-grading from $\overline\T$.
Consider the subspace $\overline\T_- \subset \overline\T$ consisting of all sums $\sum_{j\ge -1} t_j \otimes u_j$ as above with $N=-1$.
Let $\A_-$ be the image of $\overline\T_-$ in the quotient $\A$.
Then $\A_-$ is a left ideal of $\A$, because in the relations from \leref{lem4.5}, any commutator $[u_m,v_k]$ with
$u,v\in\{L,\ell,\xi,\bar\xi\}$ and $k,m\ge-1$, is expressed as an element of $\A_-$.

We let $V$ be the left $\A$-module $\A/\A_-$.
Then $V$ is generated as an $\A$-module by the even vector $\vac := 1+\A_-$, where $1\in\A$ is the identity element, with the properties that
\begin{equation}\label{logex7}
L_{n}\vac=\ell_{n}\vac=\xi_{n}\vac=\bar{\xi}_{n}\vac = \eta\vac=\bar{\eta}\vac=0\,,  \qquad n\geq -1\,.
\end{equation}

\begin{lemma}\label{lem4.6}
The superspace\/ $V$ is linearly spanned by\/ $\vac$ and monomials of the form
\begin{equation*}
a^{1}_{n_1} \cdots a^{r}_{n_r} \vac \quad\text{where}\quad
n_i \le -2 \,, \; a^i\in\{L,\ell,\xi,\bar{\xi}\} \,, \; 1\le i\le r \,.
\end{equation*}
Hence, $V$ is graded by\/ $\ZZ_+$ so that\/ $\deg\vac=0$.
In particular, for every\/ $v\in V$, we have
\begin{equation*}
L_{n}v=\ell_{n}v=\xi_{n}v=\bar{\xi}_{n}v =0\,,  \qquad n\gg 0\,.
\end{equation*}
\end{lemma}
\begin{proof}
This is similar to the case of Lie algebras. We will prove by induction on $r$ that every monomial $g_1\cdots g_r\vac$, where each $g_i$ is a generator of $\A$, 
can be rewritten as a linear combination of monomials with $r$ or fewer factors such that $\deg g_i>1$ for all $i$.
The claim is trivial for $r=0$ and $r=1$.

For $r\ge2$, by the inductive assumption, we can assume that $\deg g_i>1$ for $2\le i\le r$.
If $\deg g_1>1$, we are done, so let $\deg g_1\le 1$.
Then we use the relations of $\A$ to rewrite $g_1g_2$ as $(-1)^{p(g_1)p(g_2)} g_2g_1 + [g_1,g_2]$ and apply again
the inductive assumption for $g_1g_3 \cdots g_r\vac$. The bracket $[g_1,g_2]$ is either already in the desired form or involves terms in which the degree of the first
factor is strictly larger that $\deg g_1$. This means that repeating this process finitely many times will give us terms with $\deg g_1>1$.
\end{proof}

It follows from \leref{lem4.6} that all infinite sums in $\A$ become finite when acting on a fixed $v\in V$.
We define the elements of $V$:
\begin{equation}\label{logex8}
\om=L_{-2}\vac \,, \qquad \ell=\ell_{-2}\vac \,, \qquad \xi=\xi_{-2}\vac \,, \qquad \bar\xi=\bar{\xi}_{-2}\vac \,,
\end{equation}
and the logarithmic fields $L(z)$, $\ell(z)$, $\xi(z)$, $\bar\xi(z)$ on $V$ given by \eqref{logex0}, \eqref{logex01} and \eqref{logex6}.
We let the translation operator $T=L_{-1}$, and the braiding map $\N$ be as in \eqref{logex5}. As before, this gives rise to a braiding map
$\NN:=(\ad\otimes\ad)(\N)$ on $\LF(V)$.

\begin{theorem}\label{thm4.7}
The logarithmic fields\/ $L(z)$, $\ell(z)$, $\xi(z)$, $\bar\xi(z)$ are translation covariant and\/ $\NN$-local, and they generate
the structure of a logVA on\/ $V$ such that\/ $\eta,\bar\eta\in\Der(V)$. The singular OPEs among these fields are given by\/ \eqref{logex-ope1}, \eqref{logex-ope2}.
Moreover, $V$ is a conformal logVA with a conformal vector\/ $\om$ and central charge\/ $0$.
\end{theorem}
\begin{proof}
First, it is easy to see that $\N$ is indeed a braiding map on $V$ and $[T\otimes I,\N]=0$. Hence, $\NN$ is a braiding map on $\LF(V)$ satisfying
the conditions of \thref{l2.19}. Let us check the other assumptions of this theorem. From the relations in $\A$, we see that
\begin{align*}
[T,L_n] &= -(n+1)L_{n-1} \,, & &[T,\ell_{n}]=-(n+1)\ell_{n-1} \!-\! \xi_{n-1}\bar{\eta} \!-\! \bar{\xi}_{n-1}\eta \,, \\
[T, {\xi}_{n}]&=-(n+1){\xi}_{n-1}+\frac{1}{2}L_{n-1}\eta\,, &
&[T, \bar{\xi}_{n}]=-(n+1)\bar{\xi}_{n-1}-\frac{1}{2}L_{n-1}\bar{\eta}\,,
\end{align*}
which imply the translation covariance of $L(z)$, $\ell(z)$, $\xi(z)$ and $\bar\xi(z)$.
The relations also imply that $\eta$ and $\bar\eta$ are derivations of all $(n+\N)$-th products.
The completeness of $V$ follows from \leref{lem4.6}.

Using again the relations in $\A$ and \eqref{logex7}, \eqref{logex8}, we check the $(n+\N)$-th products
\eqref{logex1a}--\eqref{logex1c}, \eqref{logex2}--\eqref{logex3b}. Let us do just two examples here.
First,
\begin{align*}
\ell_0 \ell &= \ell_0 \ell_{-2} \vac = [\ell_{0}, \ell_{-2}] \vac \\[3pt]
&= \ell_{-2} \vac +2\sum_{j\geq 1}\frac{1}{j} \bigl(\xi_{-j}\bar\xi_{-2+j}+\bar{\xi}_{-2-j}\xi_{j} -\bar{\xi}_{-j}\xi_{-2+j}-\xi_{-2-j}\bar\xi_{j} \bigr) \vac \\[3pt]
& \qquad\;\;\;\;\;+2\sum_{j\geq 2}\frac{h(j-1)}{j}\bigl( L_{-j}L_{-2+j}-L_{-2-j}L_{j}\bigr) \vac \\[3pt]
&= \ell_{-2} \vac = \ell \,.
\end{align*}
Similarly,
\begin{align*}
\xi_{-1}\bar\xi &=\xi_{-1}\bar{\xi}_{-2}\vac =  [\xi_{-1},\bar{\xi}_{-2}]\vac=\frac{1}{8}L_{-3} \vac + \frac{1}{2}\ell_{-3}\vac \\[3pt]
&=\frac{1}{8}[T,L_{-2}] \vac + \frac{1}{2}[T,\ell_{-2}]\vac 
= \frac{1}{8}TL_{-2}\vac + \frac{1}{2}T\ell_{-2} \vac
%= \frac{T}8(4\ell+\om) \,.
= \frac{1}{8}T\om + \frac{1}{2}T\ell \,.
\end{align*}
The rest of the $(n+\N)$-th products are verified in the same way.

Then, by the proof of \leref{lem4.5},
the Borcherds identity \eqref{borcherds2} holds for $n=0$ and all $k,m\in\ZZ$ when applied to $a\otimes b\otimes v$,
where $a,b\in\{\om,\ell,\xi,\bar\xi\}$ and $v\in V$ is arbitrary. From \reref{rembor}, we obtain that
our generating fields span an $\NN$-local subspace of $\LF(V)$. By \thref{t3.14}, $V$ has the structure of a logVA.

Finally, the singular OPEs are as in \eqref{logex-ope1} and \eqref{logex-ope2}, due to \leref{lem4.4}.
In particular, the OPEs \eqref{logex-ope1} of $L(z)$ imply that $\om$ is a conformal vector with central charge $0$.
\end{proof}

\subsection*{Acknowledgments}
Preliminary versions of some of the results of this paper were presented by the first author at the
Mini-conference on Vertex Algebras at the University of Denver in October 2016
and the MIT Infinite Dimensional Algebra Seminar in May 2017.
He is grateful to the organizers for these opportunities, and to Thomas Creutzig, Pavel Etingof, Victor Gurarie, and Antun Milas
for valuable discussions. The second author would like to thank Marco Aldi and Reimundo Heluani for 
stimulating discussions on a related subject,
and North Carolina State University for the hospitality during July 2021 when this paper was completed.
The first author was supported in part by a Simons Foundation grant 584741.

%%%%%%%%%%%%%%%%%%%%%%%%%%%%%%%%%%%
\bibliographystyle{amsalpha}

\end{document}